\documentclass[a4paper,11pt,reqno,natbib]{amsart}

\usepackage[utf8]{inputenc}
\usepackage[english]{babel}
\usepackage{lmodern}

\usepackage[left=3cm,right=3cm,top=3cm,bottom=3cm]{geometry}
\usepackage{amssymb}
\usepackage{amsthm}
\usepackage{paralist}
\usepackage{epsfig} 
\usepackage{epstopdf}  
\usepackage{amsfonts}%\usepackage{dsfont}
\usepackage{bm}
\usepackage{mathrsfs}  
\usepackage{mathtools} 
\allowdisplaybreaks

\usepackage{stmaryrd} % ll/rrbracket
\usepackage[all]{xy}
\usepackage{blkarray}

\usepackage{graphicx}
\usepackage{tikz} 
\usepackage{color}
\usepackage[hidelinks]{hyperref}
\usepackage{enumitem} 
\usepackage{calligra}
\usepackage[algoruled,vlined,boxed,longend,english,lined,boxed,commentsnumbered]{algorithm2e}

\usepackage{cancel}

\newtheorem{thm}{Theorem}
\newtheorem{lem}{Lemma}
\newtheorem{prop}{Proposition}

\newtheorem{ex}{Example}

\newtheorem{rem}{Remark}

\usepackage[normalem]{ulem}

\usepackage[color]{changebar}
\cbcolor{magenta}

\newcommand{\RR}{\mathbb R}
\newcommand{\VV}{\mathbb V}

\newcommand{\M}{\mathcal M}

\newcommand{\Opn}{\mathcal O}%\O exists 
\newcommand{\R}{\mathcal R}
\newcommand{\nU}{\mathcal U}

\newcommand{\X}{\mathcal X}
\newcommand{\Y}{\mathcal Y}
\newcommand{\Z}{\mathcal Z}

\newcommand{\Bf}{\mathbf f}

\newcommand{\cof}{\mathop {\rm cof}\nolimits}

\newcommand{\tr}{\mathop {\rm tr}\nolimits}

\def \div {\mathop {\rm div}\nolimits} % exists 

\newcommand{\gd}{\nabla}
\newcommand{\gs}{\nabla^{s}}

\newcommand{\bd}[1]{\boldsymbol{#1}}
\newcommand{\mrm}[1]{\mathrm{#1}}

\definecolor{veert}{RGB}{11, 164 , 20}
\definecolor{col1}{rgb}{0.04,0.5,0.15}  
\definecolor{col2}{rgb}{0.5,0.5,0.5} 
  
\newcommand{\lV}{\lVert}
\newcommand{\rV}{\rVert}
\newcommand{\lv}{\lvert}
\newcommand{\rv}{\rvert}
\newcommand{\lb}{\lbrace}
\newcommand{\rb}{\rbrace}

\newcommand{\class}{C}
\newcommand{\sep}{\mid}

\newcommand{\Ksub}{\subset\subset}

\newcommand{\id}{\mathrm{id}}%\newcommand{\id}{\mathop{\mathrm{id}}\nolimits}
\newcommand{\had}{D}%hold-all-domain

\newcommand{\dn}{n}%dimension
\newcommand{\Id}{\mathrm{I}}%{\mbox{I}}
\newcommand{\tp}{\top}%\newcommand{\tp}{\mathsf{T}}

\newcommand{\Mscalp}{:}
\newcommand{\opEx}{\mathop {P}\nolimits} %extension operator
\newcommand{\dx}{\, dx}

\newcommand{\nv}{n}
\newcommand{\nvI}{n}

\newcommand{\SHPfun}{\mathcal{J}}
\newcommand{\shpFUN}{j}
\newcommand{\adm}{\mathrm{ad}}
\newcommand{\stressF}{\varsigma}

\newcommand{\contract}{\boldsymbol{S}}

\newcommand{\dsp}{\mathrm{w}} 
\newcommand{\dspDual}{\mathfrak{w}}

\newcommand{\dspTestb}{w}

\newcommand{\LagVol}{\mathrm{s}} 
\newcommand{\LagVolDual}{\mathfrak{s}}

\newcommand{\vct}{\mathrm{u}} %velocity 
\newcommand{\vcT}{\mathrm{v}} %velocity transport on the reference config  
\newcommand{\vcTDual}{\mathfrak{v}}

\newcommand{\pr}{\mathrm{p}} %pression
\newcommand{\prT}{\mathrm{q}} %pression transport on the reference config  
\newcommand{\prTDual}{\mathfrak{q}}

\newcommand{\eT}{T} %transformation due to interaction F-S (mange field sur domaine perturbé)
\newcommand{\lT}{T} %transformation due to interaction F-S (mange field sur domaine initial)

\newcommand{\Tupt}{\lT^{t}} 

\newcommand{\rmH}{\boldsymbol{\mathrm{H}}}
\newcommand{\rmK}{\boldsymbol{\mathrm{K}}}
\newcommand{\rmb}{\boldsymbol{\mathrm{b}}} %pour champs de déplacement fixé

\newcommand{\rmA}{\boldsymbol{\mathrm{A}}}
\newcommand{\rmB}{\boldsymbol{\mathrm{B}}}

\newcommand{\BF}{A}
\newcommand{\LF}{L}
\newcommand{\NL}{\mathfrak{F}}

\newcommand{\dLf}{\, \mathrm{d}X}

\newcommand{\dLft}{\, \mathrm{d}X_t}

\newcommand{\dLs}{\, \mathrm{d}Y}

\newcommand{\dLst}{\, \mathrm{d}Y_t}

\newcommand{\rrF}{\boldsymbol{F}}

\usepackage[backend=bibtex, 
                firstinits=true,
                doi=false, 
               isbn=false, 
                url=false,
             eprint=false, %sorting=nty,
        maxbibnames=9]{biblatex}
\title[Shape sensitivity of a 2D Fluid-Structure Interaction problem]
{Shape sensitivity of a 2D Fluid-Structure Interaction problem between a viscous incompressible fluid and an incompressible elastic structure.}

\author{V. Calisti}
\address[V. Calisti]{
Institute of Mathematics of the Czech Academy of Sciences, 
\v{Z}itn\'{a} 25, 115 67 Praha 1, Czech Republic. 
Formerly: 
Institut \'{E}lie Cartan de Lorraine, UMR 7502, Universit\'{e} de Lorraine, B.P. 70239, 54506 Vandoeuvre-l\`{e}s-Nancy Cedex, France}
\email{calisti@math.cas.cz}

\author{I. Lucardesi}
\address[I. Lucardesi]{
Dipartimento di Matematica e Informatica, Università degli Studi di Firenze, Firenze, Italy.
Formerly: Institut \'{E}lie Cartan de Lorraine, UMR 7502, Universit\'{e} de Lorraine, B.P. 70239, 54506 Vandoeuvre-l\`{e}s-Nancy Cedex, France
}
\email{ilaria.lucardesi@unifi.it}

\author{J-F. Scheid}
\address[J-F. Scheid]{
Institut \'{E}lie Cartan de Lorraine, UMR 7502, Universit\'{e} de Lorraine, B.P. 70239, 54506 Vandoeuvre-l\`{e}s-Nancy Cedex, France}
\email{jean-francois.scheid@univ-lorraine.fr}

\date{\today}

\begin{document}

\maketitle

%Remarques : \\
%\blue{Ilaria}, \grn{Jean-Fran\c{c}ois}, \ppl{Valentin}, \red{\`{a} \'{e}crire}, \\

\begin{abstract}
We study the shape differentiability of a general functional depending on the solution of a bidimensional stationary Stokes-Elasticity system, with respect to the reference domain of the elastic structure immersed in a viscous fluid. The differentiability with respect to reference elastic domain variations are considered under shape perturbations with diffeomorphisms. The shape-derivative is then calculated with the use of the velocity method. This derivative involves the material derivatives of the solution of this Fluid-Structure Interaction problem. The adjoint method is used to obtain a simplified expression for the shape derivative.
\end{abstract}


%\tableofcontents

%#####################################################################
%## MAIN TEXT ########################################################

%#####################################################################
%#####################################################################
\section{Introduction}

Fluid Structure Interaction (FSI) problems model physical systems in which a solid body (rigid or deformable) interacts with a fluid (internal or external to the body).
In this work, we consider an elastic body in plain strain, clamped to a rigid support in its interior and immersed in a viscous incompressible fluid. The system is infinite in the anti-plane dimension. 
From the mathematical point of view, we consider a system of bi-dimensional stationary PDEs which involves on one hand the Stokes equations for the fluid flow and, on the other hand, an incompressible linearized elastic equations for the deformation of the structure. These two sub-systems are coupled through a boundary condition on the interface between the solid and the fluid, by imposing the forces continuity across the interface.

In this paper, we are interested in a shape optimization issue for this fluid-structure interaction problem. We aim to study the shape sensibility with respect to the reference domain $\Omega_0$ of the elastic body, also called the reference configuration (i.e. the domain at rest, before deformation) of a given shape functional. This functional depends on the elastic reference domain $\Omega_0$ as well as on the corresponding solution of the full PDEs system. We point out that in this context, we do not directly control the shape of the deformed elastic body which actually interacts with the fluid. 

%\medskip
The goal of this paper is to show the differentiability of a broad family of shape functionals (e.g. energy functional, drag functional,...) in which the shape is the reference configuration $\Omega_0$ of the elastic body and also to calculate the associated shape derivatives.
The differentiability is tackled with respect to the reference configuration $\Omega_0$ by considering a class of perturbations of $\Omega_0$ by diffeomorphisms.
We also provide formulas for the associated shape derivatives.
%A very powerful tool in this direction is the notion of shape derivative: on one hand, it allows to derive necessary optimality conditions; on the other hand, it provides the key ingredient for steepest descent methods in numerical shape optimisation.
%
These derivatives would be useful in a numerical shape optimisation procedure (as e.g. steepest descent methods) to determine an optimal elastic reference domain that minimizes a given shape functional
(see e.g. 
\cite{allaire2007conception,
mohammadi_pironneau2001shapo_fluid,
gao_ma_zhuang2009min_drag_NS,
yoon2010topology}).

%\medskip
Dealing with an FSI problem, the first mathematical issue is proving \textit{existence} of solutions. 
Early important contributions can be found in  
\cite{antman_lanza-de-cristoforis1991fsi, 
antman_lanza-de-cristoforis1991large_def, 
antman_lanza-de-cristoforis1992large_def_axisym} 
in which the authors study stationary flows in nonlinear elastic shells and also nonlinear elastic tubes and shells under external flow for which the velocity is prescribed.
In the early 2000, mathematicians  
started to investigate 
more intensely the interaction of a  %Navier–Stokes (viscous) 
viscous liquid with elastic bodies in steady and unsteady regime.
For steady-state problems, one can cite
\cite{rumpf1998equilibriaFSI, 
grandmont2002existence, 
bayada_chambat_vazquez2004existence,  
surulescu2007stationary, 
galdi_kyed2009steady_flow}
and for the unsteady case, we refer for example to
\cite{grandmont_maday2003FSI, 
desjardins_esteban_grandmont2001weak_sol, 
chambolle_desjardins_esteban_grandmont2005existence, 
boulakia_schwindt_takahashi2012existence, boulakia_guerrero_takahashi2019wellposedness, 
muha_canic2013existence}. %beirao2004strong_sol_evolFSI, grandmont_maday2000existence_unsteady,
One of the difficulties in the study of these kind of FSI problems is that the fluid, described in Eulerian coordinates, turns to be defined on a domain depending on the structure displacement, which is instead described in Lagragian coordinates.
For the FSI problem under consideration in this paper, we will first establish the existence and uniqueness of the solution.

%\medskip
The second issue in FSI problems is to find {\it optimal structures}  which optimize a suitable desired efficiency in fluid dynamics, possibly under constraint. Great interest has been shown in the minimization of the drag in fluid mechanics optimisation (see e.g., \cite{arumugam_pironneau1989drag_reduc,
mohammadi_pironneau2001shapo_fluid, 
gao_ma_zhuang2009min_drag_NS}),  %gao_ma2008min_drag_S
in the shape minimization of the dissipated energy in a pipe (see e.g.,  \cite{henrot_privat2010pipe, bergounioux_privat2013shapo_axisym}) 
or in the optimisation of fluid flow with or without body forces (see e.g. \cite{deng_liu_wu2013optopoNS}). 
In all this mentioned works, the shape or the geometry in which PDEs lie, are fixed and known. 
Shape optimisation applied to FSI problems, where the geometry is one of the unknowns, is more recent.
One can cite 
\cite{yoon2010topology, yoon2014optopoFSI, 
kreissl_pingen_evgrafov_maute2010topology, andreasen_schousboe_sigmung2013optopo_fsi, jenkins_maute2015shapoptopoFSI} %jenkins_maute2016shapoptopoFSI
where level-set methods are used to characterize the fluid and the structure domains,  
and also 
\cite{picelli_vicente_pavanello2015bidir,  picelli_ranjbarzadeh_sivapuram_giora_silva_2020,  lundgaard_alexandersen_zhou2018optopoFSI} in which the FSI problem is relaxed by a density design variable.
The work presented in this paper is an extension of what is done in 
\cite{scheid_sokolowski2018shapo_fs} where the shape differentiability of a simplified free-boundary one-dimensional problem is studied and for which it is proved that the shape optimisation problem is well-posed.
In the recent papers \cite{feppon_allaire_bordeu_cortial2019optopo_TFSI,  feppon_allaire_dapogny_jolivet2021body-fitted_optopo}, the shape and topological optimisation of a multiphysics thermal-fluid-structure interaction problem is studied with a velocity and adjoint method, for which the structure domain is assumed to be fixed. 
In \cite{Wick2019}, differentiability results are shown for the solutions of a stationary fluid-structure interaction 
problem in an ALE framework. The differentiability is considered with respect to variations of the given data (volume forces and boundary
values) but not with respect to the reference domain of the elastic structure, as it is done in this present paper. 
Finally, we mention the work of Haubner et al \cite{Haubner2020} where the method of mappings is used for proving differentiability results with respect to domain variations, for {\it unsteady} fluid-structure interaction problems that couple the Navier–Stokes equations and the Lamé system.
%+ Ref articles [Halanay, Murea, Tiba 2016] with the use of a fictitious domain method with penalization for the Stokes equation.
%\\

The paper is organized as follows: we start, in Section \ref{sec:FSI_model}, with a presentation of the FSI problem under study. In Section \ref{sec:FSI:resolution} we prove existence and uniqueness result for the FSI problem, first by analyzing separately the fluid equations and  the structure problem, and finally by coupling the two subsystems through a fixed point procedure.
Then, in Section \ref{sec:hadam:method}, after an introduction to the calculus of shape derivatives by the \emph{velocity method}, we apply this approach to our FSI problem: the sensitivity analysis allows us to show that the solutions of the FSI system are shape-differentiable.
The last part in Section \ref{sec:shp-der}, is devoted to the calculation of the shape derivative of an abstract shape functional. Using the \emph{adjoint method}, we also give a simplified expression of the shape derivative, not depending on the \textit{material derivatives} of the solutions of the FSI problem
but involving the solutions of adjoint problems.

\section{A two-dimensional FSI model with a shape optimisation problem}
\label{sec:FSI_model}

%{\color{blue}[]J'ai mis en commentaire l'ancien paragraphe d'intro].}

%We are interested in the optimisation of a Fluid-Structure Interaction (FSI) problem. 
%We want to investigate some optimality result regarding the shape of a given two-dimensional elastic body (the structure) in plane strain immersed in an incompressible Stokes fluid, and clamped from a part of its boundary, while applying volume forces to both fluid and elastic phases (see Figure \ref{fig:contour}).  
%We consider that this system is infinite in the anti-plane direction. 
%This results in the deformation of the free boundary of the elastic body, which is the interface where the interaction between the elastic body and the fluid takes place.
%We start by presenting an FSI model following what is done in \cite{grandmont2002existence} and \cite{scheid_sokolowski2018shapo_fs}, then we introduce a simplified model, and finally we present a general shape optimisation protocol.

In this section, we first present the FSI model under study and then the related shape optimisation problem that will be addressed in this paper.
The FSI model couples the Stokes equations with the elasticity equation and follows essentially \cite{grandmont2002existence} and \cite{scheid_sokolowski2018shapo_fs}. The difference with respect to the literature is the assumption of linear incompressible elasticity for the structure, which results in a divergence free condition for the structure's displacement.

\subsection{Notations}
In this preliminary paragraph we fix the notations that will be used throughout the paper.
Let $\lb \mrm{e}_1 , \cdots , \mrm{e}_{\dn} \rb$ be the canonical orthogonal basis of $\RR^{\dn}$, $\dn\geq 2$.
Let $u$ and $v$ be two vectors of $\RR^{\dn}$, $A$ and $B$ be two second order tensors of $\RR^{\dn}$. %, and $\mathbf{T}$ be a fourth order tensor of $\RR^{\dn}$, 
Using the Einstein summation convention, we set
\begin{align}
%\mathbf{T}A &= \mathbf{T}_{ijkl}A_{kl} \, \mathrm{e}_i \otimes \mathrm{e}_j , \\
%
AB & = A_{ik}B_{kj} \, \mathrm{e}_i \otimes \mathrm{e}_j , \\
A \Mscalp  B & = A_{ij}B_{ij} , \\
Au &= A_{ij}u_j \, \mathrm{e}_i , \\
u\cdot v &= u_iv_i  ,
\end{align}
where  $\{\mathrm{e}_i \otimes \mathrm{e}_j\}_{i,j}$ forms the canonical basis of the second order tensors on $\RR^{\dn}$.  
Denoting by $\Id$ the identity matrix, we define the trace $\tr ( A ) $ of a matrix $A$  by
\begin{equation}
\tr ( A ) = \Id \Mscalp  A , 
\end{equation} 
its symmetric part by 
\begin{equation}
\label{eq:def:sym:matrix}
A^s := \frac{1}{2} \left( A + A^{\tp} \right), 
\end{equation}
and its norm $\lv A \rv$  by
\begin{equation}
\label{eq:def:matrixnorm}
\lv A \rv = ( A \Mscalp  A )^{1/2}.
\end{equation}
Moreover, if $A$ is an invertible matrix, we define the cofactor matrix of $A$ by 
\begin{equation}
\label{eq:def:matrixcofacteur}
\cof( A ) = \det( A  ) A^{-\tp}.
\end{equation}

Let $\Omega$ be an open subset of $\RR^{\dn}$, $\dn\geq 2$. 
The functions or fields involved in the equations we study in this paper belong to Sobolev spaces $W^{m,p}(\Omega)$, for $m\geq 0$ a positive integer, and $1 \leq p \leq +\infty$. 
With this convention, $W^{0,p}(\Omega)$ stands for the Lebesgue space of $L^p (\Omega)$. The norm in $W^{m,p}(\Omega)$ is denoted by $\lV \cdot \rV_{m,p, \Omega }$, or, when no ambiguity may arise, simply by $
\lV \cdot \rV_{m,p}$. 
Finally, the space $W^{m,2}(\Omega)$ will simply be denoted by $H^{m}(\Omega)$.

\subsection{The Fluid-Structure Interaction model}
We consider a two-dimensional elastic body (the structure) immersed in an incompressible viscous fluid and clamped from a part of its boundary, while applying volume forces to both fluid and elastic parts.
%The elastic body is in plane strain so that the  system is considered to be infinite in the anti-plane direction.  
This results in the deformation of the free boundary of the elastic body, which is the interface where the interaction between the elastic body and the fluid takes place (see Figure \ref{fig:contour}).

In order to describe this setting, we fix three simply connected bounded open sets $\omega, \had_0, \had \subset \mathbb R^2$, such that $\omega \Ksub \had_0 \Ksub \had$. 
We denote by $\Gamma_0$ and $\Gamma_\omega$ the boundaries of $\had_0$ and $\omega$, respectively.
The annular domain
\begin{equation}
\label{eq:FSI_def_Omega0}
\Omega_0:= \had_0 \setminus \overline{\omega}
\end{equation}
represents the region occupied by the elastic body, that we assume to be clamped on the boundary part $\Gamma_\omega$. The complementary set in the box $D$, namely the annular domain
\begin{equation}
\label{eq:FSI_def_Omega0_c}
\Omega_0^c := \had \setminus \overline{D_0},
\end{equation}
is the region occupied by the fluid, that we take incompressible. The elastic body and the fluid interact through the interface $\Gamma_0$, which is deformable.

%The domain $\Omega_0$ is the configuration at rest of this elastic medium.

\begin{figure}[ht]
\centering
\begin{tikzpicture}[scale=1.8]
\tikzstyle{flecheFS}=[->,>=latex,line width=0.2mm]
%%%fig1 Omega0
\draw (-1.5,-1.5) -- (-1.5,1.5) -- (1.5,1.5) -- (1.5,-1.5) -- cycle;
\draw [domain=0:2*pi, samples=80, smooth] plot ({-(1 + 0.5*exp(-cos(\x r)))*cos(\x r)/2}, {sin(\x r)/2}) ;
\fill[color=gray!20] [domain=0:2*pi, samples=80, smooth] plot ({-(1 + 0.5*exp(-cos(\x r)))*cos(\x r)/2}, {sin(\x r)/2}) ;
\fill[color=gray!60] (0,0) circle (0.2) ; 
\draw (0,0) circle (0.2) ;
\draw (0,0) node{$\omega$} ;
\draw (-1.5,1.3) node[left]{$\had$} ;
\draw (1.1,1.2) node{$\Omega_{0}^{c}$} ;
\draw (0.5,0) node{$\Omega_{0}$} ;
\draw (0,0.8) node{$\Gamma_{0}$} ;
\draw[flecheFS] (0.2,0.7) to[bend left=20] (0.6,0.39);
\draw (-0.65,-0.65) node{$\Gamma_{\omega}$} ;
\draw[flecheFS] (-0.58,-0.58) to[bend left=20] (-0.16,-0.13);
\draw (1.075,-0.7) node{$n_{0}$} ;
\draw[flecheFS] (0.745,-0.345) to  (0.9,-0.7);
%
%
%%fig1 T ( Omega0 )
\begin{scope}[xshift = 4cm]
\draw [domain=0:2*pi, samples=80, dashed] plot ({-(1 + 0.5*exp(-cos(\x r)))*cos(\x r)/2}, {sin(\x r)/2}) ;
\draw (-1.5,-1.5) -- (-1.5,1.5) -- (1.5,1.5) -- (1.5,-1.5) -- cycle;
\draw [domain=0:2*pi, samples=80, smooth] plot ({-(1 + 0.5*exp(-cos(\x r)))*cos(\x r)/2}, {(1 - 0.3*( exp(-10*(\x - 3*pi/4)^2 ) +exp(-10*(\x - 5*pi/4)^2 ) ) )*sin(\x r)/2}) ;
\fill[color=gray!20] [domain=0:2*pi, samples=80, smooth] plot ({-(1 + 0.5*exp(-cos(\x r)))*cos(\x r)/2}, {(1 - 0.3*( exp(-10*(\x - 3*pi/4)^2 ) +exp(-10*(\x - 5*pi/4)^2 ) ) )*sin(\x r)/2}) ;
\fill[color=gray!60] (0,0) circle (0.2) ; 
\draw (0,0) circle (0.2) ;
\draw (0,0) node{$\omega$} ;
\draw (-1.5,1.3) node[left]{$\had$} ;
\draw (1.1,1.2) node{$\Omega_{F}$} ;
\draw (0.5,0) node{$\Omega_{S}$} ;
\draw (0,0.8) node{$\Gamma_{FS}$} ;
\draw[flecheFS] (0.25,0.7) to[bend left=20] (0.6,0.29);
\draw (1.1,-0.6) node{$n_{FS}$} ;
\draw[flecheFS] (0.745,-0.24) to  (0.8,-0.62);
\end{scope}
\end{tikzpicture}
\caption{The geometry of the FSI system,  before (left) and after (right) deformation induced by the interaction between the fluid and the structure.}
\label{fig:contour}
\end{figure}
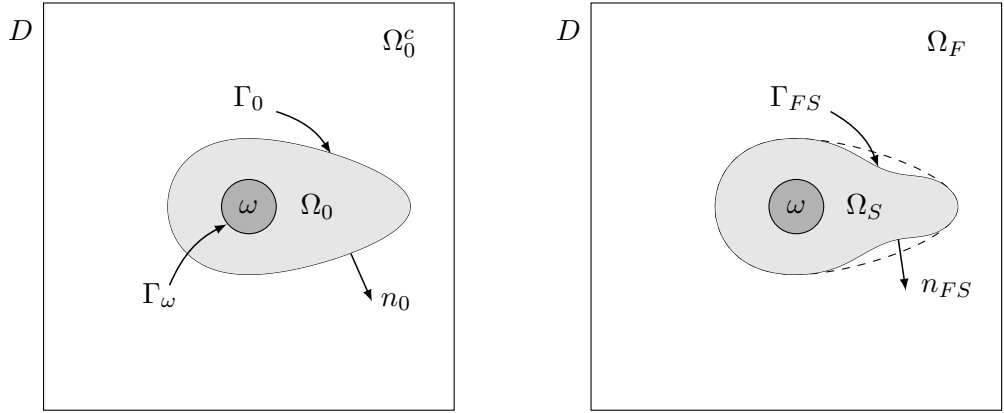

The fluid and the structure are subject to volume forces which result in a deformation of the elastic part. In our analysis, we assume that the system is at equilibrium, in particular, the time variable will not appear in the model.

%We assume that the interaction with the fluid, together with the applied volume forces, deforms the elastic body and leads to an equilibrium state in the fluid and the structure parts.

The deformed elastic body, denoted by $\Omega_S$, is described in Lagrangian coordinates, that is, through a function defined in the reference configuration:
$$
\Omega_S:=\eT(\dsp)(\Omega_0), 
$$
with 
\begin{equation}
\label{eq:def:T:i}
\eT(\dsp):\Omega_0 \to D\setminus \overline{\omega}, \quad 
\eT(\dsp)= \id_{\RR^2} + \dsp,
\end{equation}
where $\id_{\RR^2}$ is the identity in $\mathbb R^2$ and $\dsp$ is the \textit{elastic displacement field} in $\Omega_0$.
Accordingly, the deformed fluid-structure interface is
\begin{equation}
\label{eq:def:GammaS}
\Gamma_{FS} := \eT (\dsp) (\Gamma_0) =  ( \id_{\RR^2} + \dsp )( \Gamma_0 ).
\end{equation}

%Each point of the initial elastic body $X\in\Omega_0$ is deformed  into a new point  
%$ x = \eT (X) $, where
%\begin{equation}
%\label{eq:def:T:i}
%\eT (X) =  ( \id_{\RR^2} + \dsp )(X) \in \Omega_S ,
%\end{equation}
%where $\dsp : \Omega_S \rightarrow \RR^{2}$ is called the %\emph{displacement field}, and 
%\begin{equation}
%\label{eq:def:OmegaS}
%\Omega_S :=  \eT (\Omega_0) = ( \id_{\RR^2} + \dsp )( %\Omega_0 ) 
%\end{equation} 
%is the domain representing the deformed elastic body.

On the other hand, the fluid is described in Eulerian coordinates, namely through functions defined in the region surrounding the deformed elastic body
\begin{equation}
\label{eq:def:fluid_domain}
\Omega_F := \had \setminus \overline{\Omega_S \cup \omega}.
\end{equation} 
The functions describing the fluid are the \emph{velocity field} $\vct : \Omega_F \rightarrow \RR^2$ and the \emph{pressure field} $\pr  : \Omega_F \rightarrow  \RR$.

In the following paragraphs, we will specify the PDEs governing the two phases of the system, and their interaction.  

\subsubsection{Fluid equations}
In the framework of incompressibility, the velocity field $u$ and the pressure field $p$ are governed by Stokes equations:
\begin{equation} 
\label{eq:S-ou-NS0}
\left\{
\begin{aligned}
- \div \varsigma( \vct , \pr )    & =  f && \text{ in }   \Omega_{F}, \\
\div\ \vct & =  0  &&  \text{ in }  \Omega_{F},  \\
\vct & =  0  && \text{ on }  \partial\Omega_{F} .
\end{aligned}
\right.  
\end{equation}
In the system, $f: \RR^2 \to \mathbb R^2$ is the applied force, defined in the whole space, whereas $\stressF$ is the \emph{Cauchy stress tensor}, defined by
\begin{equation}
\stressF( \vct , \pr ):= 2\nu \gs \vct 
- \pr \Id,
\end{equation} 
with $\nu > 0$ the viscosity of the fluid. We recall that the superscript $s$ stands for the symmetrization operator (see \eqref{eq:def:sym:matrix}).

%\begin{equation}
%\label{def:gradient-sym}
%\nabla^s \vct := \frac{1}{2} \left( \nabla \vct + \nabla %\vct ^{\tp} \right).
%\end{equation}
%The force $f$ in \eqref{eq:S-ou-NS} is given in $\RR^2$.  

\subsubsection{Structure equations}
We suppose that the elastic body is attached to the rigid support $\omega$ via its boundary $\Gamma_\omega$. This assumption results in a Dirichlet boundary condition for the elastic displacement $\dsp$:
\begin{equation}
\dsp = 0 \quad \text{ on } \Gamma_\omega.
\end{equation} 
A given volume force $g$ is applied to the structure in $\Omega_0$ and the elastic displacement $\dsp$ satisfies the elasticity equation
\begin{equation}
\label{eq:FSI:general}
-\div \sigma(\dsp) = g \ \text{ in } \Omega_0,
\end{equation}
where $\sigma$ is the  \emph{linearized stress tensor} (also called the second Piola-Kirchoff stress tensor) or simply \emph{stress tensor}
\begin{equation}
\label{eq:def:sigma}
\sigma ( \dsp ) := 
2\mu \gs \dsp+ \lambda  ( \div \dsp )   \Id. \end{equation}
Here $\lambda$ and $\mu$ are the so-called Lam\'{e} coefficients (see e.g. \cite{ciarlet_three-dimensional_1988}).
Furthermore, we impose the equilibrium of the surface forces on the free boundary $\Gamma_0$ which reads as
\begin{equation}
\label{eq:FSI:1}
\int_{\Gamma_0} \sigma ( \dsp ) \nvI_0 \cdot ( v \circ ( \id_{\RR^2} + \dsp ) ) \mathrm{d}\Gamma_0
= \int_{\Gamma_{FS}}  \varsigma ( \vct , \pr ) \nvI_{FS} \cdot v \  \mathrm{d}\Gamma_{FS}  ,
\end{equation}
for all function $v$ defined on $\Omega_F$. In the above relation, $\Gamma_{FS}$ is defined in \eqref{eq:def:GammaS} and denotes the boundary between the fluid domain $\Omega_F$ and the deformed elastic body $\Omega_S$, whereas $\mathrm{d}\Gamma_0$ and $\mathrm{d}\Gamma_{FS}$ are the length elements of the boundaries $\Gamma_0$ and $\Gamma_{FS}$ respectively, 
and finally $\nvI_0$ and $\nvI_{FS}$ are the outer unit normal vectors to $\Gamma_0$ and $\Gamma_{FS}$ respectively. 
Recalling that $\Gamma_{FS}$ is the image of $\Gamma_0$ via $T (\dsp) =\id_{\RR^2} + \dsp$, cf. \eqref{eq:def:T:i}-\eqref{eq:def:GammaS}, we infer (see e.g. \cite{ciarlet_three-dimensional_1988}) that
\begin{equation}
\label{eq:change:var:surf}
\nvI_{FS} \mathrm{d}\Gamma_{FS}  =  [ \det ( \gd (\eT(\dsp)) )   \gd (\eT(\dsp)) ^{-\mathsf{T}}  \nvI_{0}  ]  \mathrm{d}\Gamma_0 .
\end{equation}  
Thus, using $\eT (\dsp)$ for a change of variables in \eqref{eq:FSI:1} together with \eqref{eq:change:var:surf}, we get 
the following boundary condition 

%Thus the change of variable with $T$ in \eqref{eq:FSI:1} and the use of \eqref{eq:change:var:surf} gives the following boundary condition 
\begin{equation}
\label{eq:BC_FS_traction_Gamma0}
\sigma ( \dsp ) \nv_0
= (  \varsigma ( \vct , \pr ) \circ \eT ) \cof ( \gd \eT ) \nv_0 
\quad \text{ on } \Gamma_0,
\end{equation}
%where  $\eT$ is defined in \eqref{eq:def:T:i} and
where
\begin{equation}
\cof( \gd \eT ) = \det( \gd \eT  ) (\gd \eT) ^{-T}
\end{equation}
is the cofactor matrix of the jacobian matrix of $\eT := T( \dsp )$.

\smallskip
In this paper, we consider the special case of linear incompressible elasticity for the structure, by imposing the following equation for the displacement:
\begin{equation}
\label{eq:def:elast-incomp-constr}
\div \dsp = 0 .
\end{equation}
We introduce a Lagrange multiplier function $ \LagVol $ associated to the incompressibility constraint \eqref{eq:def:elast-incomp-constr}. Then, the structure equation \eqref{eq:FSI:general} together with the continuity condition of forces \eqref{eq:BC_FS_traction_Gamma0} become for $( \dsp , \LagVol )$:
\begin{equation}
\label{eq:bvp:simplified_elast}  
\left\lbrace
\begin{aligned} 
-  \div \sigma (\dsp ) +\nabla \LagVol & =  g  && \text { in }   \Omega_{0} , \\
%
%\div \dsp  & = 0 &&\text { in } \Omega_{0}  &&  \\
%
%\dsp  & = 0 &&\text { on } \Gamma_{\omega}  &&   \\
%
\left(\sigma (\dsp )  - \LagVol  \Id \right) \nvI _{0} & =  ( \varsigma ( \vct, \pr ) \circ \eT) \cof ( \nabla \eT ) \nvI _{0} && \text { on }  \Gamma_{0} .   
\end{aligned}
\right.
\end{equation}

\subsubsection{Full FSI coupled system}
Using the fact that both the velocity $\vct$ and the displacement $\dsp$ are divergence free,
the FSI system for $(\vct,\pr)$ and $(\dsp,\LagVol)$ that we consider in this paper is the following:
%and we have that the structure equations, whose solution is $( \dsp , %\LagVol )$, are given by
\begin{equation} 
\label{eq:S-ou-NS}
\left\{
\begin{aligned}
- \nu \div ( \gd  \vct ) +  \gd \pr     & =  f && \text{ in }  \Omega_{F}, \\
\div\ \vct & =  0 &&  \text{ in }  \Omega_{F},  \\
\vct & =  0 && \text{ on }  \partial\Omega_{F}, 
\end{aligned}
\right.  
\end{equation}
and
\begin{equation}
\label{eq:bvp:incompress_elast}
\left\lbrace
\begin{aligned}
- \mu \div ( \gd \dsp ) + \nabla \LagVol & = g  &&\text { in } \Omega_{0} &&  \\
\div \dsp  & = 0 &&\text { in } \Omega_{0}  &&  \\
\dsp  & = 0 &&\text { on } \Gamma_{\omega}  &&   \\
\left( \mu  \gd  \dsp  - \LagVol  \Id \right) \nvI _{0} &=( ( \nu \gd \vct - \pr \Id ) \circ \eT) \cof ( \nabla \eT ) \nvI _{0} &&\text { on } \Gamma_{0} .   &&  
\end{aligned}
\right.
\end{equation}

\begin{rem}
%\red{Remarque sur la contrainte sur le volume qu'on n'a pas eu besoin de rajouter puisqu'on a l'incompressibilite ?}
We point out that if we consider the general linear elastic equation \eqref{eq:FSI:general} in place of the special incompressible case given by \eqref{eq:def:elast-incomp-constr} and \eqref{eq:bvp:simplified_elast}, we have to add an area preserved constraint for the displacement $\dsp$, due to the incompressibility of the fluid. This extra constraint reads as:
\begin{equation}
\label{eq:contrainte_fs_1}
\lv \Omega_S \rv = \int_{\Omega_0} \det ( \Id + \gd \dsp ) \dx = \lv \Omega_0 \rv,
\end{equation}
where $\lv \cdot \rv$ denotes the Lebesgue measure for domains.
We have that 
\begin{equation}\label{eq:dev_det}
\det \left( \Id + \nabla \dsp \right) = 1 + \div ( \dsp ) +\det ( \nabla \dsp ) = 1 + \div ( \dsp ) + O \left( \lV \nabla \dsp \rV_{\infty}^{2} \right) .
\end{equation}
So, under the condition that $\div \dsp = 0$ and neglecting the second order terms in \eqref{eq:dev_det}, we obtain that the area constraint \eqref{eq:contrainte_fs_1} is satisfied. 
%Finally we must add a constraint on the displacement field, in order to make the deformation it creates compatible with the incompressibility of fluid. That is we have the following condition
%\begin{equation}
%\lv \Omega_S \rv = \lv \Omega_0 \rv,
%\end{equation}
%where $\lv \cdot \rv$ denotes the Lebesgue measure. This condition is actually a condition on $\dsp$:
%\begin{equation}
%\label{eq:contrainte_fs_1}
%\lv \Omega_S \rv = \int_{\Omega_0} \det ( \Id + \gd \dsp ) \dx = \lv %\Omega_0 \rv ,
%\end{equation}
%}

%We also want to perform a linearization for this constraint.
%%Then we want a model for which the global constraint \eqref{eq:contrainte_fs_1} is no longer needed.
%For a matrix $A$ of size $2\times 2$, we recall that 
%\begin{equation}
%\det \left( \Id + A \right) = 1 + \tr (A) +  \det (A).
%\end{equation}
%Hence, we get 
%\begin{equation}
%\det \left( \Id + \nabla \dsp \right) = 1 + \div ( \dsp ) +\det ( \nabla \dsp ) = 1 + \div ( \dsp ) + O \left( \lV \nabla \dsp \rV_{\infty}^{2} \right) .
%\end{equation}
%If we assume that
%\begin{equation}
%\label{eq:contrainte_incompress_elastq}
%\div \dsp = 0   ,
%\end{equation}
%and if we neglect the second order terms, we obtain that the area constraint \eqref{eq:contrainte_fs_1} is verified. 
\end{rem}

We can observe that the coupling of the FSI problem \eqref{eq:S-ou-NS}-\eqref{eq:bvp:incompress_elast} is twofold:
\begin{itemize}
\item the structure displacement $\dsp$ affects and defines the domain $\Omega_F$ on which the fluid equations are posed and where the velocity $\vct$ and the pressure $\pr$ are defined,
\item the velocity $\vct$ and the pressure $\pr$ of the fluid give rise to a surface force which influences the calculation of the displacement $\dsp$.
\end{itemize}
One of the main difficulties lies in the fact that there are two kinds of variables under consideration. On the one hand, the FSI problem involves Eulerian variables with the fluid velocity
$\vct$ and pressure $\pr$, and on the other hand, the elastic displacement $\dsp$ and the multiplier $\LagVol$ are Lagrangian variables.

%One  difficulty arises.
%Indeed we have two kinds of variables: Eulerian variables (the fluid velocity
%$\vct$ and pressure $\pr$, and Lagrangian variables (the displacement $\dsp$ and the multiplier $\LagVol$). 
Moreover, the domain $\Omega_F$ on
which the fluid equations are written is unknown.
To overcome these difficulties, we need to transport the fluid equations into a reference domain matching with the elastic reference domain $\Omega_0$.

%*********************************************************************
%*********************************************************************
\subsection{Fixed domain formulation of the FSI problem}
\label{sec:FSI:fixed-domain-formul}

In order to tackle the FSI problem \eqref{eq:S-ou-NS}-\eqref{eq:bvp:incompress_elast}, we 
transpose the fluid equations \eqref{eq:S-ou-NS} posed on the fluid domain $\Omega_F$ onto the fixed domain $\Omega_0^c$ defined by
\begin{equation}\label{eq:complementary_omega0}
\Omega_0^c := \had \setminus \overline{\Omega_0 \cup \omega}.
\end{equation}
%In order to transport the boundary value problem \eqref{eq:S-ou-NS} from the fluid domain $\Omega_F$ defined by \eqref{eq:def:fluid_domain}, to the reference domain $\Omega_0^c$
Thus, we need a $C^1$-diffeomorphism which maps $\Omega_0^c$ to $\Omega_F$.
To this aim, we consider an extension of the map $T$, initially defined on $\Omega_0$ in \eqref{eq:def:T:i}, to the whole box $D$. With a slight abuse of notation, we use the same letter $T$ and we set
\begin{equation}
\label{eq:def:T:ii}
\eT ( \dsp )   =   \id_{\RR^2} + \opEx ( \dsp ), 
\end{equation} 
where $\dsp$ is a displacement field defined in the initial elastic body domain $\Omega_0$, and $\opEx$ is an extension operator from  $\Omega_0$ to  $\had$, such that $\opEx ( \dsp )$ is defined in $\had$ and $\eT ( \dsp )$ is one to one in $\had$.  
This allows us to consider the fluid domain $\Omega_F$ defined as:
\begin{equation}
\Omega_F = \eT ( \dsp ) ( \Omega_0^c ),
\end{equation}
where $\Omega_0^c$ is defined in \eqref{eq:complementary_omega0} (see also Figure \ref{fig:contour}). We will go through this extension procedure in details later on, to give a rigorous definition of $T$. \\

In the same way as in \cite{grandmont2002existence}, 
we can define the transported velocity and pressure fields 
\begin{align}
\label{eq:def:v}
\vcT := \vct \circ \eT(\dsp) , \\
\label{eq:def:q}
\prT := \pr  \circ \eT(\dsp). 
\end{align}
With these new variables, we can write 
the fluid equations transported onto the reference domain $\Omega^c_{0}$, (e.g., by using the variational formulation as in \cite{calisti2021}, Section 3.2.2), 
and the complete FSI problem reads as:
\begin{equation}
\label{eq:complete-syst-stokes-stokes-transported}
\left\{  
\begin{aligned}
- \nu\div ( (\gd \vcT ) F(\dsp)) + G( \dsp)\gd  \prT  & =  (f\circ \eT(\dsp) ) J( \dsp  ) && \text{ in }  \Omega_{0}^c, \\
\div   (G(\dsp  )^{\tp} \vcT )  & =  0 &&  \text{ in }   \Omega_{0}^c ,  \\
\vcT & =  0 && \text{ on }  \partial \Omega_{0}^c , \\
- \mu \div  ( \gd \dsp) + \gd \LagVol & =  g && \text{ in }  \Omega_{0}, \\ 
\div \dsp & =  0 && \text{ in }  \Omega_0 ,\\
\dsp & =  0 && \text{ on }  \Gamma_\omega ,\\
( \mu \gd \dsp - \LagVol \Id )n_{0} &  =    \nu ( \gd \vcT ) F(\dsp )   n_{0} &&   \\
     &  \hspace{7ex} - \prT  G( \dsp )   n_{0} && \text{ on }  \Gamma_{0},
\end{aligned}
\right.
\end{equation}
where we have set
\begin{align}
\label{def:F(T)2}
F( \dsp)&:=(\gd \eT(\dsp) )^{-1}\cof (\gd \eT(\dsp) ) , \\
\label{def:G(T)2}
G(\dsp  )&:= \cof (\gd \eT(\dsp) ) , \\
\label{def:J(T)2}
J(\dsp)&:= \det (\nabla \eT(\dsp) ) .
\end{align}
The boundary condition on $\Gamma_0$ appearing in \eqref{eq:complete-syst-stokes-stokes-transported} comes from the computation of the surface force applied on the structure, given in \eqref{eq:BC_FS_traction_Gamma0} by $(  \varsigma ( \vct , \pr ) \circ \eT ) \cof ( \gd \eT ) \nv_0$, in terms of the new variables $\vcT $ and $\prT $:
\begin{equation}
\label{eq:struct:surface-force2}
( \varsigma ( \vct  , \pr  ) \circ \eT  ) \cof ( \nabla \eT  ) \nvI _{0} 
= (\nu(\gd \vcT ) F(\dsp ) -\prT  G(\dsp)) \nvI_{0 } .
\end{equation}

We point out that the FSI problem \eqref{eq:complete-syst-stokes-stokes-transported} is a sort of hybrid model compared to \cite{grandmont2002existence}, coming from the linearization of the equilibrium equation of the structure (that is to say the Piola-Kirchhoff stress tensor) and the area constraint \eqref{eq:contrainte_fs_1}. This has been done in order to simplify the shape optimisation analysis performed in this paper.
%
%We will obtain in Section \ref{sec:FSI:resolution} an existence and uniqueness result in Theorem \ref{thm:FSI:pt-fixe} for our two-dimensional system  \eqref{eq:complete-syst-stokes-stokes-transported}, while for the three-dimensional Navier-Stokes/St Venant-Kirchhoff FSI problem,  Grandmont only obtains an existence result in \cite{grandmont2002existence}. 
%For our purpose, the uniqueness of the solution is required to address an associated optimisation problem.
%
Moreover, we do not have linearized the terms arising from the fluid equations change of variables, i.e., $J( \dsp )$, $G( \dsp )$, and $F( \dsp )$, because we want to compute shape derivatives %in Section \ref{sec:hadam:method} 
by keeping as much information as possible, for possible further applications and calculation purposes for a general system. 

%{\bf Nevertheless, even by linearizing $J( \eT(\dsp) )$, $G( \eT(\dsp) )$, and $F( \eT(\dsp) )$ in the following, the results we obtain are not trivial, because it keeps a trace of the deformation of the domain.} \grn{$->$ dernière phrase pas très claire, je serai d'avis de l'enlever ?}.

%*********************************************************************
%*********************************************************************
\subsection{Optimisation of the FSI system}
\label{sec:FSI:intro:optim}

The shape sensitivity analysis of the FSI model \eqref{eq:complete-syst-stokes-stokes-transported} carried out in this article, is motivated by a shape optimisation problem. This problem consists in 
%The aim of this paper is to study a shape %optimisation problem for the FSI model %\eqref{eq:complete-syst-stokes-stokes-transported}%. We are 
seeking for an optimal shape of the elastic reference domain $\Omega_0$ that minimizes a functional depending on the solution of the FSI system associated to $\Omega_0$. The shape optimisation problem we consider is of the following form:
\begin{equation}\label{eq:Shape_optim_pb}
\min_{\Omega_{0}\in \nU_{\adm}}  \SHPfun (\Omega_{0}) , 
\end{equation}
where $\SHPfun (\Omega_0)$ is a quite general shape functional depending on the initial elastic domain defined by
\begin{equation}
\label{eq:def:shp:fun}
\SHPfun (\Omega_0) 
= \int_{\Omega_0} \shpFUN_S (Y, \dsp(Y), \gd \dsp(Y)) \,\mathrm{d}Y
+ \int_{\Omega_F} \shpFUN_F (x, \vct(x), \gd \vct(x)) \,\mathrm{d}x ,
\end{equation}
where  $\shpFUN_F$ and $\shpFUN_S$ are smooth functions depending respectively on $\vct = \vcT \circ T(\dsp)^{-1}$ and $\dsp$. 
The fields $\vcT$ and $\dsp$ are the velocity and the displacement solutions of the FSI problem \eqref{eq:complete-syst-stokes-stokes-transported} posed on $\Omega_0 \cup \Omega_0^c$.
The domain $\Omega_0 \in \nU_{\adm}$ belongs to a class $\nU_{\adm}$ of smooth domains admissible for the FSI problem. 
For example, we can consider
\begin{align*}
\nU_{\adm} := 
\{ A \subset \mathbb R^2,\  & A=B\setminus\overline{\omega} \text{ with } B \text{ smooth,}  \\
&\text{ simply connected},  \omega \subset B\subset D \text{ and } |A|=|\Omega_0|\} .
\end{align*}
% \begin{itemize}
%     \item Avec ou sans la contrainte d'aire $|A|=|\Omega_0|$ ?
%     \item Sans contrainte d'aire, on risque d'avoir la forme optimale $\Omega_0^*=\emptyset$, non ?
%     \item Dans la "velocity method" avec $\Phi_t=id+tV$, si $V$ est telle que $div (V\circ \Phi_t^{-1}) = 0$, on doit avoir 
%     $| \Omega_{0,t}|:=| \Phi_t(\Omega_0)|= |\Omega_0|$ i.e. $\Omega_{0,t} \in \nU_{\adm}(\Omega_0)$ avec la contrainte d'aire vérifiée.
% \end{itemize}

In this paper, we do not go as far as to solve the complete optimisation problem~\eqref{eq:Shape_optim_pb}. 
We will restrict our study to the shape sensitivity analysis of the FSI model~\eqref{eq:complete-syst-stokes-stokes-transported}. 

% We will first prove the shape differentiability of the FSI system with respect to the elastic reference domain $\Omega_0$ (see Section  \ref{sec:hadam:method}). Then, we will calculate the so-called \textit{shape derivative} of the functional $\SHPfun (\Omega_0)$ with the use of the \emph{velocity method} (see Section \ref{sec:shp-der}). The shape derivative obtained will be simplified by applying an \emph{adjoint method}.

%We start in Section \ref{sec:FSI:resolution} to \grn{establish} an existence and uniqueness result for the FSI problem \eqref{eq:complete-syst-stokes-stokes-transported}.
%Then in Section \ref{sec:hadam:method}, we present the velocity method and %prove the shape differentiability of the FSI system. 
%In Section \ref{sec:shp-der}, we finally we compute shape derivatives of the %functional $\SHPfun (\Omega_0)$ with the use of the \emph{velocity method}, % 
%and we simplify the shape derivative obtained previously by applying an \emph{adjoint method}.

%#####################################################################
%#####################################################################
\section{Existence and uniqueness result for the FSI problem}
\label{sec:FSI:resolution}

%*********************************************************************
%%*********************************************************************
%\subsection{Existence and uniqueness}

% \begin{figure}[ht]
% \centering
% \begin{tikzpicture}[scale=1.2]
% %%%fig1 Omega0
% \draw (-1.5,-1.5) -- (-1.5,1.5) -- (1.5,1.5) -- (1.5,-1.5) -- cycle;
% \draw [domain=0:2*pi, samples=80, smooth] plot ({-(1 + 0.5*exp(-cos(\x r)))*cos(\x r)/2}, {sin(\x r)/2}) ;
% \fill[color=gray!20] [domain=0:2*pi, samples=80, smooth] plot ({-(1 + 0.5*exp(-cos(\x r)))*cos(\x r)/2}, {sin(\x r)/2}) ;
% \fill[color=gray!60] (0,0) circle (0.2) ; 
% \draw (0,0) circle (0.2) ;
% \draw (0,0) node{$\omega$} ;
% \draw (-1.5,1.3) node[left]{$\had$} ;
% \draw (1.1,1.2) node{$\Omega_{0}^{c}$} ;
% \draw (0.5,0) node{$\Omega_{0}$} ;
% \end{tikzpicture}
% \caption{Initial Fluid-Structure configuration.}
% \label{fig:init-shape}
% \end{figure}

In this section, we establish an existence and uniqueness result written in Theorem \ref{thm:FSI:pt-fixe} for the FSI problem \eqref{eq:complete-syst-stokes-stokes-transported}.
In \cite{grandmont2002existence}, an existence result is obtained for the Navier-Stokes equations coupled with a St Venant-Kirchhoff material in $3$D. 
For volume forces regular and small enough, the author finds a solution, not necessarily unique, to the FSI problem by applying a fixed point procedure. 
For our purpose, the uniqueness of the solution is required to address the associated optimisation problem and its shape sensitivity analysis.
%In our case, we wish to optimize the initial distribution of elastic material, and the uniqueness of the solution seems to be essential.
For the same reason, 
we need higher regularity results for the data and the solutions of Problem %\eqref{pbm:stokes-displacement-fixed:strong} and
\eqref{pbm:elasstokes-displacement-fixed:strong}. 
The existence and uniqueness result for our semi-linearized model is obtained
by adapting what is done in \cite{grandmont2002existence}.
\\
\begin{thm}
\label{thm:FSI:pt-fixe}
Let $\had$, $\Omega_0$ and $\omega$ be domains of the form \eqref{eq:FSI_def_Omega0}-\eqref{eq:FSI_def_Omega0_c} with boundary components $\partial \had$  and $\Gamma_{\omega}$ of class $C^{3}$ and $\Gamma_0$ of class $C^{3,1}$.
Let $f\in (H^{2} ( \RR^2 ) )^2$ and  $g\in (H^1 ( \Omega_0 ))^2$. 
There exists a positive constant $C$
such that if $\lVert f \rVert_{2,2} \leq C$ and $\lVert g \rVert_{1,2} \leq C$, 
then there exists a unique solution 
\begin{equation}
(\vcT, \prT , \dsp, \LagVol ) \in 
( H^1_0 (\Omega^c_0) \cap H^3 (\Omega^c_0) )^2 \times 
( L^2_0(\Omega^c_0) \cap  H^2(\Omega^c_0) ) \times
( H^1_{0,\Gamma_{\omega}}(\Omega_0) \cap H^3(\Omega_0) )^2 \times 
H^2(\Omega_0)
\end{equation}
to the FSI problem \eqref{eq:complete-syst-stokes-stokes-transported}. 
Furthermore, there exists a positive constant $C_{FS}$ such that
\begin{equation}
\lV \vcT \rV_{3,2,\Omega^c_0}
+ \lV \prT  \rV_{2,2,\Omega^c_0}
+ \lV \dsp \rV_{3,2,\Omega_0}
+ \lV \LagVol \rV_{2,2,\Omega_0}
\leq 
C_{FS}
(\lV f \rV_{2,2, \RR^2 }
+ \lV g \rV_{1,2, \Omega_0 }) .
\label{eq:estimate:theorem}
\end{equation}
\end{thm}

Before going through the proof of Theorem \ref{thm:FSI:pt-fixe}, let us introduce some preliminary elements that allow to well define the bijective map $\eT$ introduced in \eqref{eq:def:T:ii}.

% and give a sketch of the approach used.}

%*********************************************************************
%*********************************************************************
\subsection{Preliminaries}
\label{sec:FS_resol_prelim}
Let $\rmb$ be a vector field belonging to $( H^3(\Omega_0) )^2$.  
We define the following transformation map:
\begin{equation}
\label{eq:FS_cdv_def1}
\begin{array}{rccc}
\eT  : &  ( H^3(\Omega_{0}) )^2 & \longrightarrow & ( H^3(\Omega_{0}^c) )^2 \\
  &   \rmb   & \longmapsto  &  \id_{\RR^2}  + \mathcal R  (  \gamma  ( \rmb ) ) , 
\end{array}
\end{equation}
where $\gamma$ is the trace operator on $\Gamma_0$:
\begin{equation}
\label{eq:def:trace-op}
\gamma : H^3(\Omega_{0})   \rightarrow   H^{3-1/2}(\Gamma_0) , 
\end{equation}
and $\R$ is a lifting operator from $\Gamma_0$ to $\Omega_0^c$:
\begin{equation}
\label{eq:def:lifting-op}
\R :  H^{3-1/2} (\Gamma_0) \rightarrow  H^3(\Omega_{0}^c) .
\end{equation} 
We note that $\gamma$ and $\R$ are continuous linear operators. 
The extension operator $P= \R \circ \gamma$ can then be used to define the transformation map $\eT(\dsp)$ introduced in \eqref{eq:def:T:ii}.
This map has to be a $C^1$-diffeomorphism,  which requires some regularity property of the displacement field $\dsp$. The following Lemma ensures that for a function 
$\rmb$ regular enough, the map $\eT ( \rmb )$ defined in \eqref{eq:FS_cdv_def1} can be used as a change of variable in the Stokes equations. A proof of this result can be found in \cite{grandmont2002existence}.
\begin{lem}
\label{lem:phi(b)-regularite}
There exists a positive constant  $\mathcal{M} $ such that if $\rmb \in ( H^3( \Omega_0 ) )^2$ satisfies 
\begin{equation}
\lVert \rmb \rVert_{H^3( \Omega_0 )} \leq \mathcal{M},
\end{equation}
then the following properties hold true:
\begin{enumerate}[label=\textit{(\roman*)}]
\item $\nabla ( \id_{\RR^2} + \mathcal{R}(\gamma (\rmb) ) ) = \Id + \nabla \mathcal{R} ( \gamma (\rmb) ) $ is an invertible matrix in $(H^2(\Omega_0^c ))^{2\times 2}$,
\item $\eT (\rmb) = \id_{\RR^2} + \mathcal{R} ( \gamma (\rmb) ) $ is one to one on $\overline{\Omega_0^c }$,
\item $\eT (\rmb)$ is a $C^1$-diffeomorphism from $\Omega_0^c$ onto $\eT (\rmb)(\Omega_0^c )$.
\end{enumerate}
\end{lem}
Note that the change of variables in the Stokes equations shows up some terms such as $(\gd \vcT ) F( \dsp  ))$ or $G( \dsp  )\gd  \prT$, see \eqref{eq:complete-syst-stokes-stokes-transported}. If we want them to be well-defined, we still need higher regularity for $\dsp$, and we need an algebra structure allowing products of functions. This is done with the following result offering  an algebra structure for Sobolev spaces (see \cite[Theorem 4.39, p. 106]{adams2003sobolev}).

\begin{lem}{}
\label{lem:algebra}
Let $\Omega$ be a bounded domain of $\mathbb{R}^{2}$ of class $\class^1$.
There exists a positive constant $C_{\mrm{a}}$ 
such that for all $u , \, v \in H^2(\Omega )$, we have $ u v \in  H^2(\Omega )$ and
\begin{equation}
\label{eq:lem-algebra1}
\left\lVert uv \right\rVert_{2,2,\Omega} 
\leq C_{\mrm{a}} \left\lVert u \right\rVert_{2,2,\Omega} 
       \left\lVert v \right\rVert_{2,2,\Omega}  .
\end{equation}
Furthermore, for all $ w \in H^1(\Omega )$ and  $u \in H^2(\Omega )$, we have $ u w \in  H^1(\Omega )$ and 
\begin{equation}
\label{eq:lem-algebra2}
\left\lVert u w \right\rVert_{1,2,\Omega} 
\leq C_{\mrm{a}} \left\lVert u \right\rVert_{2,2,\Omega} 
       \left\lVert w \right\rVert_{1,2,\Omega}  .
\end{equation}
\end{lem} 
Now, we define the set
\begin{equation}
\label{eq:def:Bp}
B_{\M} := \lbrace \rmb \in ( H^3(\Omega_0) )^2 \sep \lVert \rmb \rVert_{3,2} \leq \M  \rbrace.
\end{equation}
Then, from the two preceding Lemmas, the following maps are well-defined:
\begin{eqnarray}
&J : (H^3(\Omega_0))^2 \rightarrow H^2( \Omega_0^c )& \nonumber\\
&J( \rmb )= \det ( \nabla \eT ( \rmb ) ),& 
\label{def:J(T)_2}
\end{eqnarray}
\begin{eqnarray}
&G : B_{\M} \rightarrow (H^2( \Omega_0^c ))^{2\times 2}& \nonumber \\
&G( \rmb )= \cof (\gd \eT (  \rmb  ) )  ,
\label{def:G(T)_2}&
\end{eqnarray}
and  
\begin{eqnarray}
&F : B_{\M} \rightarrow (H^2( \Omega_0^c ))^{2\times 2}& \nonumber\\
&F(  \rmb  )=(\gd \eT ( \rmb ) )^{-1} \cof (\gd \eT ( \rmb ) ).& 
\label{def:F(T)_2}
\end{eqnarray}

Now we give a result concerning the regularity of $J$, $G$, $F$ (see \cite{grandmont2002existence}).
\begin{lem}{}
\label{lem:regu:JGF}
The mapping $J$ %defined from $W^{2,p}(\Omega_s)$ in $W^{1,p}(\Omega_f)$ 
is of class $C^\infty$ in $(H^3(\Omega_0))^2$.
The mappings  $G$ and $F$ %is defined from $B_{\M}$ to $W^{1,p}(\Omega_f)$ and 
are infinitely differentiable everywhere in $B_{\M}$ defined in \eqref{eq:def:Bp}. 
\end{lem}

We conclude the paragraph with some remarks which will turn out useful in Sections \ref{sec:S_constraction} and \ref{sec:domain-diff}. 
From  Lemmas \ref{lem:algebra} and \ref{lem:regu:JGF} we have that $J$ defined from $B_{\M}$ into $H^2(\Omega_0^c)$  and $G$ and $F$ defined from $B_{\M}$ into $(H^2(\Omega_0^c))^{ 2 \times 2}$ are of class $C^\infty$, and the norms of their derivatives are bounded on $B_{\M}$. We set
\begin{align}
\label{eq:point-fixe:10}
\lVert DJ \rVert_{\mathcal{M}} &:= \sup_{\rmb \in B_{\M}} \lVert DJ(\rmb) \rVert_{\mathcal{L} ( H^3 (\Omega_0) , H^2 (\Omega_0^c)) } , \\
\label{eq:point-fixe:11}
\lVert DG \rVert_{\mathcal{M}} &:= \sup_{\rmb \in B_{\M}} \lVert DG(\rmb) \rVert_{\mathcal{L} ( H^3 (\Omega_0) , ( H^2 (\Omega_0^c))^{2 \times 2} ) } , \\
\label{eq:point-fixe:12}
\lVert DF \rVert_{\mathcal{M}} &:= \sup_{\rmb \in B_{\M}} \lVert DF(\rmb) \rVert_{\mathcal{L} ( H^3 (\Omega_0) , ( H^2 (\Omega_0^c))^{2\times 2} ) } .
\end{align} 
% We can write point-wise 
% \begin{equation}
% J(\rmb) (x) 
% = \det \eT (\rmb) (x)
% %
% = 1 + \tr ( \nabla \mathcal{R} ( \gamma (\rmb) ) (x) ) + o( \lvert \nabla \mathcal{R}( \gamma (\rmb) ) (x) \rvert )    , 
% \quad
% \text{ for a.e. } x \in \Omega_0^c ,
% \end{equation}
Noting that $J(0) \equiv 1$, $\nabla \eT (0) \equiv \Id$, and that from Sobolev Injection Theorem, $H^2(\Omega_0^c)$ is continuously embedded into $L^{\infty}(\Omega_0^c)$, we can choose $\mathcal{M}$ small enough in \eqref{eq:def:Bp}, so that there exists two positive constants $0 < C_1 < C_2$, such that for all $\rmb\in B_{\M}$ we have 
\begin{equation}
\label{eq:point-fixe:05bis}
C_1 \leq 
\lVert J(\rmb) \rVert_{2,2}, \,
\lVert J(\rmb)^{-1} \rVert_{2,2}, \,
\lVert \nabla \eT (\rmb) \rVert_{2,2}, \, 
\lVert \nabla \eT (\rmb)^{-1} \rVert_{2,2} 
\leq C_2 
\end{equation}
and
\begin{equation}
\label{eq:point-fixe:05}
 C_1 \leq 
\lVert J(\rmb) \rVert_{0,\infty}, \,
\lVert J(\rmb)^{-1} \rVert_{0,\infty}, \,
\lVert \nabla \eT (\rmb) \rVert_{0,\infty}, \, 
\lVert \nabla \eT (\rmb)^{-1} \rVert_{0,\infty} 
\leq C_2 . 
\end{equation}
%
% \grn{$\Big[$\it Il faudrait quand même préciser les normes vectorielles et surtout matricielles utilisées (à ajouter à la fin de la section 2.1 Notations). Pour avoir \eqref{eq:point-fixe:05bis}, il faut une norme matricielle subordonnée (norme d'opérateur) i.e. $\|A\| = \sup_{x\ne 0} \frac{\| Ax \|}{\| x\|}$ car alors $\|I_d\|=1$. 
% Si on prend une norme de type "Frobénius" commme c'est souvent le cas,  $\|A\|_F = (\sum_{i,j} \|a_{ij}\|_{2,\Omega}^2)^{1/2}$, on a $\|I_d\|_F = (2|\Omega|)^{1/2}$ et on n'aura pas la borne sup. 2 dans \eqref{eq:point-fixe:05bis}. $\Big]$
% }
% \\\\
% \smallskip

Finally, let $\eta \in H^1 ( \RR^2 )$.
In view of Lemma \ref{lem:phi(b)-regularite}, $\eT(\rmb)$ is a $C^1$-diffeomorphism.
Thus, we have  $\eta \circ \eT(\rmb) \in H^1 (\Omega_0^c)$ and $\nabla ( \eta \circ \eT(\rmb) ) = \big((\nabla \eta ) \circ \eT(\rmb) \big) \nabla \eT(\rmb)$, where $\nabla \eT(\rmb)$ is bounded in $H^{2}(\Omega_0^c)$ and then in $L^{\infty}(\Omega_0^c)$. 
It follows that for all $\rmb \in B_{\M}$, 
\begin{equation}
\label{eq:point-fixe:04}
\lVert \eta \circ \eT(\rmb) \rVert_{1,2, \Omega_0^c} 
\leq C %4
\lVert  \eta   \rVert_{1,2, \RR^2 },
\end{equation}
for all $\eta \in H^1 ( \RR^2 )$, where $C$ is a positive constant depending on $\Omega_0$, $C_1$, and $C_2$.

\smallskip
 %{\bf A mettre ailleurs, "au bon endroit"}
Furthermore, we recall a useful calculus property called \emph{Piola's identity} (see e.g., \cite{ciarlet_three-dimensional_1988}). 
For $ 1\leq \dn < p$, and $\Psi \in ( W^{2,p} )^\dn$, we have
\begin{equation}
\label{eq:Piola-identity}
\div \left( \cof \gd \Psi \right) = 0 . 
\end{equation}

\subsection{Fixed point procedure.}\label{subsection:FixPointProc}
The proof of Theorem \ref{thm:FSI:pt-fixe} for the existence and uniqueness of the solution of the FSI problem \eqref{eq:complete-syst-stokes-stokes-transported} relies on a fixed point argument that we present in this subsection, by first considering the two following problems. 
\begin{enumerate}[label=\textbf{\arabic*.}]
\item Let $f\in ( H^2 (\RR^2 ))^2$, and $(\vcT(\rmb) , \prT (\rmb)) $ be the solution of the system
%\in H^1_0 (\Omega_0^c) \times L^2_0 (\Omega_0^c )
\begin{equation}
\label{pbm:stokes-displacement-fixed:strong}
\left\lbrace
\begin{aligned}
-\nu \div( ( \nabla  \vcT(\rmb) ) F(\rmb) ) +  G(\rmb) \nabla \prT (\rmb) & =  J(\rmb)( f \circ \eT(\rmb) ) && \text{ in } \Omega_0^c ,  \\
\div ( G(\rmb)^T \vcT(\rmb) ) & =  0 && \text{ in }  \Omega_0^c, \\
\vcT(\rmb) & =  0 && \text{ on }  \partial \Omega_0^c ,
\end{aligned}
\right.
\end{equation}
where the maps $J$, $G$ and $F$ are defined  by
\eqref{def:J(T)_2}--\eqref{def:F(T)_2}.

\smallskip
\item Let $g\in ( H^1 (\Omega_0))^2$, and $(\dsp(\rmb) , \LagVol(\rmb))$ be the solution of the system
\begin{equation}
\label{pbm:elasstokes-displacement-fixed:strong}
\left\lbrace
\begin{aligned}
-  \mu \div ( \nabla \dsp(\rmb) ) + \nabla \LagVol(\rmb)  & =  g    && \text{ in } \Omega_0 ,\\
\div \dsp(\rmb) & =  0  && \text{ in } \Omega_0 ,\\
\dsp(\rmb) & =  0  && \text{ on } \Gamma_{\omega} , \\
(   \mu  \nabla \dsp(\rmb)   - \LagVol(\rmb) \Id ) n_0  & =  ( \nu \nabla \vcT(\rmb) F(\rmb) - \prT (\rmb) G(\rmb)  ) n_0  && \text{ on } \Gamma_0 .
\end{aligned}
\right.
\end{equation}

\end{enumerate}

For a fixed $ \rmb $ small enough, we will show that the problem \eqref{pbm:stokes-displacement-fixed:strong} admits a unique solution $(\vcT(\rmb) , \prT (\rmb)) $, and then that the problem \eqref{pbm:elasstokes-displacement-fixed:strong} depending on $(\vcT(\rmb) , \prT (\rmb)) $ admits also a unique solution denoted by $(\dsp(\rmb) , \LagVol(\rmb))$
In particular we will see that $\dsp ( \rmb ) $ belongs to $H^3(\Omega_0)$.
Thus we will be able to define  a map 
\begin{equation}
\label{eq:def:contraction}
\begin{array}{rccc}
\contract  : & B_{\M}  & \longrightarrow & ( H^3(\Omega_{0}) )^2 \\
& \rmb & \longmapsto &  \dsp ( \rmb ) 
\end{array},
\end{equation}
and we will show in Section \ref{sec:S_constraction} that this map is actually a contraction, so that we can apply the Banach Fixed Point Theorem, and deduce that the solution we search for the FSI problem is unique and is given by the fixed point of $\contract$. \\

In the next subsection, we present some useful results for the resolution of problems \eqref{pbm:stokes-displacement-fixed:strong} and \eqref{pbm:elasstokes-displacement-fixed:strong}, which are investigated in \S \ref{sec:resol_fluid-pbm} and \S\ref{sec:resol_structure-pbm}, respectively. The last part of the section, \S \ref{sec:S_constraction}, is devoted to the proof of the contraction character of $\contract$.

%~~~~~~~~~~~~~~~~~~~~~~~~~~~~~~~~~~~~~~~~~~~~~~~~~~~~~~~~~~~~~~~~~~~~~
%~~~~~~~~~~~~~~~~~~~~~~~~~~~~~~~~~~~~~~~~~~~~~~~~~~~~~~~~~~~~~~~~~~~~~
\subsection{Resolution of the fluid problem}
\label{sec:resol_fluid-pbm}
Problem \eqref{pbm:stokes-displacement-fixed:strong} is a slightly perturbed incompressible Stokes problem with non-slip boundary condition.
Let us introduce the null mean-value pressure space
\begin{equation}
L_{0}^{2} ( \Omega_{0}^c )= \left\lbrace q \in L^{2} ( \Omega_{0}^c )
\ \Big\vert \   \int_{\Omega_{0}^c} q \mathrm{d}x = 0  \right\rbrace .
\end{equation}  
The following result extends the standard, and well-known, Stokes existence and uniqueness result for the fluid problem \eqref{pbm:stokes-displacement-fixed:strong}.

\begin{thm}
\label{thm:cattabriga-perturbe}
%(see \cite{grandmont2002existence}) 

Let $\Omega$ be a bounded domain of $\mathbb{R}^{2}$, having a $C^{3}$ boundary $\partial\Omega$. 
Let $f_F\in ( H^{1}(\Omega) )^2 $, $ h_F  \in  H^{2}(\Omega)$ be such that 
\begin{equation}
\label{eq:compatibility-condition:2}
\int_{\Omega}  h_F  \mathrm{d}x =0 ,
\end{equation}
and $\rmA,\rmB \in (H^{2}(\Omega))^{2 \times 2}$ be two matrix fields. 
We assume that there exists $\psi\in (H^{3}(\Omega ))^{2}$ such that 
\begin{equation}
\rmB = \cof\nabla\psi .
\end{equation}

There exists a positive constant $C_{\mrm{pert}}$ such that, if
\begin{equation}
\label{eq:stokes-norm-proche-idt}
\lVert \Id - \rmA \rVert_{(H^{2}(\Omega))^{2 \times 2}} \leq C_{\mrm{pert}},
\quad \text{ and } \quad 
\lVert \Id - \rmB \rVert_{(H^{2}(\Omega))^{2 \times 2}} \leq C_{\mrm{pert}},
\end{equation}
then there exists a unique solution $(v,p) \in (H^1_0(\Omega) \cap H^{3}(\Omega) )^{2} \times ( L^{2}_0(\Omega) \cap H^{2}(\Omega) ) $ of the perturbed Stokes system: 
\begin{equation}
\label{syst:stokes-perturb-v}
\left\lbrace
\begin{aligned}
-\nu \div ( (\gd v) \rmA ) + \rmB \nabla p &= f_F && \text{ in } \Omega , \\
\div ( \rmB ^{\tp} v ) &=  h_F  && \text{ in } \Omega , \\
v &= 0  && \text{ on } \partial\Omega ,
\end{aligned}
\right.
\end{equation}
and there exists a positive constant $C_{\mrm{F}}$ such that
\begin{equation}
\label{eq:a-p-estimate_stokes_pert}
\lVert v \rVert_{3,2,\Omega} + \lVert p \rVert_{2,2,\Omega}
\leq
C_{\mrm{F}}
( \lVert f_F \rVert_{1,2,\Omega}  
+ \lVert  h_F  \rVert_{2,2,\Omega}
  ).
\end{equation}
\end{thm}

We refer to \cite{grandmont2002existence} where the proof of this result is entirely given.
The demonstration relies on a fixed point argument -- leading to conditions \eqref{eq:stokes-norm-proche-idt} --, and on the classical regularity result of Stokes problem
%, as a consequence of what is established in \cite{agmon_douglis_niremberg1959estimatesI} and \cite{agmon_douglis_niremberg1964estimatesII}
(see e.g. \cite{boyer_mathematical_2013}). 
For a complete proof of well-posedness and regularity for Stokes problem, we may refer to \cite{cattabriga1961su_un_problema} for the $3$-dimensional case, and 
to \cite[Proposition 2.3 p. 35]{temam1984NavierStokesequations} for the $2$-dimensional case. A complete development on these questions is carried out in \cite{galdi2011introduction}.

%~~~~~~~~~~~~~~~~~~~~~~~~~~~~~~~~~~~~~~~~~~~~~~~~~~~~~~~~~~~~~~~~~~~~~
%~~~~~~~~~~~~~~~~~~~~~~~~~~~~~~~~~~~~~~~~~~~~~~~~~~~~~~~~~~~~~~~~~~~~~
\subsection{Resolution of the structure problem}
\label{sec:resol_structure-pbm}

Now we solve the structure problem \eqref{pbm:elasstokes-displacement-fixed:strong}. 
Under the assumptions we have made, the elastic material is incompressible, and the state equation is linear, so that it can be described by Stokes-like  equations. 
The only difference with the behaviour of the fluid is that we do not impose a non-slip boundary condition on $\Gamma_0$ for the structure displacement $\dsp$.
Instead of that, the equilibrium of the surface forces leads to a \emph{stress boundary condition} on $\Gamma_0$.

Usually, the Dirichlet condition for the Stokes problem implies that we have a solution for which the pressure field is defined up to a constant (which is often chosen such that the pressure has a zero mean value), whereas pure Neumann or pure stress condition brings to a solution for which the velocity field is defined up to a constant. 
In the case of a mixed boundary condition, i.e with Dirichlet condition on a part of the boundary and stress condition on the rest of the boundary, we will note that the velocity together with the pressure are completely determined, and no zero mean value has to be imposed for the pressure. 

Given a bounded connected open set $\Opn$ in $\RR^2$ with boundary $\Gamma$ and given a connected open subset $\omega \Ksub \Opn$ with boundary $\Gamma_\omega$, we define the annular domain 
\begin{equation}
\label{def:Omega:elast}
\Omega := \Opn \setminus \overline{\omega},
\end{equation}
with boundary $\partial \Omega = \Gamma \cup \Gamma_\omega$.
For such a domain, we introduce the subspace of $H^1$ functions, vanishing on the boundary component $\Gamma_\omega$: we set
\begin{equation}
\label{eq:def:H1_0_omega}
H^1_{0,\Gamma_\omega}(\Omega) := 
\lbrace u\in   H^1(\Omega)    \sep u = 0 \text{ on } \Gamma_\omega 
 \rbrace .
\end{equation}
We give an existence, uniqueness and regularity result for the solution to the structure problem when the stress boundary condition on $\Gamma$ is given.

\begin{thm}
\label{thm:resol_stokes_mixedBC}
Let $\Omega$ be a domain of the form \eqref{def:Omega:elast} with boundary components $\Gamma$ and $\Gamma_\omega$ of class $C^{3,1}$, 
and let $( g , h_S , f_b ) \in ( H^1(\Omega) )^{2}  \times
H^2(\Omega) \times
( H^{3/2}(\Gamma ) )^{2}$. 
Let $\rmA$, $\rmB \in (H^{2}(\Omega))^{2 \times 2}$ such that there exists $\psi\in (H^{3}(\Omega ))^{2}$ satisfying
\begin{equation}
\rmB = \cof\nabla\psi .
\end{equation}
There exists a positive constant $C_{\mrm{pert},2}$ such that, if
\begin{equation}
\label{eq:stokes-norm-proche-idt:2}
\lVert \Id - \rmA \rVert_{(H^{2}(\Omega))^{2 \times 2}} \leq C_{\mrm{pert},2},
\quad \text{ and } \quad 
\lVert \Id - \rmB \rVert_{(H^{2}(\Omega))^{2 \times 2}} \leq C_{\mrm{pert},2},
\end{equation}
then there exists a unique pair $(w, s)$ in $ (H^1_{0,\Gamma_\omega}(\Omega) \cap H^{3} (\Omega) )^2  \times H^2(\Omega)$ solution of the problem:
\begin{equation}
\label{eq:stokes_mixedBC:1}
\left\lbrace
\begin{aligned}
- \mu \div ( (\gd w) \rmA ) + \rmB \gd s & =  g  && \text{ in } \Omega  , \\
\div ( \rmB^\top w ) & = h_S  && \text{ in } \Omega , \\
w & = 0 && \text{ on } \Gamma_\omega  , \\
( \mu (\gd w) \rmA - s \rmB ) \nv &= f_b && \text{ on } \Gamma  ,
\end{aligned}
\right.
\end{equation}
where $\nv$ is the outward normal vector to $\Gamma$. 
Furthermore, there exists a positive constant $C_{\mrm{s}}$ depending only on $\Omega$ such that
\begin{equation}
\lV w \rV_{3,2,\Omega} + \lV s \rV_{2,2,\Omega} 
\leq C_{\mrm{s}} (
\lV  g  \rV_{1,2,\Omega}  
+ \lV  h_S  \rV_{2,2,\Omega}
+ \lV f_b \rV_{H^{3/2}(\Gamma)}
).
\end{equation}
\end{thm}

Problem \eqref{eq:stokes_mixedBC:1} involves non standard boundary conditions of different types. 
In the case where $\rmA = \rmB = \Id$, a proof of the first part of Theorem \ref{thm:resol_stokes_mixedBC} for the existence of a unique weak solution  is given in \cite[Section 3.3.3]{calisti2021}, and the regularity result can be obtained following the approach presented in \cite[Section IV.7]{boyer_mathematical_2013} in the case where the stress boundary condition lies on the whole boundary $\partial \Omega $. 
From there, Theorem \ref{thm:resol_stokes_mixedBC} can be proved in exactly the same way as Theorem \ref{thm:cattabriga-perturbe}, with a fixed point argument.

%*********************************************************************
%*********************************************************************
\subsection{Proof of Theorem \ref{thm:FSI:pt-fixe}.}
\label{sec:S_constraction}

Now, we turn to the proof of Theorem \ref{thm:FSI:pt-fixe} for the existence and uniqueness of the solution of the FSI problem \eqref{eq:complete-syst-stokes-stokes-transported} by means of the fixed point procedure introduced in subsection \ref{subsection:FixPointProc}.
From now on, we will denote by $C$ any generic positive constant depending only on $\Omega_0$ and on the constants $C_1$ and $C_2$ appearing in inequalities~\eqref{eq:point-fixe:05bis} and \eqref{eq:point-fixe:05}. 
The proof is divided into 3 steps.

\smallskip\noindent
\textbullet\textit{Step 1: continuity of the fluid problem.}
We start to prove that Problem \eqref{pbm:stokes-displacement-fixed:strong} possesses a unique solution. 
We have that $G( 0 ) = F( 0 ) = \Id$.
For $\rmb \in B_{\M}$ (see \eqref{eq:def:Bp}), we deduce from Lemma \ref{lem:regu:JGF} that if $\M$ is small enough, we have then that $G (\rmb)$ and $F (\rmb)$ are such that
$\lVert \Id - F(\rmb) \rVert_{(H^{2}(\Omega))^{2 \times 2}} \leq C_{\mrm{pert}}$ and $\lVert \Id - G(\rmb) \rVert_{(H^{2}(\Omega))^{2 \times 2}} \leq C_{\mrm{pert}}$, 
where $C_{\mrm{pert}}>0$ is the positive  constant from inequalities \eqref{eq:stokes-norm-proche-idt} of Theorem \ref{thm:cattabriga-perturbe}. 
Moreover, 
from Lemma \ref{lem:phi(b)-regularite}  we know that $\eT(\rmb)$ is a $C^1$-diffeomorphism and consequently $f \circ \eT(\rmb)\in \left(H^1(\Omega_0^c)\right)^2$. Since $J(\rmb)\in H^2(\Omega_0^c)$, we deduce from \eqref{eq:lem-algebra1} in Lemma \ref{lem:algebra} that 
$J(\rmb)(f\circ T(\rmb))\in \left(H^1(\Omega_0^c)\right)^2$.
Thus, we can apply Theorem \ref{thm:cattabriga-perturbe} with $f_F = J(\rmb) f \circ \eT(\rmb)$ and  $h_F\equiv 0$ for Problem \eqref{pbm:stokes-displacement-fixed:strong}: %(the compatibility condition \eqref{eq:compatibility-condition:2} being trivially satisfied)
for all $\rmb \in B_{\M}$ with $\M$ small enough, Problem \eqref{pbm:stokes-displacement-fixed:strong} admits a unique solution $( \vcT(\rmb) , \prT(\rmb)) \in (H^1_0(\Omega_0^c) \cap H^{3}(\Omega_0^c) )^{2} \times ( L^{2}_0(\Omega_0^c) \cap H^{2}(\Omega_0^c) )$,
satisfying the following estimate:
\begin{equation}
\label{eq:point-fixe:-1}
\lVert \vcT(\rmb) \rVert_{3,2,\Omega} + \lVert \prT(\rmb) \rVert_{2,2,\Omega}
\leq
C_{\mrm{F}}
 \lVert J(\rmb)( f \circ \eT(\rmb) )  \rVert_{1,2,\Omega}   .    
\end{equation}

\medskip
Now, we prove a continuity property for the solutions of Problem \eqref{pbm:stokes-displacement-fixed:strong}.
Let  then $( \vcT(\rmb_1) , \prT(\rmb_1))$ and $( \vcT(\rmb_2) , \prT(\rmb_2))$ be the solutions of Problem \eqref{pbm:stokes-displacement-fixed:strong} for respectively $\rmb_1$ and  
$\rmb_2$ in $B_{\M}$.
%, where
%\begin{equation}
%\label{eq:def:Bp2}
%B_{\M} := \lbrace \rmb \in ( H^3(\Omega_0) )^2 \mid \lVert \rmb %\rVert_{3,2} \leq \mathcal{M} \rbrace ,
%\end{equation}
%for $\mathcal{M}$ a given constant (see \eqref{eq:def:Bp}).
%
We set $\boldsymbol{\delta} \vcT := \vcT(\rmb_1) - \vcT(\rmb_2) $ and $\boldsymbol{\delta} \prT := \prT(\rmb_1) - \prT(\rmb_2) $. 
We want to estimate $
\lVert \boldsymbol{\delta} \vcT \rVert_{3,2, \Omega_0^c}$ and $ \lVert \boldsymbol{\delta} \prT \rVert_{2,2, \Omega_0^c}$
with respect to the difference $ \lVert \rmb_1 -\rmb_2 \rVert_{3,2, \Omega_0}$.
In view of \eqref{pbm:stokes-displacement-fixed:strong}, by difference, we infer that the pair $(\boldsymbol{\delta} \vcT, \boldsymbol{\delta} \prT)$  solves
\begin{equation}
\label{syst:delta-v}
\left\lbrace
\begin{aligned}
-\nu\div ( \nabla ( \boldsymbol{\delta} \vcT  ) F( \rmb_1 ) ) + G( \rmb_1 ) \nabla \boldsymbol{\delta} \prT 
&=  f_F && \text{ in }  \Omega^c_0 , \\
\div ( G( \rmb_1 )^{\tp} \boldsymbol{\delta} \vcT )  
&= h_F  && \text{ in }  \Omega^c_0 , \\
\boldsymbol{\delta} \vcT  &= 0 && \text{ on } \partial\Omega^c_0  ,
\end{aligned}
\right.
\end{equation}
%%% 
where now $ f_F $ and $ h_F$ are defined by 
\begin{align}
f_F
& :=  J(\rmb_1) f \circ \eT(\rmb_1) -  J(\rmb_2) f \circ \eT(\rmb_2)  
+ \nu \div( \nabla (\vcT( \rmb_2 )) ( F( \rmb_1 ) - F( \rmb_2 )  ) ) \notag \\
& \hspace{3ex}- ( G( \rmb_1 ) - G( \rmb_2 ) ) \nabla \prT ( \rmb_2 ) ,
\label{eq:point-fixe:00} \\
h_F 
& := - \div ( ( G( \rmb_1 ) - G( \rmb_2 ) )^{\tp} \vcT (\rmb_2) )  
\label{eq:point-fixe:01} .
\end{align}

The compatibility condition \eqref{eq:compatibility-condition:2} for $h_F$ is valid because of the homogeneous Dirichlet condition satisfied by $\vcT (\rmb_2)$.
In view of the regularity of $\rmb_1$, $\rmb_2$, $\vcT (\rmb_2)$ and $\prT ( \rmb_2 )$,  
we can apply Theorem \ref{thm:cattabriga-perturbe}
for Problem \eqref{syst:delta-v}. % for $m=1$ and $q=2$. 
Indeed, from Piola's identity \eqref{eq:Piola-identity}, we have that $h_F=-\div ( ( G( \rmb_1 ) - G( \rmb_2 ) )^{\tp} \vcT (\rmb_2) ) = -( G( \rmb_1 ) - G( \rmb_2 ) ) \cdot \nabla  \vcT (\rmb_2)$, which belongs to $H^2(\Omega_0^c)$ thanks to Lemma \ref{lem:algebra}. 
Still from Lemma \ref{lem:algebra}, we directly have that $\div( \nabla (\vcT( \rmb_2 )) ( F( \rmb_1 ) - F( \rmb_2 )  ) )$ is in $H^1(\Omega_0^c)$. 
From the second part \eqref{eq:lem-algebra2} of Lemma \ref{lem:algebra}, $ ( G( \rmb_1 ) - G( \rmb_2 ) ) \nabla \prT ( \rmb_2 )$ belongs to $H^1(\Omega_0^c)$. 
As a result from \eqref{eq:point-fixe:00}, we deduce that $f_F\in ( H^1(\Omega_0^c) )^2$ and we can apply Theorem \ref{thm:cattabriga-perturbe} for Problem \eqref{syst:delta-v}.
Thus, %if $\mathcal{M}$ is small enough, then 
for all $\rmb_1$, $\rmb_2$ in $B_{\M}$, 
the solution $( \boldsymbol{\delta} \vcT , \boldsymbol{\delta} \prT )$ of Problem \eqref{syst:delta-v}  belongs to $( H^1_0(\Omega_0^c) \cap H^3(\Omega_0^c) )^2  \times (  L^2_0(\Omega_0^c) \cap H^2(\Omega_0^c) )$ and satisfies 
\begin{equation}
\lVert \boldsymbol{\delta} \vcT  \rVert_{3,2,\Omega_0^c} + \lVert \boldsymbol{\delta} \prT  \rVert_{2,2,\Omega_0^c}
\leq
C_{\mrm{F}} \left(  \lVert f_F   \rVert_{1,2,\Omega_0^c}  + \lVert h_F   \rVert_{2,2,\Omega_0^c}  \right) .
\label{eq:point-fixe:000} 
\end{equation}

\medskip
Let us first estimate the term $ f_F $,
%From \cite{grandmont2002existence} Lemma 2, we have know that for $\mathcal{M}$ small enough, $\nabla \eT(\rmb)$ is an invertible matrix in $(W^{1,p}(\Omega_0^c))^{2 \times 2}$, $\eT(\rn)$ is one to one on $\overline{\Omega_0^c}$, and $\eT(\rmb)$ is a $C^1$-diffeomorphism from $\Omega_0^c$ to $\eT(\rmb) ( \Omega_0^c )$. %We also have $\eT(\rmb) \in W^{2,p}(\Omeg_0^c)$, where  
%Thus a change of variable in \eqref{def:f_F(b)} gives
%\begin{equation}
%\lVert \vcT(\rmb) \rVert_{2,p} + \lVert \prT(\rmb) \rVert_{1,p}
%\leq
%C_{L_p} \lVert f   \rVert_{0,p} .
%\end{equation}
starting by considering the terms depending on $\vcT( \rmb_2 )$ and $\prT( \rmb_2 )$ in \eqref{eq:point-fixe:00}. 
From Theorem \ref{thm:cattabriga-perturbe} %for $m=1$ and $q=2$ 
applied to problem \eqref{pbm:stokes-displacement-fixed:strong} written for $\rmb_2$, we have the estimate
\begin{equation}
\label{eq:point-fixe:02}
\lVert \gd \vcT(\rmb_2) \rVert_{2,2,\Omega_0^c,} 
+ \lVert  \gd \prT (\rmb_2) \rVert_{1,2, \Omega_0^c} 
\leq C_{\mrm{F}}   \lVert   J(\rmb_2)( f \circ \eT(\rmb_2) ) \rVert_{1,2,\Omega_0^c}.
\end{equation}

In view of Lemma \ref{lem:algebra}, and inequalities \eqref{eq:point-fixe:05bis} and \eqref{eq:point-fixe:04}, we have,  up to taking a smaller $\mathcal M$,
\begin{equation}
\label{eq:point-fixe:18}
\lVert   J(\rmb_2)( f \circ \eT(\rmb_2) ) \rVert_{1,2, \Omega_0^c} 
\leq C %8 
C_{\mrm{a}}  \lVert f \rVert_{1,2, \RR^2}.
\end{equation}
From Lemma \ref{lem:algebra} with \eqref{eq:point-fixe:02}, \eqref{eq:point-fixe:18} and \eqref{eq:point-fixe:12}, we deduce:  %Acroissement fini
\begin{align}
\lVert \gd \vcT(\rmb_2) ( F(\rmb_1) - F(\rmb_2) ) \rVert_{2,2, \Omega_0^c} 
&\leq C_{\mrm{a}} 
\lVert \gd \vcT(\rmb_2) \rVert_{2,2, \Omega_0^c} 
\lVert ( F(\rmb_1) - F(\rmb_2) ) \rVert_{2,2, \Omega_0^c}  \notag \\
&\leq C %8 
C_{\mrm{a}}^2 C_{\mrm{F}}   
\lVert f   \rVert_{1,2, \RR^2} 
  \lVert DF \rVert_{\mathcal{M}} \lVert \rmb_1 - \rmb_2 \rVert_{3,2, \Omega_0}   ,
\label{eq:point-fixe:1}
\end{align}
and similarly we find 
\begin{equation}
\lVert  ( G(\rmb_1) - G(\rmb_2) ) \gd \prT (\rmb_2) \rVert_{1,2, \Omega_0^c} 
\leq   C%8 
C_{\mrm{a}}^2 C_{\mrm{F}}   
\lVert f   \rVert_{1,2, \RR^2 } 
  \lVert DG \rVert_{\mathcal{M}} \lVert \rmb_1 - \rmb_2 \rVert_{3,2, \Omega_0} .
\label{eq:point-fixe:2}
\end{equation}
%In both previous inequalities, we can show that the remaining terms $o (\lVert \rmb_1 - \rmb_2 \rVert_{2,p})   $ are actually bounded by $C \lVert \rmb_1 - \rmb_2 \rVert_{2,p}^2$ when $\rmb_1, \rmb_2  \in B_{\M}$, and $C$ is a constant depending on $2$, $\nC_{\mrm{a}}$, $\nC_{\R\gamma}$, and $\M$.

In order to obtain a bound for $f_F$, we also need to treat the first two terms in the right-hand side of \eqref{eq:point-fixe:00}, which we rewrite as follows:
\begin{align}
\lVert  J(\rmb_1) f \circ \eT(\rmb_1) -  J(\rmb_2) f \circ \eT(\rmb_2) \rVert_{1,2, \Omega_0^c}
&\leq 
\lVert ( J(\rmb_1)  -  J(\rmb_2) )  f \circ \eT(\rmb_1) \rVert_{1,2, \Omega_0^c} \notag \\
&+
\lVert  J(\rmb_2) ( f \circ \eT(\rmb_1) - f \circ \eT(\rmb_2) \rVert_{1,2, \Omega_0^c} .
\label{eq:point-fixe:3} 
\end{align}
For the first term of the right-hand side of \eqref{eq:point-fixe:3}, 
% we recall that $H^2(\Omega_0^c)$ is embedded continuously into $L^{\infty}(\Omega_0^c)$ and thus
% \begin{equation}
% \label{eq:point-fixe:03}
% \lVert J(\rmb_1) - J(\rmb_2) \rVert_{\infty, \Omega_0^c} \leq C_{\infty} \lVert  J(\rmb_1) - J(\rmb_2) \rVert_{2,2, \Omega_0^c}   . 
% \end{equation}
% With these elements,
we have from Lemma \ref{lem:algebra} %, %\eqref{eq:point-fixe:03} 
and \eqref{eq:point-fixe:04}: 
\begin{align}
\lVert ( J(\rmb_1)  -  J(\rmb_2) )  f \circ \eT(\rmb_1) \rVert_{1,2, \Omega_0^c} 
\leq C%4 
C_{\mrm{a}} 
\lVert f   \rVert_{1,2, \RR^2} 
\lVert DJ \rVert_{\mathcal{M}} \lVert \rmb_1 - \rmb_2 \rVert_{3,2, \Omega_0}  .
\label{eq:point-fixe:4} 
\end{align}
For the second term of the right-hand side of \eqref{eq:point-fixe:3},
we rely on Lemma 5.3.9 from \cite{henrot_pierre2018shapo_geom}.  
Let us remark that it is at this stage, i.e. for the application of this Lemma, that we need more regularity for $f$ when normally $H^1$-regularity would have been enough to solve the fluid problem.
Indeed, this lemma states that if $f \in H^{2}(\mathbb{R}^2)$, then the map
\begin{equation}
( W^{1,\infty}(\RR^2 )  )^{2} \ni \theta  \mapsto  f  \circ ( \id_{\RR^2} + \theta ) \in H^1 (\mathbb{R}^2) 
\end{equation}
is of class $C^1$ in the vicinity of $0$, 
and the differential is given by $D( f  \circ ( \id_{\RR^2} + \theta ) ) \xi = ( \nabla  f  ) \circ( \id_{\RR^2} + \theta ) \cdot  \xi $ for all $\xi$ in $( W^{1,\infty}(\RR^2 )  )^{2}$. %\mathcal R  (  \gamma  ( \xi ) )
Yet we have that $\eT(\rmb)$ defined in \eqref{eq:FS_cdv_def1} can in fact be defined as $\eT(\rmb) = \id_{\RR^2} + \mathcal R  (  \gamma  ( \rmb ) )$ with  $B_{\M} \ni \rmb \mapsto  \mathcal R  (  \gamma  ( \rmb ) ) \in H^3 ( \RR^2 )$. 
From Sobolev embedding, we have that $( H^3( \RR^2 ))^2$ is continuously embedded into $( W^{1,\infty} (\RR^2) ) ^2$, and we denote by $C_{\infty}$ the embedding constant.
We also note that $\rmb \mapsto \eT(\rmb)$ is continuously affine and then smooth. 
As a consequence we have that the map
\begin{equation}
\label{eq:poin-fixe:19}
B_{\M} \ni \rmb \mapsto  f  \circ \eT(\rmb) \in ( H^1 (\mathbb{R}^2) )^2
\end{equation}
is well-defined and is of class $C^1$ in the vicinity of $0$.
Its differential is given by $D_{\rmb} ( f  \circ \eT(\rmb) ) \xi = ( \nabla  f  ) \circ \eT(\rmb) \cdot  \mathcal R  (  \gamma  ( \xi ) ) $ for all $\xi$ in $( H^3(\Omega_0 )  )^{2}$. 
In view of Lemma \ref{lem:algebra} with $f\in \left(H^2(\RR^2)\right)^2$, $D_{\rmb} ( f  \circ \eT(\rmb) ) \xi$ is indeed in $( H^1 (\mathbb{R}^2) )^2$ and
\begin{equation}
    \lVert D_{\rmb} ( f  \circ \eT(\rmb) ) \rVert_{\mathcal{L} ( ( H^3 (\Omega_0) )^2 , ( H^1 (\RR^2 ))^{2\times 2} ) }
    \leq
    C%\sqrt{2} 
    C_{\R\gamma} C_{\infty}
    \lVert ( \nabla f ) \circ \eT (\rmb) \rVert_{1,2,\RR^2} ,
\end{equation}
where $C_{\R\gamma}$ stands for the continuity constant of the operator $\R \circ \gamma$. 
Thus, for $f \in ( H^2( \RR^2 ) )^2$, we have
\begin{equation}
\lVert f \circ \eT(\rmb_1) - f \circ \eT(\rmb_2)  \rVert_{1,2, \Omega_0^c}
\leq 
C%\sqrt{2} 
C_{\R\gamma} C_{\infty}
\sup_{\rmb \in B_{\M}} \left\lbrace \lVert (\nabla f) \circ \eT(\rmb) ) \rVert_{1,2, \RR^2 }  \right\rbrace
\lVert \rmb_1 -\rmb_2 \rVert_{3,2, \Omega_0} .
\label{eq:point-fixe:5} 
\end{equation}
% Once again, from the regularity of $\rmb_1$ and $\rmb_2$ we have the following injection
% \begin{equation}
% \lVert \rmb_1 -\rmb_2 \rVert_{1,\infty} 
% \leq  C_{\infty} \lVert \rmb_1 -\rmb_2 \rVert_{3,2} ,
% \end{equation} 
In light of \eqref{eq:point-fixe:04},  we have similarly for all $\rmb \in B_{\M}$:
\begin{align}
\lVert ( \gd f ) \circ \eT(\rmb)  \rVert_{1,2, \RR^2} 
\leq C%4  
\lVert f \rVert_{2,2, \RR^2 } ,
\label{eq:point-fixe:6} 
\end{align}
and then by arguing in the same way as for \eqref{eq:point-fixe:18} we have
\begin{equation} 
\label{eq:point-fixe:06}
\lVert  J(\rmb_2) ( f \circ \eT(\rmb_1) - f \circ \eT(\rmb_2)) \rVert_{1,2, \Omega_0^c} 
\leq 
C%8\sqrt{2}  
C_{\mrm{a}} C_{\R\gamma} C_{\infty}  \lVert f \rVert_{2,2, \RR^2 } \lVert \rmb_1 -\rmb_2 \rVert_{3,2,\Omega_0} .
\end{equation}

We recall that $f_F$ is given  by \eqref{eq:point-fixe:00}. 
We have completely estimated $\lVert  f_F   \rVert_{1,2}$ by combining \eqref{eq:point-fixe:1}, \eqref{eq:point-fixe:2}, \eqref{eq:point-fixe:4}, and \eqref{eq:point-fixe:06}.
We obtain
\begin{align}
\lVert f_F   \rVert_{1,2, \Omega_0^c} 
& \leq
\lVert f   \rVert_{1,2, \RR^2 } 
\left( C% 8 
C_{\mrm{a}}^2 C_{\mrm{F}} ( \nu \lVert DF \rVert_{\mathcal{M}} +
                           \lVert DG \rVert_{\mathcal{M}} )
     + C%4 
     C_{\mrm{a}} \lVert DJ \rVert_{\mathcal{M}} \right) 
\lVert \rmb_1 - \rmb_2 \rVert_{3,2, \Omega_0} \notag \\
&+ C %8\sqrt{2}  
C_{\mrm{a}} C_{\R\gamma} C_{\infty} \lVert f \rVert_{2,2, \RR^2 } \lVert \rmb_1 - \rmb_2 \rVert_{3,2, \Omega_0} ,
\end{align}   
and finally we have a constant $C_{\bd{1}} = C_{\bd{1}} ( C_{\mrm{F}} , C_{\mrm{a}} , C_{\infty} , C_{\R\gamma} , \M )$ such that
\begin{equation}
\lVert f_F   \rVert_{1,2, \Omega_0^c} 
\leq
C_{\bd{1}}
\lVert f   \rVert_{2,2, \RR^2 }  
\lVert \rmb_1 - \rmb_2 \rVert_{3,2, \Omega_0} .
\label{eq:point-fixe:7}
\end{equation}  

\medskip
Let us now pass to the estimate of  $\lVert h_F \rVert_{2,2}$. We recall that in view of Piola's identity \eqref{eq:Piola-identity} we can write 
\begin{equation}
h_F  = - \div ( ( G( \rmb_1 ) - G( \rmb_2 ) )^{\tp} \vcT (\rmb_2) ) 
= - ( G( \rmb_1 ) - G( \rmb_2 ) ) \Mscalp  \gd \vcT(\rmb_2) ,
\end{equation}
so that in a same manner as for \eqref{eq:point-fixe:1}, we have 
\begin{align}
\lVert h_F \rVert_{2,2, \Omega_0^c}
& \leq  
C_{\mrm{a}} 
\lVert  G( \rmb_1 ) - G( \rmb_2 )  \rVert_{2,2, \Omega_0^c} \lVert  \gd \vcT(\rmb_2) \rVert_{2,2, \Omega_0^c} \notag\\
& \leq  
C%8 
C_{\mrm{a}}^2  C_{\mrm{F}}
\lVert f   \rVert_{1,2, \RR^2 } 
\lVert DG \rVert_{\mathcal{M}} \lVert \rmb_1 - \rmb_2 \rVert_{3,2, \Omega_0} .
\label{eq:point-fixe:14}
\end{align}

\medskip
At this point we have computed two upper bounds for the norms of $f_F$ and $h_F$.   
Thus, by combining \eqref{eq:point-fixe:000}, \eqref{eq:point-fixe:7}, and \eqref{eq:point-fixe:14}, we finally obtain that there exists a constant $\bd{C_{\mrm{F}}} = \bd{C_{\mrm{F}}} ( C_{\mrm{F}} , C_{\mrm{a}} , C_{\infty} , C_{\R\gamma} , \M )$ such that for all $\rmb_1$, $\rmb_2$ in $B_{\M}$: 
\begin{equation}
\lVert \boldsymbol{\delta} \vcT \rVert_{3,2, \Omega_0^c} +  \lVert \boldsymbol{\delta} \prT \rVert_{2,2, \Omega_0^c} 
\leq \bd{C_{\mrm{F}}}
%(\mathcal{M},C_{\mrm{a}},C_{\R\gamma},C_{\mrm{F}}) 
\lVert f \rVert_{2,2, \RR^2 } 
 \lVert \rmb_1 -\rmb_2 \rVert_{3,2, \Omega_0} .
 \label{eq:point-fixe:9}
%%%%%%%%+ o(\lVert \rmb_1 -\rmb_2 \rVert_{2,p}) ),
\end{equation}

%\mgt{where $o(\lVert \rmb_1 -\rmb_2 \rVert_{2,p})$ is actually bounded by $\bd{C}(\mathcal{M},C_{\mrm{a}},C_{\R\gamma},C_{\mrm{F}}) \lVert \rmb_1 -\rmb_2 \rVert_{2,p} ^2$.} \\

%\vspace{5ex}

\smallskip\noindent
\textbullet\textit{Step 2: continuity of the structure problem.}
We first prove that Problem \eqref{pbm:elasstokes-displacement-fixed:strong} has a unique solution. For $\rmb \in B_{\M}$, Problem \eqref{pbm:elasstokes-displacement-fixed:strong} involves the source term on $\Gamma_0$:
\begin{equation}
\label{eq:point-fixe:8_a}
f_b = \left[ 
\nu \nabla \vcT(\rmb) F(\rmb) - \prT (\rmb) G(\rmb)
 \right]\nvI_0 , 
\end{equation}
where $\left( \vcT (\rmb), \prT (\rmb)\right)$ is the unique solution of the fluid equations \eqref{pbm:stokes-displacement-fixed:strong} studied in Step~1.
In view of the regularity of the fields involved in the expression \eqref{eq:point-fixe:8_a} and from Lemma~\ref{lem:algebra}, we have that  
\begin{equation}\label{eq:source_term0}
\nu \nabla \vcT(\rmb)F(\rmb) - \prT (\rmb) G(\rmb) \in H^2(\Omega^c_0) .
\end{equation}
Thus $f_b$ belongs to $ ( H^{3/2}  (\Gamma_0) )^2$ and Theorem  \ref{thm:resol_stokes_mixedBC} for $\rmA = \rmB =\Id$ can be applied: for all $\rmb\in B_{\M}$, there exists a unique solution $\left(\dsp(\rmb), \LagVol(\rmb)\right) \in ( H^1_{0,\Gamma_{\omega}}(\Omega_0) \cap H^2(\Omega_0) )^2 \times H^2(\Omega_0)$ of Problem \eqref{pbm:elasstokes-displacement-fixed:strong} and there exists a positive constant $C_{\mrm{s}}$ such that
\begin{equation}\label{eq:apriori0}
\lVert  \dsp(\rmb) \rVert_{3,2,\Omega_0}
+ \lVert \LagVol(\rmb) \rVert_{2,2,\Omega_0}
\leq 
C_{\mrm{s}} \left( \lVert g \rVert_{1,2,\Omega_0} + \lVert f_b \rVert_{H^{3/2}(\Gamma_0)} \right).
\end{equation}

\medskip
Now, we establish a continuity property for Problem \eqref{pbm:elasstokes-displacement-fixed:strong}.
Let  $( \dsp(\rmb_1) , \LagVol(\rmb_1) )$ and $( \dsp(\rmb_2) , \LagVol(\rmb_2) )$ be the solutions of problem \eqref{pbm:elasstokes-displacement-fixed:strong} for respectively $\rmb_1$ and  $\rmb_2$ in $B_{\M}$ (note that in the system $\vcT(\rmb_i)$ are given and solve the fluid equation studied in Step~1).
We set $\boldsymbol{\delta} \dsp := \dsp(\rmb_1) - \dsp(\rmb_2) $ and $\boldsymbol{\delta} \LagVol := \LagVol(\rmb_1) - \LagVol(\rmb_2) $. In view of \eqref{pbm:elasstokes-displacement-fixed:strong}, by difference, we infer that the pair $(\boldsymbol \delta \dsp, \boldsymbol \delta \LagVol)$ solves
\begin{equation}
\label{syst:delta-w}
\left\lbrace
\begin{aligned}
- \mu \div ( \nabla^s \bd{\delta} \dsp ) + \nabla \bd{\delta}\LagVol  & =  0   && \text{ in } \Omega_0 ,\\
\div  \bd{\delta} \dsp & =  0  && \text{ in } \Omega_0 ,\\
 \bd{\delta} \dsp & =   0  && \text{ on } \Gamma_{\omega} , \\
(  \mu  \gd \bd{\delta} \dsp  -\bd{\delta}\LagVol \Id ) n_0  & = 
f_b
%(\nabla \vcT(\rmb_1))F(\rmb_1) - (\nabla \vcT(\rmb_2))F(\rmb_2) ) \nvI_0 & \\
% &  &  -( q(\rmb_1) G(\rmb_1) -  q(\rmb_2) G(\rmb_2)  )\nvI_0   
&& \text{ on } \Gamma_0,
\end{aligned}
\right.
\end{equation}
with $f_b$ the surface force on $\Gamma_0$,
\begin{equation}
\label{eq:point-fixe:8}
f_b = [ 
\nu \nabla \vcT(\rmb_1) F(\rmb_1) - \nu \nabla \vcT(\rmb_2) F(\rmb_2)
 - q(\rmb_1) G(\rmb_1) +  q(\rmb_2) G(\rmb_2) 
 ]\nvI_0   .
\end{equation}
In view of \eqref{eq:source_term0},
$f_b \in ( H^{3/2}  (\Gamma_0) )^2$ and Theorem  \ref{thm:resol_stokes_mixedBC} applies  giving the a priori estimate
\begin{equation}\label{eq:apriori}
\lVert  \boldsymbol{\delta}\dsp \rVert_{3,2}
+ \lVert  \boldsymbol{\delta}\LagVol \rVert_{2,2}
\leq 
C_{\mrm{s}} \lVert f_b \rVert_{H^{3/2}(\Gamma_0)} . 
\end{equation}
%for some positive constant $C_{\mrm{s}}$.
Let us furtherly bound from above the right-hand side, in order to make the norm of the difference $\rmb_1-\rmb_2$ appear.
The first two terms of $f_b$ (see expression~\eqref{eq:point-fixe:8}) satisfy
\begin{align}
%\label{eq:point-fixe:15}
& \lVert \left(\nabla \vcT(\rmb_1)F(\rmb_1) - \nabla \vcT(\rmb_2)F(\rmb_2) \right) \nvI_0 \rVert_{3/2 , 2,  \Gamma_0 }
\notag \\
& \hspace{6ex} \leq  C \left(
\lVert \nabla \vcT(\rmb_2) ( F(\rmb_1) -F(\rmb_2) )\rVert_{2,2, \Omega^c_0}   \right.
\left. +
\lVert (\nabla \vcT(\rmb_1)   - \nabla \vcT(\rmb_2))F(\rmb_1) \rVert_{2,2 , \Omega^c_0}\right).
\label{eq:point-fixe:15}
\end{align}
We bound the two terms of the right-hand side of \eqref{eq:point-fixe:15} by using respectively \eqref{eq:point-fixe:1} and \eqref{eq:point-fixe:9}, and noting that $H^{2}$ norm of $F( \rmb )$ is bounded in $B_{\M}$ by a positive constant $C_{\bd{2}} = C_{\bd{2}} ( \M )$. 
This gives:
\begin{align}
&
\nu \lVert (\nabla \vcT(\rmb_1) F(\rmb_1) - \nabla \vcT(\rmb_2)F(\rmb_2) ) \nvI_0 \rVert_{3/2 , 2 , \Gamma_0} \notag
\\
& \hspace{20ex} \leq  \nu
\left( C%8 
C_{\mrm{a}}^2 C_{\mrm{F}} \lVert DF \rVert_{\mathcal{M}} + 
C_{\mrm{a}} \bd{C_{\mrm{F}}} C_{\bd{2}} \right)
\lVert f \rVert_{2,2, \RR^2}\lVert \rmb_1 -\rmb_2 \rVert_{3,2}. \label{eq:point-fixe:16} 
\end{align}
In a same manner, exploiting \eqref{eq:point-fixe:2}, \eqref{eq:point-fixe:9}, and a bound $C_{\bd{3}} = C_{\bd{3}} ( \M )$ of  the $H^{2}$ norm of $G (\rmb )$ for $\mrm{b}$ in $B_{\M}$, we get
\begin{align}
& \lVert ( q(\rmb_1) G(\rmb_1) -  q(\rmb_2) G(\rmb_2)  )\nvI_0 \rVert_{3/2 , 2 , \Gamma_0 } \notag \\
& \hspace{20ex}  \leq 
( C%8 
C_{\mrm{a}}^2 C_{\mrm{F}}  \lVert DG \rVert_{\mathcal{M}} + 
C_{\mrm{a}} \bd{C_{\mrm{F}}} C_{\bd{3}} )
\lVert f \rVert_{2,2, \RR^2 }   \lVert \rmb_1 -\rmb_2 \rVert_{3,2}    .
\label{eq:point-fixe:17}
\end{align}
By combining \eqref{eq:apriori}, \eqref{eq:point-fixe:16}, and \eqref{eq:point-fixe:17}, we conclude that there exists
a positive constant 
$\bd{C_{\M}} = \bd{C_{\M}} ( C_{\mrm{F}} , C_{\mrm{a}} , C_{\infty} , C_{\R\gamma} , \M )$ 
%$C_{\M} =C_{\M} (\mathcal{M},C_{\mrm{a}},C_{\R\gamma},C_{\mrm{F}}) > 0$ 
such that:
\begin{equation}
\label{eq:point-fixe:13}
\lVert  \dsp(\rmb_1) - \dsp(\rmb_2) \rVert_{3,2}
+ \lVert  \LagVol(\rmb_1) - \LagVol(\rmb_2) \rVert_{2,2}
\leq C_{\mrm{s}} \bd{C_{\M}}  \lVert f \rVert_{2,2, \RR^2 }   \lVert \rmb_1 - \rmb_2 \rVert_{3,2}  .
\end{equation}  
%
%\ppl{It is worth noticing  that the constant $\bd{C_{\M}}$ depends on the constant $\M$ appearing in the definition \eqref{eq:def:Bp} of $B_{\M}$, only through the norms defined in \eqref{eq:point-fixe:10}, \eqref{eq:point-fixe:11}, and \eqref{eq:point-fixe:12}.In particular, for a fixed $\Omega_0$, $\bd{C_{\M}}$ decreases when $\M$ decreases. %, so that every time we say that we eventually need to decrease $\M$, we don't .}

\smallskip\noindent
\textbullet \textit{Step 3: contraction property.} In the sequel we prove that the map $\contract : \rmb \mapsto \dsp(\rmb)$ defined in \eqref{eq:def:contraction} is a contraction.  
From estimate \eqref{eq:point-fixe:13}, we have that there exists a positive constant $C_I$ with  $C_I C_{\mrm{s}} \bd{C_{\M}} < 1$ such that if $\lVert f \rVert_{2,2, \RR^2 } < C_I$, then  $\contract $ is a contraction in $B_{\M}$. 
From \eqref{eq:apriori0}, we deduce that 
there exists a constant $C_{II}$ such that if $\lVert f \rVert_{1,2, \RR^2 } < C_{II}$ and $\lVert g \rVert_{1,2, \Omega_0} <  C_{II}$, then 
\begin{equation}
\label{eq:poin-fixe:15}
\lVert  \dsp(\rmb )  \rVert_{3,2}
+ \lVert  \LagVol(\rmb )   \rVert_{2,2}
\leq \M .
\end{equation} 
By defining
\begin{equation}
C_{\contract} = \min ( C_I , C_{II}),
\end{equation}
we have that if $\lVert f \rVert_{2,2, \RR^2 } < C_{\contract}$ and $\lVert g \rVert_{1,2, \Omega_0} < C_{\contract}$, then the map $\contract$ is a contraction which maps $B_{\M}$ onto $B_{\M}$. 
Thus, the Banach fixed-point theorem ensures that $\contract$ admits a unique fixed point in $B_{\M}$ denoted by $\dsp$.
It results that the solution $( \vcT (\dsp) , \prT(\dsp) , \dsp , \LagVol (\dsp) )$ is the unique solution to the Fluid Structure Interaction problem \eqref{eq:complete-syst-stokes-stokes-transported}.
Finally, combining \eqref{eq:point-fixe:-1}, \eqref{eq:point-fixe:8_a}, \eqref{eq:apriori0}, and \eqref{eq:poin-fixe:15}, we obtain estimate \eqref{eq:estimate:theorem}. 
The proof of Theorem \ref{thm:FSI:pt-fixe} is then complete.

%\end{proof}

%#####################################################################
%#####################################################################
\section{Velocity method and shape differentiability  of the FSI system}
\label{sec:hadam:method}
After having proved the existence of solutions of the FSI system for a prescribed reference configuration, we now address the so called {\it shape sensitivity analysis}: we analyse the behavior of the solutions with respect to infinitesimal perturbations of the reference configuration. The section is organized as follows: we start, in \S \ref{sec:hadam:intro}, by introducing the classical velocity method. Then, in \S \ref{subsec:settings} and \S \ref{subsec:settings:fixed:domain}, we perform some preliminary computations, preparatory to the main result of this section: the shape differentiability of the solutions of the FSI problem, stated in \S \ref{sec:domain-diff}, Theorem \ref{thm:mainofsection4}.

%*********************************************************************
%*********************************************************************
\subsection{Presentation of the method} 
\label{sec:hadam:intro}

We are interested in the study of the behavior of a shape functional $\SHPfun ( \Omega )$ with respect to infinitesimal variations of its argument, the set $\Omega$. 
This topic, referred to as {\it shape derivative} or {\it shape sensitivity analysis}, is now a standard tool in shape optimization. See, e.g., \cite[Chapter 2]{sokolowski_zolesio1992introduction}, \cite[Section 5.1]{henrot_pierre2018shapo_geom}, or \cite[Chapter 6]{allaire2007conception}.

Let us present the classical approach: the {\it velocity method}.
Given an admissible domain $\Omega_0$ for $\SHPfun$, we consider a 1-parameter family of shapes $(\Omega_{0,t})_{t}$ of the form
\begin{equation}
\label{eq:def:Omega_t}
\Omega_{0,t} :=  \Phi_t (\Omega_0)  , 
\end{equation}
where $(\Phi_t)_{t}$ is family of diffeomorphisms, chosen with the following properties:
\begin{itemize}
\item[-] at $t=0$ there holds $\Phi_0=\id_{\mathbb R^2}$;
\item[-] the map $t\mapsto\Phi_t$ is of class $C^1$;
\item[-] each diffeomorphism $\Phi_t$ preserves the imposed geometrical constraints on $\Omega_0$, so that every $\Omega_{0,t}$ is admissible for $\SHPfun$.
\end{itemize}

If the function $t\mapsto \SHPfun(\Omega_{0,t})$ is differentiable at $t=0$, then it admits the following development in $t$:
$$
\SHPfun ( \Omega_{0,t} ) = \SHPfun ( \Omega_0 ) + t \SHPfun ' (\Omega_0 )   + o ( t ) .
$$
The coefficient $\SHPfun ' (\Omega_0 ) $ of $t$ is the so-called shape derivative of $\SHPfun$ at $\Omega_0$ with respect to the deformations $(\Phi_t)_t$.
\\
In the literature, it is classical to take diffeomorphisms of the form 
$$
 \Phi_t = \id_{\RR^{2}} + t  V,
$$
for a suitable vector field $V$, representing the velocity (when $t$ is seen as the time) of $\Phi_t$ at $t=0$.

In order to write the expression of $\SHPfun ' (\Omega_0 ) $, it is useful to introduce the notion of {\it material derivative} of a family of functions $(\varphi_t)_t$
defined on the family of transformed domains $( \Omega_{0,t} )_{t\geq 0}$ given by \eqref{eq:def:Omega_t}.
By definition, $\varphi_t \circ \Phi_t$ are all defined in the fixed domain $\Omega_0$. If the map $t\mapsto \varphi_t \circ \Phi_t$ is differentiable at $t=0$, we define the material derivative $\dot{\varphi}$ of $\varphi_t$ at $t=0$ as the coefficient of $t$ in the expansion
$$
\varphi_t \circ \Phi_t = \varphi_0 + t \dot{\varphi} + o(t).
$$
Note that $\varphi_0$ and $\dot{\varphi}$ do not depend on $t$.

%*********************************************************************
%*********************************************************************
\subsection{Shape transformation of the FSI problem} 
\label{subsec:settings}

%%%%%%%%%%%%%%%%%%%%%%%%%%%%%%%%%%%%%%figure
%%%%%%%%%%%%%%%%%%%
%%%%%%%%%%%%%%%%%%%
\begin{figure}[ht]
\centering
\begin{tikzpicture}[scale=1.1]
%%%fig1 Omega0
\draw (-1.5,-1.5) -- (-1.5,1.5) -- (1.5,1.5) -- (1.5,-1.5) -- cycle;
\draw [domain=0:2*pi, samples=80, smooth] plot ({-(1 + 0.5*exp(-cos(\x r)))*cos(\x r)/2}, {sin(\x r)/2}) ;
\fill[color=gray!20] [domain=0:2*pi, samples=80, smooth] plot ({-(1 + 0.5*exp(-cos(\x r)))*cos(\x r)/2}, {sin(\x r)/2}) ;
\fill[color=gray!60] (0,0) circle (0.2) ; 
\draw (0,0) circle (0.2) ;
\draw (0,0) node{$\omega$} ;
\draw (-1.5,1.5) node[above]{$\had$} ;
\draw (-1.1,1.2) node{$\Omega_{0}^{c}$} ;
\draw (0.5,0) node{$\Omega_{0}$} ;
%%%
%%%fig2 Omega0,t
\begin{scope}[xshift = 6cm]
\draw (-1.5,-1.5) -- (-1.5,1.5) -- (1.5,1.5) -- (1.5,-1.5) -- cycle;
\draw [domain=0:2*pi, samples=80, smooth] plot ({-(1 + 0.5*exp(-cos(\x r)))*cos(\x r)/2}, {(1 - 0.6*( exp(-10*(\x - 3*pi/4)^2 ) ) )*sin(\x r)/2}) ;
\fill[color=gray!20] [domain=0:2*pi, samples=80, smooth] plot ({-(1 + 0.5*exp(-cos(\x r)))*cos(\x r)/2}, {(1 - 0.6*( exp(-10*(\x - 3*pi/4)^2 ) ) )*sin(\x r)/2}) ;
\fill[color=gray!60] (0,0) circle (0.2) ; 
\draw (0,0) circle (0.2) ;
\draw (0,0) node{$\omega$} ;
\draw (-1.5,1.5) node[above]{$\had$} ;
\draw (-1.1,1.2) node{$\Omega_{0,t}^{c}$} ;
\draw (0.6,0) node{$\Omega_{0,t}$} ;
\end{scope}
%%%
%%%fig3 OmegaS
\begin{scope}[yshift = -5.5cm]
\draw (-1.5,-1.5) -- (-1.5,1.5) -- (1.5,1.5) -- (1.5,-1.5) -- cycle;
\draw [domain=0:2*pi, samples=80, smooth] plot ({-(1 + 0.5*exp(-cos(\x r)))*cos(\x r)/2}, {(1 - 0.3*( exp(-10*(\x - 3*pi/4)^2 ) +exp(-10*(\x - 5*pi/4)^2 ) ) )*sin(\x r)/2}) ;
\fill[color=gray!20] [domain=0:2*pi, samples=80, smooth] plot ({-(1 + 0.5*exp(-cos(\x r)))*cos(\x r)/2}, {(1 - 0.3*( exp(-10*(\x - 3*pi/4)^2 ) +exp(-10*(\x - 5*pi/4)^2 ) ) )*sin(\x r)/2}) ;
\fill[color=gray!60] (0,0) circle (0.2) ; 
\draw (0,0) circle (0.2) ;
\draw (0,0) node{$\omega$} ;
\draw (-1.5,1.5) node[above]{$\had$} ;
\draw (-1.1,1.2) node{$\Omega_{F}$} ;
\draw (0.5,0) node{$\Omega_{S}$} ;
\end{scope}
%%%
%%%fig4 OmegaS,t
\begin{scope}[xshift = 6cm, yshift = -5.5cm]
\draw (-1.5,-1.5) -- (-1.5,1.5) -- (1.5,1.5) -- (1.5,-1.5) -- cycle;
\draw [domain=0:2*pi, samples=80, smooth] plot ({-(1 + 0.5*exp(-cos(\x r)) + 0.4*exp(-10*(\x - pi)^2 ) )*cos(\x  r)/2}, {(1 - 0.8*( exp(-6*(\x - 3*pi/4)^2 ) ) )*sin(\x r)/2}) ;
\fill[color=gray!20] [domain=0:2*pi, samples=80, smooth] plot ({-(1 + 0.5*exp(-cos(\x r)) + 0.4*exp(-10*(\x - pi)^2 ) )*cos(\x  r)/2}, {(1 - 0.8*( exp(-6*(\x - 3*pi/4)^2 ) ) )*sin(\x r)/2}) ;
\fill[color=gray!60] (0,0) circle (0.2) ; 
\draw (0,0) circle (0.2) ;
\draw (0,0) node{$\omega$} ;
\draw (-1.5,1.5) node[above]{$\had$} ;
\draw (-1.1,1.2) node{$\Omega_{F,t}$} ;
\draw (0.6,-0.1) node{$\Omega_{S,t}$} ;
\end{scope}
%%%
\draw (0,-1.5) node[below]{$Y\in\Omega_{0},\ X\in\Omega_{0}^{c} $} ;
\draw (6,-1.5) node[below]{$Y_{t}\in\Omega_{0,t},\ X_{t}\in\Omega_{0,t}^{c} $} ;
\draw (0,-7.) node[below]{$y\in\Omega_{S},\ x\in\Omega_{F} $} ;
\draw (6,-7.) node[below]{$y_{t}\in\Omega_{S,t},\ x_{t}\in\Omega_{F,t} $} ;
%%%
\draw [->] (0,-2.25) -- (0,-3.75) ;
	\draw (0,-3.) node[right]{$\eT$} ;
\draw [->] (6,-2.25) -- (6,-3.75) ;
	\draw (6,-3.) node[right]{$\eT_{t}$} ;
\draw [->] (2,0) -- (4,0) ;
	\draw (3.,0) node[above]{$\Phi  _{t}$} ;
\end{tikzpicture}
\caption{The geometries of the fluid-elasticity system submitted to transformation $\Phi  _{t}$ and the resolution of the coupled problems, characterised by $\eT_{t}$.}
\label{fig:domains}
\end{figure}
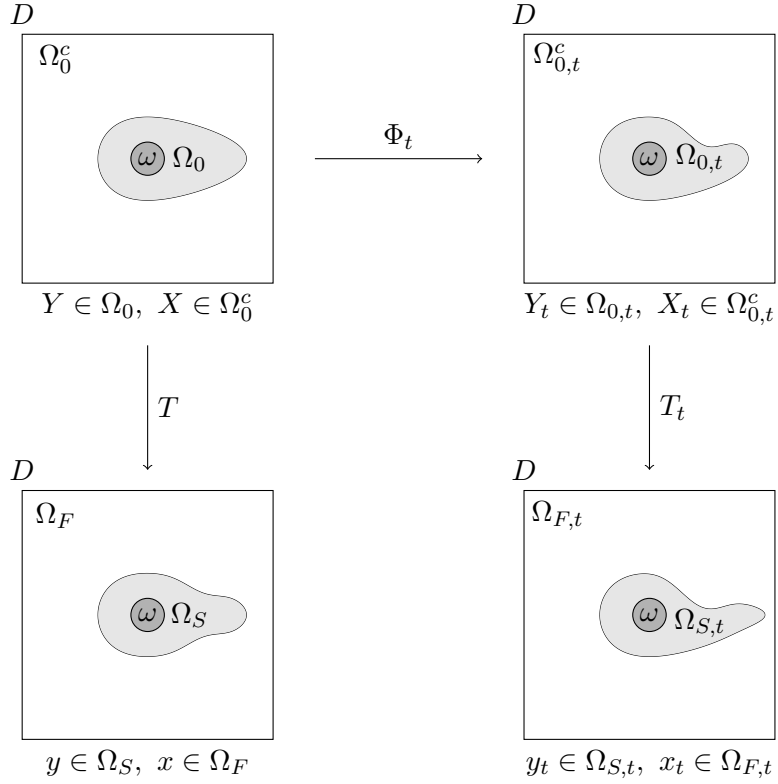
%%%%%%%%%%%%%%%%%%%
%%%%%%%%%%%%%%%%%%%
%%%%%%%%%%%%%%%%%%%%%%%%%%%%%%%%%%%%%%figure

In order to apply the velocity method in our framework, let us start by specifying the transformations $\Phi_t$ that we choose. We consider $t\geq0$ small (the threshold will be specified later) and
\begin{equation}
\label{eq:def:Phit}
\Phi  _{t} := \id_{\RR^{2}} + t V .
\end{equation}
Here $V$ is taken in the space
\begin{equation}
\label{eq:def:THETA}
\Theta :=
\left\lbrace  V  \in  H^{3}(\RR^2 , \RR^2 ) \ :\ 
\mathrm{supp} V \subset \subset  D \setminus \overline{\omega}\right\rbrace.
\end{equation}

Let $\Omega_0$ be defined as in Section \ref{sec:FSI_model}, namely the reference configuration of an elastic body contained into $D$ and attached to the rigid support $\omega$. For $t\geq0$ (small), we set 
\begin{equation}
\label{eq:def:Omega0_t}
\Omega_{0,t}:=\Phi  _{t} (\Omega_{0}),
\quad
\Omega_{0,t}^{c}:=\Phi  _{t} (\Omega_{0}^{c}),
\quad \text{ and } \quad 
\Gamma_{0,t} = \Phi  _{t} (\Gamma_{0}).
\end{equation}
We recall that $\Omega_{0}^{c}$ is the open complementary of $\Omega_0$ in $D\setminus\overline{\omega}$ (see Figure \ref{fig:domains}). 
The assumptions on $\Theta$ ensure that every $\Omega_{0,t}$ is contained into $D$ and its boundary is the union of $\Gamma_{0,t}$ and $\Gamma_\omega$.
Let $(\vct_{t}, \pr_t,  \dsp_{t}, \LagVol_t )$ be the solution of the coupled FSI problem (see \eqref{eq:S-ou-NS}-\eqref{eq:bvp:incompress_elast}) posed on the perturbed elastic body $\Omega_{0,t}$ and on the perturbed fluid domain $\Omega_{F,t}$, defined by 
\begin{align}
\Omega_{F,t} &:= 
%\had_\omega \setminus \overline{\Omega_{S,t}} , 
D\setminus (\overline{\Omega_{S,t}\cup \omega})\label{eq:def:OmegaSt}
\\
 \Omega_{S,t} &:= (\id_{\mathbb R^2} + \dsp_{t})(\Omega_{0,t}).
\label{eq:def:OmegaFt}
\end{align}
The map $\id_{\mathbb R^2} + \dsp_{t}$ is one to one from $\Omega_{0,t}$ to $\Omega_{S,t}$ for a $\dsp_t$ small enough (see Lemma~\ref{lem:phi(b)-regularite}). 
Thus $\Omega_{S,t}$ and $\Omega_{F,t}$ represent respectively the shape of the elastic body and the incompressible fluid after resolution of the coupled problem. 
In the same way as in Section \ref{sec:FSI:fixed-domain-formul}, we can transport the fluid equations on the reference domain $\Omega_{0,t}^c$. 
In principle, we could repeat the very same steps, by replacing $\Omega_0$ with $\Omega_{0,t}$ and by introducing suitable lifting and trace operators which depend on $t$. An alternative approach, that we follow here, consists in exploiting the change of variables $\Phi_t$, which allows to use the lifting and trace operators defined in \eqref{eq:def:trace-op}-\eqref{eq:def:lifting-op} constructed starting from $\Omega_0$, which of course do not depend on $t$:
\begin{align}
\mathcal{R}  \  : \   H^{3-1/2}  (\Gamma_{0 }) \longrightarrow H^3(\Omega^c_{0 }) ,
\end{align} 
and 
\begin{equation}
\gamma   \   : \ H^3(\Omega_{0 }) \longrightarrow  H^{3-1/2} (\Gamma_{0 }) . 
\end{equation}
We set
\begin{equation}
\label{eq:def:Tt:new}
\eT_t := \id_{\RR^2} + \mathcal{R} ( \gamma  ( \dsp_t \circ \Phi  _t ) ) \circ \Phi  _t^{-1},
\end{equation}
where $\dsp_t \in H^3(\Omega_{0,t})$ is the displacement solving the fluid-structure problem and $\Phi_t$ is defined in \eqref{eq:def:Phit}. The transformation $\eT_t$ maps the domain $\Omega_{0,t}$ onto $\Omega_{S,t}$ and the domain $\Omega_{0,t}^c$ onto $\Omega_{F,t}$ (see Figure \ref{fig:domains}).

Now we can define the Lagrangian fluid velocity and pressure variables
\begin{align}
\label{def:vt-}
\vcT_t &:= \vct_t \circ \eT_{t} , \\
\label{def:-qt}
 \prT_t &:= \pr_t \circ \eT_{t},
\end{align} 
and we find that the transported FSI problem for $( \vcT_t , \prT_t , \dsp_t , \LagVol_t )$ can be written as follows (see \eqref{eq:complete-syst-stokes-stokes-transported}):

\begin{equation}
\label{eq:complete-syst-stokes-stokes-transported-shape-transf1}
\left\{  
\begin{aligned}
- \nu\div ( (\gd \vcT_t ) F( \eT_t  )) + G( \eT_t  )\gd  \prT_t  
%    &   &  & & \\
%
%+ \ParamPbm( \vcT_t   \cdot G( \eT_t  ) \nabla ) \vcT_t  
 & =  (f\circ \eT_t ) J( \eT_t  ) && \text{ in }  \Omega_{0,t}^c, \\
\div   (G( \eT_t  )^{\tp} \vcT_t )  & =  0 &&  \text{ in }   \Omega_{0,t}^c ,  \\
\vcT_t & =  0 && \text{ on }  \partial \Omega_{0,t}^c , \\
\end{aligned}
\right.
\end{equation}
\begin{equation}
\label{eq:complete-syst-stokes-stokes-transported-shape-transf2}
\left\{  
\begin{aligned}
- \mu \div ( \gd \dsp_t ) + \gd \LagVol_t & =  g && \text{ in }  \Omega_{0,t}, \\ 
\div \dsp_t & =  0 && \text{ in }  \Omega_{0,t} ,\\
\dsp_t & =  0 && \text{ on }  \Gamma_{\omega} ,\\
( \mu \gd \dsp_t - \LagVol_t \Id )n_{0,t} &  =    \nu ( \gd \vcT_t ) F(\eT_t )   n_{0,t} 
%&&  
%\\
%& \hspace{10ex}  
  - \prT_t  G( \eT_t )   n_{0,t} && \text{ on }  \Gamma_{0,t} ,
\end{aligned}
\right.
\end{equation}
where we formally define for any vector field $\varphi$:
\begin{align}
%\label{def:F(Tt)}
%F( \eT_{t} )&=(\gd \eT_{t})^{-1}\cof (\gd \eT_{t}) , \\
%\label{def:G(Tt)}
%G( \eT_{t} )&= \cof (\gd \eT_{t}) , \\
%\label{def:J(Tt)}
%J( \eT_{t} )&= \det (\nabla \eT_{t}) .
%
%
\label{def:F(Tt)}
F( \varphi )&=(\gd \varphi)^{-1}\cof (\gd \varphi) , \\
\label{def:G(Tt)}
G( \varphi )&= \cof (\gd \varphi) , \\
\label{def:J(Tt)}
J( \varphi )&= \det (\nabla \varphi) .
\end{align}
It has to be noted that these maps, which will be used in the rest of the article, differ from the ones defined in \eqref{def:J(T)_2}, \eqref{def:G(T)_2}, and \eqref{def:F(T)_2} and used in Section \ref{sec:FSI:resolution}. 
Nevertheless, we still denote them by $F$, $G$, and $J$ for the sake of readability.
\\

Now we want to investigate in Section \ref{sec:domain-diff} the shape differentiability at $0$ of the solutions of the family of FSI problems \eqref{eq:complete-syst-stokes-stokes-transported-shape-transf1}-\eqref{eq:complete-syst-stokes-stokes-transported-shape-transf2},
for $t\geq 0$. 
In particular we want to show the existence of the material derivative at $0$ of $( \vcT_t , \prT_t , \dsp_t , \LagVol_t )$, which amounts to study the differentiability of $( \vcT_t , \prT_t , \dsp_t , \LagVol_t ) \circ \Phi_t$. 
For this, we transport in the following Section \ref{subsec:settings:fixed:domain} the fields which are defined on the transformed domains $\Omega_{0,t}$ and $\Omega^c_{0,t}$ onto the reference domains $\Omega_{0}$ and $\Omega^c_{0}$.

%*********************************************************************
%*********************************************************************
\subsection{Formulation in a fixed domain} 
\label{subsec:settings:fixed:domain}

%~~~~~~~~~~~~~~~~~~~~~~~~~~~~~~~~~~~~~~~~~~~~~~~~~~~~~~~~~~~~~~~~~~~~~
%~~~~~~~~~~~~~~~~~~~~~~~~~~~~~~~~~~~~~~~~~~~~~~~~~~~~~~~~~~~~~~~~~~~~~
\subsubsection{Fluid equations}

In Section \ref{subsec:settings}, we have written the Stokes equations transported onto the reference domain $\Omega_{0,t}^c$, by setting new variables $\vcT_t = \vct_t \circ \eT_{t}$ and $\pr_t = \prT_t \circ \eT_{t}$ (see~\eqref{def:vt-} and \eqref{def:-qt}). 
The variational formulation of problem \eqref{eq:complete-syst-stokes-stokes-transported-shape-transf1} can be written as follows.
\begin{equation} 
\label{eq:var:NS:sur-Omega0tc-2}
\left\lbrace
\begin{aligned}
& \text{Find }
(\vcT_t, \prT_t) \in ( H^1_0(\Omega^c_{0,t}) )^2 \times L^2_0(\Omega^c_{0,t}), \\
&\text{ such that }
\forall ( \vcTDual , \prTDual ) \in ( H^1_0(\Omega^c_{0,t}) )^2 \times L^2 (\Omega^c_{0,t})
\text{ :} \\
&\nu\int_{\Omega_{0,t}^c } \gd (\vcT_t)F(\eT_{t} ) \Mscalp  \gd  \vcTDual  
- \int_{\Omega_{0,t}^c } \prT_t ( G(\eT_{t} ) \Mscalp  \nabla  \vcTDual ) 
= \int_{\Omega_{0,t}^c } f_t J(\eT_{t} ) \cdot \vcTDual  ,  \\
& \int_{\Omega_{0,t}^c }   \prTDual   ( G(\eT_{t}) \Mscalp  \nabla \vcT_t) = 0 ,
\end{aligned}
\right.
\end{equation}   
where $f_t := f \circ \eT_{t}$, and $F( \eT_{t} )=(\gd \eT_{t})^{-1}\cof (\gd \eT_{t})$, $G( \eT_{t} )= \cof (\gd \eT_{t})$, and $J( \eT_{t} ) = \det (\nabla \eT_{t})$ (see \eqref{def:F(Tt)}, \eqref{def:G(Tt)}, and \eqref{def:J(Tt)}), with $\eT_{t}$ defined by \eqref{eq:def:Tt:new} for $\dsp_t$ being the displacement solution of the structure part of problem \eqref{eq:complete-syst-stokes-stokes-transported-shape-transf1}-\eqref{eq:complete-syst-stokes-stokes-transported-shape-transf2}.
In \eqref{eq:var:NS:sur-Omega0tc-2}, we have used the Piola identity \eqref{eq:Piola-identity} which yields $\div  ( G(T_t)^\top \vcTDual ) = G(T_t) \Mscalp \gd  \vcTDual$.
Let us recall that $\Phi_t$ defined in \eqref{eq:def:Phit} is such that  $\Phi_t ( \Omega_0^c )  = \Omega^c_{0,t}$.
Let $( \vcTDual , \prTDual ) \in ( H^1_0(\Omega^c_{0}) )^2 \times L^2_0(\Omega^c_{0})$.  
We rewrite the problem \eqref{eq:var:NS:sur-Omega0tc-2} with the test functions $( \vcTDual  \circ \Phi  _t^{-1},   \prTDual   \circ \Phi  _t^{-1} 
%- \int_{\Omega_{0,t}^c} \prTDual   \circ \Phi  _t^{-1} 
)$.
%where $\Phi_t$ is defined by \eqref{eq:def:Phit} and is such that $\Phi_t ( \Omega_0^c )  = \Omega^c_{0,t}$.
%
%
%We know that 
%
%
We have that the following relations hold
\begin{align}
\label{eq:relation:F:2}
 \nabla \Phi_t ^{-1} ( F(\eT_{t}) \circ \Phi  _t )  \nabla \Phi_t ^{-\tp}J(\Phi_t) 
&= F(\eT_{t} \circ \Phi  _t)  , \\
\label{eq:relation:G:2}
G(\eT_{t} )\circ \Phi  _t  \nabla \Phi_t ^{-\tp} J(\Phi_t)  &= G(\eT_{t} \circ \Phi  _t)  ,  \\
\label{eq:relation:J:2}
J(\eT_{t} )\circ \Phi  _t  J(\Phi_t)  &= J(\eT_{t} \circ \Phi  _t)  , 
\end{align}
where $F$, $G$, and $J$ are defined such as in \eqref{def:F(Tt)}, \eqref{def:G(Tt)}, and \eqref{def:J(Tt)}.
%\begin{align}
%\label{def:F(Tt_phi)}
%F(\eT_{t} \circ \Phi  _t) 
%& := (\gd ( \eT_{t} \circ \Phi  _t ) )^{-1}\cof (\gd ( \eT_{t} \circ \Phi  _t ) ), \\
%\label{def:G(Tt_phi)}
%G(\eT_{t} \circ \Phi  _t) 
%&=  \cof (\gd ( \eT_{t} \circ \Phi  _t )) , \\
%\label{def:J(Tt_phi)}
%J(\eT_{t} \circ \Phi  _t) 
%&= \det (\nabla ( \eT_{t} \circ \Phi  _t ) )  .
%\end{align}

Then we transport the integrals from \eqref{eq:var:NS:sur-Omega0tc-2} onto $\Omega_0^c$ by means of the change of variable $X_t = \Phi  _t(X)$, and we introduce 
\begin{align} 
\label{eq:def:vt2}
\vcT^t  & :=  \vcT_t \circ \Phi  _t  ,  \\
\label{eq:def:pt2}
 \prT^t   &  :=   \prT_t \circ \Phi  _t . 
\end{align} 
After a simplification using \eqref{eq:relation:F:2}, \eqref{eq:relation:G:2}, and  \eqref{eq:relation:J:2}, we obtain
that $( \vcT^t , \prT^t )$ is the solution of the following problem:
\begin{equation}
\label{eq:var:NS:sur-Omega0c2}
\left\lbrace
\begin{aligned}
& \text{Find } ( \vcT^t , \prT^t ) \text{ in } ( H^1_0(\Omega_{0}^c) )^2 \times L^2(\Omega_{0}^c) ,  \\
&  \text{ such that }
\forall (  \vcTDual , \prTDual ) \in ( H^1_0(\Omega_{0}^c) )^2 \times L^2 (\Omega_{0}^c)  
\text{ :} \\
&
\begin{aligned}
 a_{F}^t ( \dsp^t ; \vcT^t   , \vcTDual ) + b_{F}^t ( \dsp^t ; \vcTDual,  \prT^t   ) &= f_{F}^t ( \dsp^t ; \vcTDual    ) ,  \\
 b_{F}^t ( \dsp^t ; \vcT^t  , \prTDual ) &= 0 , 
\end{aligned}
\end{aligned}
\right.
\end{equation}
where we have defined for all $\vcT, \vcTDual$ in $H^1_0(\Omega^c_{0})$ and for all $\prT$ in $L^2 (\Omega^c_{0})$:
\begin{align}
\label{eq:def:aF^t}
a_{F}^t(\dsp^t ; \vcT , \vcTDual)  &:=
%
%\int_{\Omega_{0}^c} A_{F}^t(\dsp^t ; \vcT , \vcTDual) & :=
% 
\nu\int_{\Omega_{0}^c} ( \gd \vcT   ) F(\eT_{t}\circ \Phi  _t) \Mscalp  \gd \vcTDual   , \\
\label{eq:def:bF^t}
b_{F}^t(\dsp^t ; \vcTDual , \prT ) & :=  
%
%\int_{\Omega_{0}^c} B_{F}^t(\dsp^t ; \vcTDual , \prT ) & := 
%
- \int_{\Omega_{0}^c}  \prT  ( G(\eT_{t} \circ \Phi  _t) \Mscalp  \gd \vcTDual ) 
% J(\Phi_t) ^{-1}  
,  \\
\label{eq:def:fF^t}
f_{F}^t(\dsp^t ; \vcTDual )  &:=  
%
%\int_{\Omega_{0}^c} F_{F}^t(\dsp^t ; \vcTDual ) & :=  
%
\int_{\Omega_{0}^c} J(\eT_{t} \circ \Phi  _t) f \circ \eT_{t} \circ\Phi  _t   \cdot  \vcTDual ,
\end{align}
and where, by recalling that $\dsp_t$ is the structure displacement solution of problem \eqref{eq:complete-syst-stokes-stokes-transported-shape-transf1}-\eqref{eq:complete-syst-stokes-stokes-transported-shape-transf2} for $t \geq 0$, $\dsp^t$ is defined by 
\begin{equation}
\label{eq:def:wt_phi}
\dsp^t := \dsp_t \circ \Phi_t .
\end{equation}
We see from definition of $\eT_t$ in \eqref{eq:def:Tt:new} that we have
\begin{equation}\label{eq:T_toP_t}
\eT_t \circ \Phi_t = \Phi_t +  \mathcal{R} ( \gamma  ( \dsp^t ) ).
\end{equation}
The weak formulation \eqref{eq:var:NS:sur-Omega0c2} corresponds to the following problem for 
$(\vcT^t,\prT^t)$ posed in the domain $\Omega_0^c$.
%
%
%%%%%%%%%%%%%%%%%%%%%%%
% EQUATION ECRITE AVEC ALIGNED
%
\begin{equation}
\label{eq:pbm:fluid-Lag}
\left\lbrace 
\begin{aligned}
- \nu\div ( (\gd \vcT^t ) F( \eT_t \circ \Phi_t  )) + G( \eT_t \circ \Phi_t )\gd  \prT^t  
 & =  (f\circ \eT_t  \circ \Phi_t ) J( \eT_t \circ \Phi_t  ) 
 && \text{ in }  \Omega_{0}^c, \\
\div   ( G( \eT_t \circ \Phi_t  )^{\tp} \vcT^t )  & =   0 
&&  \text{ in }    \Omega_{0}^c ,  \\
\vcT^t & =   0 
&& \text{ on }  \partial \Omega_{0}^c .
\end{aligned}
\right.
\end{equation}
%
%%%%%%%%%%%%%%%%%%%%%%%
% EQUATION ECRITE AVEC ARRAY (un peu trop gros)
% \begin{equation}
% \label{eq:pbm:fluid-Lag}
% \left\{  
% \begin{array}{rllcr}
% - \nu\div ( (\gd \vcT^t ) F( \eT_t \circ \Phi_t  )) + G( \eT_t \circ \Phi_t )\gd  \prT^t  
% %    &   &  & & \\
% %
%  & = & (f\circ \eT_t  \circ \Phi_t ) J( \eT_t \circ \Phi_t  ) & \text{ in } & \Omega_{0}^c, \\
% %
% %
% \div   ( G( \eT_t \circ \Phi_t  )^{\tp} \vcT^t )  & = & 0 &  \text{in} &  \Omega_{0}^c ,  \\
% %
% %
% \vcT^t & = & 0 & \text{on} & \partial \Omega_{0}^c .
% %
% %
% \end{array}
% \right.
% \end{equation}

%~~~~~~~~~~~~~~~~~~~~~~~~~~~~~~~~~~~~~~~~~~~~~~~~~~~~~~~~~~~~~~~~~~~~~
%~~~~~~~~~~~~~~~~~~~~~~~~~~~~~~~~~~~~~~~~~~~~~~~~~~~~~~~~~~~~~~~~~~~~~
\subsubsection{Structure equations}

With the notations introduced right above, the surface force applied by the fluid on the structure can be expressed with respect to $\vcT_t$ and $\prT_t$ by $(\nu(\gd \vcT_t) F(\eT_{t}) - \prT_t G(\eT_{t})) n_{0,t}$. 
Let us then write the variational formulation of the structure problem
\eqref{eq:complete-syst-stokes-stokes-transported-shape-transf2}.
\begin{equation}
\left\lbrace
\begin{aligned}
\label{pbm:elast-incomp-t:var-form}
& \text{Find } (\dsp_t , \LagVol_t)\in ( H^1_{0, \Gamma_{\omega}}(\Omega_{0,t}) )^2 \times L^2(\Omega_{0,t}),   \\
&  \text{ such that } 
\forall ( \dspDual ,  \LagVolDual ) \in ( H^1_{0, \Gamma_{\omega}}(\Omega_{0,t}) )^2 \times L^2(\Omega_{0,t}) 
\text{ :} \\
& \mu \int_{\Omega_{0,t}}  \gd \dsp_t \Mscalp \gd  \dspDual   
-   \int_{\Omega_{0,t}} \LagVol_t \div  \dspDual 
= \int_{\Omega_{0,t}} g\cdot \dspDual  \\
&  \hspace{25ex} +  \int_{\Gamma_{0,t}}  \dspDual  \cdot (\nu(\gd \vcT_t) F(\eT_{t}) - \prT_t  G(\eT_{t})) n_{0,t}\, \mathrm{d}\Gamma_{0,t} , \\
&  \int_{\Omega_{0,t}}  \LagVolDual  \div \dsp_t = 0  .
\end{aligned}
\right.
\end{equation}
where $H^1_{0, \Gamma_{\omega}}(\Omega_{0,t})$ is defined for all $t\geq 0$ such as in \eqref{eq:def:H1_0_omega} by $
H^1_{0,\Gamma_{\omega}}( \Omega_{0,t} ) := 
\lbrace u \in   H^1( \Omega_{0,t} )    \sep u = 0 \text{ on } \Gamma_{\omega} 
 \rbrace 
$.

\smallskip
Let $( \dspDual  ,  \LagVolDual )$ be in $( H^1_{0, \Gamma_{\omega}}(\Omega_{0}) )^2 \times L^2(\Omega_{0 })$.
 We insert $( \dspDual  \circ \Phi  _t^{-1} ,   \LagVolDual  \circ \Phi  _t^{-1})$ as test functions into \eqref{pbm:elast-incomp-t:var-form}, 
and  then we transport the integrals onto $\Omega_0$ or $\Gamma_0$ by means of the change of variable $Y_t = \Phi  _t(Y)$.
We recall that we have (see e.g. \cite{ciarlet_three-dimensional_1988}):
\begin{equation}
\label{eq:change:var:surf2}
\nvI_{0,t}\, \mathrm{d}\Gamma_{0,t}  =  [ \det ( \gd \Phi_t )   \gd \Phi_t  ^{-\mathsf{T}}  \nvI_{0}  ] \, \mathrm{d}\Gamma_0 ,
\end{equation} 
where $\mathrm{d}\Gamma_0$ and $\mathrm{d}\Gamma_{0,t}$ are the length elements of the surfaces $\Gamma_0$ and $\Gamma_{0,t}$ respectively, 
and $\nvI_0$ and $\nvI_{0,t}$ are the normal vectors to $\Gamma_0$ and $\Gamma_{0,t}$ respectively. 
We also recall that $\vcT_t = \vcT^t \circ \Phi_t ^{-1}$ (see  \eqref{eq:def:vt2}), and consequently we have
\begin{equation}
\gd \vcT_t = ( \gd \vcT^t )  \nabla \Phi_t ^{-1} .
\end{equation}
Thus the surface term in \eqref{pbm:elast-incomp-t:var-form} is transported as follows
\begin{align}
& \int_{\Gamma_{0,t}}  \dspDual \circ \Phi_t^{-1}  \cdot \Big(\nu(\gd \vcT_t) F(\eT_{t}) - \prT_t  G(\eT_{t})\Big) n_{0,t}\, \mathrm{d}\Gamma_{0,t}
= \notag\\
& \hspace{5ex}
\int_{\Gamma_{0 }}  \dspDual  \cdot \Big(\nu(\gd \vcT^t) \gd\Phi_t ^{-1} ( F(\eT_{t}) \circ \Phi_t  ) \gd\Phi_t ^{-\tp}    -  \prT^t ( G(\eT_{t})  \circ \Phi_t  )   \nabla \Phi_t ^{-\tp} \Big) J(\Phi_t) n_{0 } \, \mathrm{d}\Gamma_{0 },
\label{eq:mat:der:fluid:1}
\end{align}
where we recall that $\prT_t \circ \Phi  _t  =  \prT^t $ (see \eqref{eq:def:pt2}).
In view of \eqref{eq:relation:F:2} and \eqref{eq:relation:G:2}, we can rewrite \eqref{eq:mat:der:fluid:1} as
\begin{align}
 \int_{\Gamma_{0,t}}  \dspDual \circ \Phi_t^{-1}  \cdot (\nu(\gd \vcT_t) F(\eT_{t}) - \prT_t  G(\eT_{t})) n_{0,t}\, \mathrm{d}\Gamma_{0,t}
&= \notag\\
& \hspace{-25ex}
\int_{\Gamma_{0 }}  \dspDual  \cdot (\nu(\gd \vcT^t)  F( \eT_{t}  \circ \Phi_t  )  - \prT^t   G(  \eT_{t} \circ \Phi_t  ) ) n_{0 }\, \mathrm{d}\Gamma_{0 }.
\label{eq:mat:der:fluid:2}
\end{align}
With $\dsp^t$ defined in \eqref{eq:def:wt_phi} and $\LagVol^t$ defined by 
\begin{align}
\label{eq:def:st_phi}
 \LagVol^t :=   \LagVol_t \circ \Phi  _t , 
\end{align}
we have thus that $( \dsp^t , \LagVol^t )$ is solution of the problem:
\begin{equation}
\label{eq:lag:structure:pbm}
\left\lbrace
\begin{aligned}
& \text{Find } 
( \dsp^t , \LagVol^t ) \text{ in } ( H^1_{0, \Gamma_{\omega}}(\Omega_{0}) )^2 \times L^2(\Omega_{0}) ,  \\
& 
\text{ such that }
\forall ( \dspDual , \LagVolDual ) \text{ in } ( H^1_{0, \Gamma_{\omega}}(\Omega_{0}) )^2 \times L^2(\Omega_{0})  
\text{ :} \\
&
\begin{aligned}
 a_{S}^t ( \dsp^t , \dspDual) + b_{S}^t ( \dspDual, \LagVol^t ) &= f_{S}^t ( \dsp^t ; \vcT^t ; \prT^t ;  \dspDual ),  \\
 b_{S}^t ( \dsp^t , \LagVolDual ) &= 0 ,  
\end{aligned}
\end{aligned}
\right.
\end{equation}
with $\vcT^t$ and $\prT^t$ solutions of the weak formulation \eqref{eq:var:NS:sur-Omega0c2}
and where, for all $\dsp, \dspDual$ in $( H^1_{0, \Gamma_{\omega}}(\Omega_{0}) )^2$,
and for all $\LagVol$ in $L^2(\Omega_{0})$:
\begin{align}
\label{eq:def:aS^t}  
a_{S}^t ( \dsp , \dspDual) &:= 
%
%\int_{\Omega_{0}} A_{S}^t ( \dsp , \dspDual) &:= 
%
\mu \int_{\Omega_{0}}  [ ( \gd \dsp )  \nabla \Phi_t  ^{-1} ]  \Mscalp   [ ( \gd \dspDual )  \nabla \Phi_t  ^{-1} ]    J(\Phi_t)  , \\
\label{eq:def:bS^t}
b_{S}^t ( \dspDual, \LagVol ) &:= 
%
%\int_{\Omega_{0}} B_{S}^t ( \dspDual, \LagVol ) &:= 
%
 - \int_{\Omega_{0}} \LagVol ( \Id \Mscalp  ( \gd \dspDual )  \nabla \Phi_t  ^{-1} )  J(\Phi_t)   ,  \\
\label{eq:def:fS^t}
f_{S}^t ( \dsp ; \vcT ; \prT ;  \dspDual ) &:=
%
%\int_{\Omega_{0}} F_{S}^t ( \dsp ; \vcT ; \prT ;  \dspDual ) &:= 
%
 \int_{\Omega_{0}}  J(\Phi_t)   ( g\circ \Phi  _t ) \cdot \dspDual \notag\\
 & + \int_{\Gamma_{0}} \dspDual\cdot \Big(\nu (\gd \vcT ) F( \eT_t \circ \Phi  _t ) - \prT  G( \eT_t \circ \Phi  _t ) 
 % J(\Phi_t) ^{-1}
  \Big) n_{0}\,\mathrm{d}\Gamma_0  .
\end{align}
Thus, from definitions of $F$, $G$, and $J$ given in \eqref{def:F(Tt)}-\eqref{def:G(Tt)}, we can write the transformed structure problem written on the fixed domain $\Omega_0$ as follows:
\begin{equation}
\label{eq:pbm:struct-Lag}
\left\lbrace  
\begin{aligned}
- \mu \div ( (\gd \dsp^t) F(\Phi_t) ) 
+ G(\Phi_t) \gd \LagVol^t & =   ( g\circ \Phi  _t ) J(\Phi_t) && \text{ in }  \Omega_{0}, \\ 
\div ( G(\Phi_t)^\top  \dsp^t ) & =  0 && \text{ in }  \Omega_{0} ,\\
\dsp^t & =  0 && \text{ on }  \Gamma_{\omega} ,\\
(  \mu (\gd \dsp^t) F(\Phi_t)
- \LagVol_t G(\Phi_t)   )n_{0 } 
&  =    
(
\nu ( \gd \vcT^t ) F( \eT_t \circ \Phi_t )   
%n_{0} &&   \\
%& \hspace{4ex} 
 - \prT^t  G( \eT_t \circ \Phi_t )  ) n_{0}  && \text{ on }  \Gamma_{0} .
\end{aligned}
\right.
\end{equation}

%~~~~~~~~~~~~~~~~~~~~~~~~~~~~~~~~~~~~~~~~~~~~~~~~~~~~~~~~~~~~~~~~~~~~~
%~~~~~~~~~~~~~~~~~~~~~~~~~~~~~~~~~~~~~~~~~~~~~~~~~~~~~~~~~~~~~~~~~~~~~
\subsubsection{FSI problem in a fixed domain formulation}

In the two previous subsections, we have shown that the fields $( \vcT^t , \prT^t , \dsp^t , \LagVol^t  )$ defined in 
\eqref{eq:def:vt2}, \eqref{eq:def:pt2}, \eqref{eq:def:wt_phi}, and \eqref{eq:def:st_phi}
from the solution of the transformed FSI problem~\eqref{eq:complete-syst-stokes-stokes-transported-shape-transf1}-\eqref{eq:complete-syst-stokes-stokes-transported-shape-transf2}, 
are solutions to the following problem:
%%%%%%%%%%%%%%%%%%%%%%%%%%%%%%%
%% EQUATION AVEC ALIGNED
\begin{equation}
\left\lbrace 
\begin{aligned} 
- \nu\div ( (\gd \vcT^t ) F( \eT_t \circ \Phi_t  )) + G( \eT_t \circ \Phi_t )\gd  \prT^t  
%    &   &  & & \\
%
 & =   (f\circ \eT_t  \circ \Phi_t ) J( \eT_t \circ \Phi_t  ) 
 && \text{ in }   \Omega_{0}^c, \\
\div   ( G( \eT_t \circ \Phi_t  )^{\tp} \vcT^t )  & =   0 
&&  \text{ in }   \Omega_{0}^c ,  \\
\vcT^t & =   0 
&& \text{ on }  \partial \Omega_{0}^c , \\
- \mu \div ( (\gd \dsp^t) F(\Phi_t) ) 
+ G(\Phi_t) \gd \LagVol^t & =   ( g\circ \Phi  _t ) J(\Phi_t) 
&& \text{ in }  \Omega_{0}, \\ 
\div ( G(\Phi_t)^\top  \dsp^t ) & =  0 
&& \text{ in }  \Omega_{0} ,\\
\dsp^t & =  0 
&& \text{ on }  \Gamma_{\omega} ,\\
(  \mu (\gd \dsp^t) F(\Phi_t)
- \LagVol_t G(\Phi_t)   )n_{0 } 
&  =    \nu ( \gd \vcT^t ) F( \eT_t \circ \Phi_t )   n_{0} &&  \\
  &  \hspace{7ex}  - \prT^t  G( \eT_t \circ \Phi_t )   n_{0}  
  && \text{ on }  \Gamma_{0} .
\end{aligned}
\right.
\label{pbm:FSI:t:domaine-fixe}
\end{equation}
%
%
%
%
%%%%%%%%%%%%%%%%%%%%%%%%%%%%%%%
%% EQUATION AVEC ARRAY
% \begin{equation}
% \left\{  
% \begin{array}{rllcr}
% - \nu\div ( (\gd \vcT^t ) F( \eT_t \circ \Phi_t  )) + G( \eT_t \circ \Phi_t )\gd  \prT^t  
% %    &   &  & & \\
% %
%  & = & (f\circ \eT_t  \circ \Phi_t ) J( \eT_t \circ \Phi_t  ) & \text{ in } & \Omega_{0}^c, \\
% %
% %
% \div   ( G( \eT_t \circ \Phi_t  )^{\tp} \vcT^t )  & = & 0 &  \text{in} &  \Omega_{0}^c ,  \\
% %
% %
% \vcT^t & = & 0 & \text{on} & \partial \Omega_{0}^c . \\
% %
% %
% %
% - \mu \div ( (\gd \dsp^t) F(\Phi_t) ) 
% + G(\Phi_t) \gd \LagVol^t & = &  ( g\circ \Phi  _t ) J(\Phi_t) & \text{in} & \Omega_{0}, \\ 
% %
% %
% \div ( G(\Phi_t)^\top  \dsp^t ) & = & 0 & \text{in} & \Omega_{0} ,\\
% %
% %
% \dsp^t & = & 0 & \text{on} & \Gamma_{\omega} ,\\
% %
% %
% (  \mu (\gd \dsp^t) F(\Phi_t)
% - \LagVol_t G(\Phi_t)   )n_{0 } 
% &  = &   \nu ( \gd \vcT^t ) F( \eT_t \circ \Phi_t )   n_{0} &   &   \\
% %
%   &    &   - \prT^t  G( \eT_t \circ \Phi_t )   n_{0}, & \text{on} & \Gamma_{0} .
% \end{array}
% \right.
% %
% \label{pbm:FSI:t:domaine-fixe}
% \end{equation}

%*********************************************************************
%*********************************************************************
\subsection{Uniform well-posedness for small $t$}

Applying the same procedure as in Section~\ref{sec:FSI:resolution}, we have the following result.

\begin{prop}
\label{thm:wellposed:unif}
Let $f\in (H^{2} ( \RR^2 ) )^2$ and  $g\in (H^1 (\RR^2 ))^2$.
There exists a three positive constants $t_\M$, $C_{\contract}$ and $C_{FS}$ such that if
$\lV f \rV_{2,2} \leq C_{\contract}$ and $\lV g \rV_{1,2} \leq C_{\contract}$ then for all $t\in [0 , t_\M)$, 
Problem \eqref{pbm:FSI:t:domaine-fixe} admits a unique solution 
$( \vcT^t , \prT^t , \dsp^t , \LagVol^t  )\in 
( H^1_0 (\Omega^c_0) \cap H^3 (\Omega^c_0) )^2 \times 
( L^2_0(\Omega^c_0) \cap  H^2(\Omega^c_0) ) \times
( H^1_{0,\Gamma_{\omega}}(\Omega_0) \cap H^3(\Omega_0) )^2 \times 
H^2(\Omega_0)$.
%there exists a unique solution $( \vcT^t , \prT^t , \dsp^t , \LagVol^t  )$ to problem \eqref{pbm:FSI:t:domaine-fixe},
Furthermore, there exists a positive constant $C_{FS}$ which does not depend on $t$, such that
\begin{equation}
\lV \vcT^t \rV_{3,2,\Omega^c_0}
+ \lV \prT^t  \rV_{2,2,\Omega^c_0}
+ \lV \dsp^t \rV_{3,2,\Omega_0}
+ \lV \LagVol^t \rV_{2,2,\Omega_0}
\leq 
C_{FS}
(\lV f \rV_{2,2, \RR^2 }
+ \lV g \rV_{1,2, \RR^2 })  .
\end{equation}
%where $C_{FS}$ does not depend on $t$.
\end{prop}

\noindent
 {\it Proof of Proposition} \ref{thm:wellposed:unif}.
We copy the fixed point procedure built in Section 3.2 with the new definition of the transformation $T_t$ in \eqref{eq:def:Tt:new}. This leads to consider the adapted transformation
\begin{equation}
\eT (\rmb) = \id_{\RR^2} + \R \left(\gamma ( \rmb )\right)\circ \Phi_t^{-1}
\end{equation}
for $\rmb\in H^3(\Omega_0)^2$ and to introduce the following problem
%%%%%%%%%%%%%%%%%%%%%%%%%%%%%%%
%% EQUATION AVEC ALIGNED
\begin{equation} \label{eq:FSI_uniform1}
\left\lbrace 
\begin{aligned}
- \nu\div ( (\gd \vcT^t(\rmb) ) F( \eT (\rmb) \circ \Phi_t  )) \hspace{1cm} & && \\
 + G( \eT (\rmb) \circ \Phi_t )\gd  \prT^t(\rmb)  
%    &   &  & & \\
%
 &  =  (f\circ \eT(\rmb)  \circ \Phi_t ) J( \eT(\rmb) \circ \Phi_t  ) 
 && \text{ in }   \Omega_{0}^c, 
 \\
\div   ( G( \eT(\rmb) \circ \Phi_t  )^{\tp} \vcT^t(\rmb) )  & =  0  
&&  \text{ in }   \Omega_{0}^c ,  
\\
\vcT^t(\rmb) & =  0 
&& \text{ on }  \partial \Omega_{0}^c , 
\\
- \mu \div ( (\gd \dsp^t(\rmb) ) F(\Phi_t) ) 
+ G(\Phi_t) \gd \LagVol^t(\rmb) & =   ( g\circ \Phi  _t ) J(\Phi_t) 
&& \text{ in }  \Omega_{0},
\\ 
\div ( G(\Phi_t)^\top  \dsp^t(\rmb) ) & =  0 
&& \text{ in }  \Omega_{0} ,
\\
\dsp^t & =   0 
&& \text{ on }  \Gamma_{\omega} ,
\\
(  \mu (\gd \dsp^t (\rmb) ) F(\Phi_t)
- \LagVol^t(\rmb) G(\Phi_t)   )n_{0 } 
&  =     \nu ( \gd \vcT^t(\rmb) ) F( \eT(\rmb) \circ \Phi_t )   n_{0} 
&&      
\\
  &   - \prT^t(\rmb)  G( \eT(\rmb) \circ \Phi_t )   n_{0}  
  && \text{ on }  \Gamma_{0} .
\end{aligned}
\right.
\end{equation}
%
%
%
%
%%%%%%%%%%%%%%%%%%%%%%%%%%%%%%%
%% EQUATION AVEC ARRAY
% \begin{equation} \label{eq:FSI_uniform1}
% \left\{  
% \begin{array}{rllcr}
% - \nu\div ( (\gd \vcT^t(\rmb) ) F( \eT (\rmb) \circ \Phi_t  )) + G( \eT (\rmb) \circ \Phi_t )\gd  \prT^t(\rmb)  
% %    &   &  & & \\
% %
%  & = & (f\circ \eT(\rmb)  \circ \Phi_t ) J( \eT(\rmb) \circ \Phi_t  ) %& \text{ in } & \Omega_{0}^c, 
%  \\
% %
% %
% \div   ( G( \eT(\rmb) \circ \Phi_t  )^{\tp} \vcT^t(\rmb) )  & = & 0  %&  \text{in} &  \Omega_{0}^c ,  
% \\
% %
% %
% \vcT^t(\rmb) & = & 0 %& \text{on} & \partial \Omega_{0}^c . 
% \\
% %
% %
% %
% - \mu \div ( (\gd \dsp^t(\rmb) ) F(\Phi_t) ) 
% + G(\Phi_t) \gd \LagVol^t(\rmb) & = &  ( g\circ \Phi  _t ) J(\Phi_t) %& \text{in} & \Omega_{0},
% \\ 
% %
% %
% \div ( G(\Phi_t)^\top  \dsp^t(\rmb) ) & = & 0 %& \text{in} & \Omega_{0} ,
% \\
% %
% %
% \dsp^t & = & 0 %& \text{on} & \Gamma_{\omega} ,
% \\
% %
% %
% (  \mu (\gd \dsp^t (\rmb) ) F(\Phi_t)
% - \LagVol^t(\rmb) G(\Phi_t)   )n_{0 } 
% &  = &   \nu ( \gd \vcT^t(\rmb) ) F( \eT(\rmb) \circ \Phi_t )   n_{0} %&   &   
% \\
% %
%   &    &   - \prT^t(\rmb)  G( \eT(\rmb) \circ \Phi_t )   n_{0}, %& \text{on} & \Gamma_{0} ,
% \end{array}
% \right.
% %
% \end{equation}
%
%
%
We can adapt the proof of Theorem \ref{thm:FSI:pt-fixe} in Section 3 to prove that the map 
\begin{equation}
\label{eq:def:contraction_t}
\begin{array}{rccc}
\contract_t  : & ( H^3(\Omega_{0}) )^2  & \longrightarrow & ( H^3(\Omega_{0}) )^2 \\
& \rmb & \longmapsto &  \dsp^t ( \rmb ) 
\end{array}
\end{equation}
has a unique fixed point $\dsp^t$ such that $\left(\vcT^t(\dsp^t), \prT^t(\dsp^t), \dsp^t, \LagVol^t(\dsp^t)\right)$ corresponds to the solution of Problem \eqref{pbm:FSI:t:domaine-fixe}.

\smallskip
We recall that 
$\Phi_t = \id_{\RR^2} + t V$ and we have
$\eT (\rmb) \circ \Phi_t = \id_{\RR^2} + \eta_t(\rmb)$ with
\begin{equation}
\eta_t(\rmb) :=t V + \R \left(\gamma ( \rmb )\right).
\end{equation}
We know that $\lV \R \left(\gamma ( \rmb ) \right)\rV_{3,2,\had} \leq C_{\R\gamma} \lV \rmb \rV_{3,2,\Omega_0}$.
Then, let $t_{\M} > 0$ be such that $t_{\M} \lV V \rV_{3,2} \leq C_{\R\gamma}  \M / 2$. Thus, we have that
$$
\lV \eta_t(\rmb) \rV_{3,2,\had} \leq C_{\R\gamma} \M
$$
for all $t\in [0 , t_\M)$ and for all $\rmb \in B_{\M/2} := \lbrace \rmb \in ( H^3(\Omega_0) )^2 \sep \lVert \rmb \rVert_{3,2,\Omega_0} \leq \M / 2  \rbrace$.
%Since we have that
% $\lV \eta_t(\rmb) \rV_{3,2,\had} \leq C_{\R\gamma} \M$ for all $\rmb\in B_{\M/2} := \lbrace \rmb \in ( H^3(\Omega_0) )^2 \sep \lVert \rmb \rVert_{3,2,\Omega_0} \leq \M / 2  \rbrace$,
Now, we can choose the constant $\M>0$ independent of $t$ such that
for all $u \in H^3 (\had)$ with $\lV u \rV_{3,2,\had} \leq C_{\R\gamma} \M$, then $(\id_{\RR^2} + u)$ satisfies all the  properties required in Section~\ref{sec:FSI:resolution}.
In particular, we have that, for all $t\in [0 , t_\M)$ and for all $\rmb \in B_{\M/2}$:
\begin{itemize}
    \item Lemma \ref{lem:phi(b)-regularite} and inequalities \eqref{eq:point-fixe:05bis} and \eqref{eq:point-fixe:05} are satisfied for both $\Phi_t$ and $\eT (\rmb) \circ \Phi_t$,
    \item Conditions \eqref{eq:stokes-norm-proche-idt} are satisfied for $\rmA = F( \eT (\rmb) \circ \Phi_t)$, $\rmB = G( \eT (\rmb) \circ \Phi_t )$, and  \eqref{eq:stokes-norm-proche-idt:2} are satisfied for $\rmA = F(\Phi_t)$, $\rmB = G(\Phi_t)$. 
\end{itemize}
As a consequence, we can proceed as in Section \eqref{sec:S_constraction} by applying successively Theorems~\ref{thm:cattabriga-perturbe} and \ref{thm:resol_stokes_mixedBC} in order to solve 
Problem \eqref{eq:FSI_uniform1}.
Thereafter,
we show that there exists a constant $C_{\contract}$ which depend only on $\M$ and $\Omega_0$ -- and not on $t$ -- such that if $\lV f \rV_{2,2} \leq C_{\contract}$ and $\lV g \rV_{1,2} \leq C_{\contract}$, then $\contract_t$ is a contraction and $\contract_t (B_{\M/2}) \subset B_{\M/2}$.

% \begin{rem}
% Actually, we are interested in the behaviour of $( \vcT^t , \prT^t , \dsp^t , \LagVol^t  )$ in the vicinity of $0$. 
% Thus, in the Proposition \ref{thm:wellposed:unif}, we can choose $t_\M>0$ such that $t_{\M} \lV V \rV_{3,2} \leq C_{\R\gamma}  \alpha $, which lead us to \grn{consider the ball} $B_{\M-\alpha} := \lbrace \rmb \in ( H^3(\Omega_0) )^2 \sep \lVert \rmb \rVert_{3,2} \leq \M  - \alpha  \rbrace$ for any small $\alpha > 0$.
% Such a \cancel{assumption} \grn{consideration} does not change the estimates needed to show that $\contract_t$ is a contraction, but allows for considering larger $f$ and $g$ to fulfill \grn{the property $\contract_t (B_{\M-\alpha}) \subset B_{\M-\alpha}$ because $B_{\M-\alpha}$ is then larger than $B_{\M/2}$.}
% \end{rem}

% \grn{\it 
% Question : faut-il garder cette remarque ?
% }

%*********************************************************************
%*********************************************************************
\subsection{Differentiability with respect to the domain}
\label{sec:domain-diff}

We want to analyse the shape sensitivity of these solutions, namely their behavior with respect to small variations of $t$.   
For this, we apply the classical method presented in \cite{henrot_pierre2018shapo_geom} Sections 5.3.3 and 5.3.4. 
The main result of this Section is the following.

\begin{thm}\label{thm:mainofsection4}
Under assumptions of Proposition \ref{thm:wellposed:unif}, let $( \vcT^t , \prT^t, \dsp^{t} , \LagVol^t )$ be the unique solution to the FSI problem \eqref{pbm:FSI:t:domaine-fixe} for all $t \in [0 , t_\M)$. In addition, we assume that $g$ belongs to $(H^2 (\RR^2 ))^2$.
Then the map
\begin{equation}
   t \in [0 , t_\M) \mapsto ( \vcT^t , \prT^t, \dsp^{t} , \LagVol^t )  
%   \in 
%   (H^1_0(\Omega_{0}^c) \cap H^3(\Omega_{0}^c))^2
%   \times L^2_0(\Omega_{0}^c)\cap H^2 (\Omega_{0}^c)
%   \times (H^1_{0, \Gamma_{\omega}}(\Omega_{0}) \cap H^3 (\Omega_{0}) )^2 
%   \times H^2 (\Omega_{0}) 
\end{equation}
is differentiable in the vicinity of $0$ in 
\begin{equation}
    (H^1_0(\Omega_{0}^c) \cap H^3(\Omega_{0}^c))^2
   \times L^2_0(\Omega_{0}^c)\cap H^2 (\Omega_{0}^c)
   \times (H^1_{0, \Gamma_{\omega}}(\Omega_{0}) \cap H^3 (\Omega_{0}) )^2 
   \times H^2 (\Omega_{0}) .
\end{equation}
\end{thm}

\begin{proof}
The key argument is the Implicit Function Theorem, that will be applied to an adequate operator characterizing the problem, and which depends on both $t$ and the state variables representing the solution. 

Let us set
\begin{align}
& \mathrm{H}_1 :=  (H^1_0(\Omega_{0}^c) \cap H^3(\Omega_{0}^c))^2,   
&& \mathrm{H}_2 :=  L^2_0(\Omega_{0}^c)\cap H^2 (\Omega_{0}^c)   ,  \\
& \mathrm{H}_3 :=  (H^1_{0, \Gamma_{\omega}}(\Omega_{0}) \cap H^3 (\Omega_{0}) )^2  ,  
&& \mathrm{H}_4 := H^2 (\Omega_{0})   ,  \\
& \mathrm{K}_1 :=  (H^1(\Omega_{0}^c))^2  ,  
&& \mathrm{K}_3 := (H^1 (\Omega_{0}) )^2  ,  \\
& \mathrm{K}_4 := H^1 (\Omega_{0})   ,  
&& \mathrm{K}_5 :=  H^{3/2} (\Gamma_{0})  ,  
\end{align}
and 
\begin{equation}
\mathrm{K}_2 := \left\lbrace h \in H^1 (\Omega_{0}^c) \ \Big\vert \  \int_{\Omega_0^c} h = 0   \right\rbrace .
\end{equation}
From this, we define the following sets:
\begin{equation}
\rmH := \mathrm{H}_1  \times \mathrm{H}_2  \times \mathrm{H}_3  \times \mathrm{H}_4 ,
\end{equation}
\begin{equation}
\rmK :=  
\mathrm{K}_1 \times \mathrm{K}_2 \times \mathrm{K}_3 \times \mathrm{K}_4 \times \mathrm{K}_5 .
\end{equation}
Before defining the adequate operator we want to study,
we can remark that the map $T_{t}$ defined in \eqref{eq:def:Tt:new} and involved in the FSI problem, depends both on the parameter $t$  through the map $\Phi_t$ given by \eqref{eq:def:Phit} and the field $\dsp^t$.
To make a distinction between these two dependencies, we introduce the following map  defined from $\RR_+ \times \mrm{H}_3$ to $( H^{3}(\Omega_0^c) )^2$ by
\begin{equation}
\label{eq:def:T-t-up-w-down}
\Tupt_w :=  \Phi_t + \mathcal{R} \gamma (w)    ,
\quad 
\forall t \geq 0 , \ \forall w \in \mrm{H}_3 .
\end{equation}
In this manner, the map $\Tupt_w$ depends on functions $w$ belonging to the fixed space $\mrm{H}_3$, 
and we have furthermore that 
\begin{equation}\label{eq:T_o_phi}
\eT_t \circ \Phi_t = \Tupt_{\dsp^t}.
\end{equation}

Let us denote by $\X^t$ the vector of $\rmH$ solution of the FSI problem defined for all $t\geq 0$ by
\begin{equation}
\X^t := ( \vcT^t , \prT^t, \dsp^{t} , \LagVol^t ) , 
\end{equation} 
while 
\begin{equation}
\X = (v,q,w,s)  \quad  \text{ and } \quad \Y = (\vcTDual,\dspDual,\prTDual,\LagVolDual)
\end{equation}
stand for arbitrary vectors of $\rmH$. 
The FSI coupling problem %\eqref{eq:pbm:fluid-Lag}-\eqref{eq:pbm:struct-Lag} 
\eqref{pbm:FSI:t:domaine-fixe}
leads us to define the following operator.
Let 
\begin{equation}
\rrF : \RR \times \rmH \rightarrow \rmK    
\end{equation}
be the map defined by
\begin{equation}
\label{eq:def:rFF-map:diff}
\left\lbrace
\begin{aligned}
\rrF_1 ( t , \X ) &:=  - \nu\div ( (\gd v ) F( \Tupt_w  )) 
+ G( \Tupt_w )\gd  q  
- (f\circ \Tupt_w ) J( \Tupt_w  ) , \\
\rrF_2 ( t , \X ) &:=  \div   ( G( \Tupt_w  )^{\tp} v )  , \\
\rrF_3 ( t , \X ) &:=  - \mu \div ( (\gd w ) F(\Phi_t) ) 
+ G(\Phi_t) \gd s - J(\Phi_t) ( g\circ \Phi  _t ) , \\
\rrF_4 ( t , \X ) &:=  \div ( G(\Phi_t)^\top  w ), \\
\rrF_5 ( t , \X ) &:=  \left[  \mu (\gd w ) F(\Phi_t)
- s G(\Phi_t)   
-  \nu ( \gd v ) F( \Tupt_w )  + q  G( \Tupt_w )   \right] n_{0}, 
\end{aligned}
\right.
\end{equation}
where we recall that $F( \Tupt_w )$, $G( \Tupt_w  )$, and $J( \Tupt_w  )$ are given by expressions \eqref{def:F(Tt)}, \eqref{def:G(Tt)}, and \eqref{def:J(Tt)} respectively. 
As we said, 
for $t=0$, 
the vector $\X^0  = ( \vcT^0 , \prT^0 , \dsp^0 , \LagVol^0 )$ is the solution of the coupling FSI problem %\eqref{eq:var:NS:sur-Omega0c2} and \eqref{eq:lag:structure:pbm} 
\eqref{pbm:FSI:t:domaine-fixe}
posed on $\Omega_0$ and $\Omega_0^c$. 
Thus by definition \eqref{eq:def:rFF-map:diff} of $\rrF$, we have $\rrF (0,\X^0) = 0$. 
From there, we want to apply the Implicit Functions Theorem to $\rrF$, by showing that:
\begin{enumerate}
\item $\rrF$ is of class $C^1$ in a neighbourhood of $(0,\X^0)$, \label{hyp:1}
\item $D_{\X} \rrF (0,\X^0)$ is a bi-continuous isomorphism.
\label{hyp:2}
\end{enumerate}
In this case, by uniqueness of the FSI problem, we will have as result that the map $t \mapsto \X^t$ is of class $C^1$ in a neighbourhood of $(0, \X^0)$.

\subsubsection{Step (1).}
We first show that the map $\rrF$ is of class $C^1$ in a neighbourhood of $(0,\X^0)$. Obviously, $\rrF=\rrF(t,v,q,w,s)$ is of class $C^1$ with respect to $v$, $q$ and $s$ since it is linear with  these variables. So we only have to check that $\rrF$ is also of class $C^1$ with $t$ and $w$.
We have that the map
$(t, \dspTestb ) \in \mathbb{R}_+ \times H^{3}(\Omega_0) \mapsto \gd ( \Phi_t + \mathcal{R}\gamma ( \dspTestb ) ) \in H^{2} (\Omega_0^c)$ 
is of class $C^{\infty}$. Indeed, $\dspTestb \mapsto \mathcal{R}\gamma ( \dspTestb )$ is linear and continuous and $t\mapsto \Phi_t $ is affine since
$\Phi_t := \id_{\RR^2}  + t V$ with $V \in  ( H^{3}(\RR^2)  )^2 $.
We can also show that $A\in ( H^2 (\Omega_0^c)  )^{2\times 2} \mapsto A^{-1} \in ( H^2 (\Omega_0^c) )^{2\times 2}$ is of class $C^{\infty}$ in a neighbourhood of the identity matrix $\Id$.  
Thus, the maps $t\mapsto  J(\Phi_t) \in  H^2 (\Omega_0^c) $ and $t\mapsto 
(\gd \Phi_t)^{-1}
\in ( H^2 (\Omega_0^c) )^{2\times 2}$ are $C^\infty$.
Finally, from Lemma \ref{lem:regu:JGF}, we have that $(t, \dspTestb ) \in \mathbb{R}_+ \times H^3 (\Omega_0) \mapsto F( \Tupt_w )$, $G( \Tupt_w ) \in (H^2 (\Omega_0^c) )^{2\times 2}$, and $J( \Tupt_w ) \in H^2 (\Omega_0^c) $ are of class $C^{\infty}$.
Finally, because of the regularity of $f \in (H^2 (\RR^2))^2$ and $g \in (H^2 (\RR^2))^2$, we have from Lemma 5.3.9 in \cite{henrot_pierre2018shapo_geom} that $(t, \dspTestb ) \mapsto (f\circ \Tupt_w ) J( \Tupt_w  ) $ and $(t, \dspTestb ) \mapsto 
 J(\Phi_t)  ( g\circ \Phi  _t )$ are $C^1$ in the vicinity of $0$.
% $\rrF$

\subsubsection{Step (2).}

We calculate the following element of $\rmK$:
\begin{align}
D_{\X}  \rrF (0, \X^0)  \X & =  
\begin{pmatrix}
   D_{\X} \rrF_1(0, \X^0)  \X  \\ 
   D_{\X} \rrF_2(0, \X^0)  \X  \\
   D_{\X} \rrF_3(0, \X^0)  \X  \\
   D_{\X} \rrF_4(0, \X^0)  \X  \\ 
   D_{\X} \rrF_5(0, \X^0)  \X 
\end{pmatrix}^{\tp}
,
\end{align}
for a $\X = (v,q,w,s)$ in $\rmH$, given from its components by
\begin{align}
D_{\X} \rrF_1(0, \X^0)  \X & = 
- \nu \div ( ( \gd  v   ) F( \eT^0 ) )
- \nu \div ( ( \gd \vcT^0 ) D_{w} F( \eT^0 )w ) , \notag \\
& %\hspace{2ex} 
+ G( \eT^0 )  \nabla q 
+ ( D_{w} G( \eT^0 )w )  \nabla \prT^0
- D_{w} ( J( \eT^0 ) f { \circ \eT^0 } ) w 
\\
D_{\X} \rrF_2(0, \X^0)  \X & = 
   \div ( G( \eT^0 )^{\tp}   v )
+  \div ( (D_{w} G( \eT^0 )w) \vcT^0 ) , 
\\
D_{\X} \rrF_3(0, \X^0)  \X & =  
- \mu \div (   \gd w   )   
+  \nabla s  ,
\\
D_{\X} \rrF_4(0, \X^0)  \X & =  
  \div ( w ) ,
\\
D_{\X} \rrF_5(0, \X^0)  \X & = 
[\mu  \gd  w   
- s \Id  
-\nu (\gd v ) F( \eT^0 ) 
- \nu (\gd \vcT^0 ) ( D_{w} F( \eT^0 )w )] n_{0}  \notag \\
&-[ 
 q  G( \eT^0 )
+ \prT^0  ( D_{w} G( \eT^0 )w  )
] n_{0} , 
\end{align}
where $\eT^0 := \id_{\RR^2} + \mathcal{R} \gamma ( \dsp^0 ) $,  where the expressions of $( D_{w} J( \eT^0 )w )$, $( D_{w} G( \eT^0 )w )$, and  $( D_{w} F( \eT^0 )w )$ 
are given in Appendix \ref{sec:derivJGF} with expressions \eqref{eq:DJbar}, \eqref{eq:DGbar}, and \eqref{eq:DFbar} respectively, and with
\begin{equation}
     D_{w} ( J( \eT^0 ) f { \circ \eT^0 } ) w 
     := ( D_{w} J( \eT^0 )w )  ( f \circ \eT^0 )
     + J( \eT^0 ) ( \gd f \circ \eT^0 ) \gd \eT^0 .
\end{equation}

Let 
$ \mathscr{F}  = ( 
\mathscr{F}_1 , 
\mathscr{F}_2 , 
\mathscr{F}_3 ,
\mathscr{F}_4 ,
\mathscr{F}_5   ) \in \rmK$, 
we want to show that there exists a unique $\X = (v,q,w,s) \in \rmH$ such that:
\begin{equation}\label{eq:isomorph_system}
D_{\X} \rrF (0,\X^0) \X 
= \mathscr{F}.
\end{equation}
This amounts to solving the following problem: find $(v,q,w,s) \in \rmH$ such that
\begin{equation}
\label{eq:pbm:diff-diffeom}
\left\{  
\begin{aligned}
- \nu\div ( (\gd v ) F( \eT^0  )) + G( \eT^0 )\gd  q  & =  \Bf_1 (w) + \mathscr{F}_1  && \text{ in }  \Omega_{0}^c, \\
\div   ( G( \eT^0  )^{\tp} v  )  & =  \Bf_2 (w) + \mathscr{F}_2  &&  \text{ in }   \Omega_{0}^c ,  \\
v & =  0 && \text{ on }   \partial \Omega_{0}^c , \\
 -  \mu \div  ( \gd w ) + \gd s  & =  \Bf_3 (w) + \mathscr{F}_3 && \text{ in }  \Omega_{0}, \\ 
\div w & =  \Bf_4 ( w ) + \mathscr{F}_4  && \text{ in }  \Omega_0 ,\\
w  & = 0 && \text{ on }  \Gamma_{\omega} ,\\
\left(
  \mu\gd w  - s \Id  -  \nu ( \gd v ) F( \eT^0 ) + q  G( \eT^0 )  \right)n_{0} &  =   \Bf_5 (w) + \mathscr{F}_5 && \text{ on }  \Gamma_{0} , 
\end{aligned}
\right.
\end{equation}
where the maps $\Bf_j$ for $j=1,\cdots 5$, are respectively linear forms from $\mathrm{H}_3 $ to $\mrm{K}_j$, given by $\Bf_3 \equiv \Bf_4 \equiv 0$, and 
\begin{align}
\Bf_1 ( w ) &:=
\nu \div ( ( \gd \vcT^0 ) D_{w} F( \eT^0 )w )  
- ( D_{w} G( \eT^0 )w )  \nabla \prT^0
+ D_{w} ( J( \eT^0 ) f { \circ \eT^0 } ) w  ,
\label{eq:diff:j1}
\\
\Bf_2 ( w ) &:= 
- \div ( (D_{w} G( \eT^0 )w) \vcT^0 ) ,\label{eq:diff:j2}
\\
\Bf_5 ( w ) &:= 
[ \nu (\gd \vcT^0 ) ( D_{w} F( \eT^0 )w )
- \prT^0  ( D_{w} G( \eT^0 )w  )
] n_{0} .\label{eq:diff:j5}
\end{align}
Let $\rmb \in H^{3}(\Omega_0)$ be an arbitrary field. 
In order to prove that Problem \eqref{eq:pbm:diff-diffeom} admits a unique solution, we introduce the following parametrized problem with $\rmb$:
%%%%%%%%%%%%%%%%%%%%%%%%%%%%%%%
%% EQUATION AVEC ALIGNED
\begin{equation}
\label{eq:pbm:diff-diffeom-b} 
\left\lbrace  
\begin{aligned}
- \nu\div ( (\gd v(\rmb) ) F( \eT^0  )) + G( \eT^0 )\gd  q(\rmb)  & =  \Bf_1 ( \rmb ) + \mathscr{F}_1  
&& \text{ in }  \Omega_{0}^c, \\
\div   ( G( \eT^0  )^{\tp} v(\rmb)  )  & =  \Bf_2 ( \rmb ) + \mathscr{F}_2 
&&  \text{ in }   \Omega_{0}^c ,  \\
v(\rmb) & =  0 
&& \text{ on }  \partial \Omega_{0}^c , \\
- \mu \div  (  \gd   w(\rmb) ) + \gd s(\rmb)  & =  \mathscr{F}_3 
&& \text{ in }  \Omega_{0}, \\ 
\div w(\rmb) & =  \mathscr{F}_4 
&& \text{ in }  \Omega_0 ,\\
w(\rmb)  & =  0 
&& \text{ on }  \Gamma_{\omega} ,\\
\left(
\mu \gd w(\rmb) - s(\rmb) \Id  -  \nu ( \gd v(\rmb) ) F( \eT^0 ) + q(\rmb)  G( \eT^0 )  \right)n_{0} &  =   \Bf_5 ( \rmb ) + \mathscr{F}_5 
&& \text{ on }  \Gamma_{0} .
\end{aligned}
\right.
\end{equation}
%
%
%%%%%%%%%%%%%%%%%%%%%%%%%%%%%%%%
%% EQUATION AVEC ARRAY
% \begin{equation}
% \label{eq:pbm:diff-diffeom-b} 
% \left\{  
% \begin{array}{rllcr}
% - \nu\div ( (\gd v(\rmb) ) F( \eT^0  )) + G( \eT^0 )\gd  q(\rmb)  & = & \Bf_1 ( \rmb ) + \mathscr{F}_1  & \text{ in } & \Omega_{0}^c, \\
% %
% %
% \div   ( G( \eT^0  )^{\tp} v(\rmb)  )  & = & \Bf_2 ( \rmb ) + \mathscr{F}_2 &  \text{in} &  \Omega_{0}^c ,  \\
% %
% %
% v(\rmb) & = & 0 & \text{on} & \partial \Omega_{0}^c , \\
% %
% %
% - \mu \div  (  \gd   w(\rmb) ) + \gd s(\rmb)  & = & \mathscr{F}_3 & \text{in} & \Omega_{0}, \\ 
% %
% %
% \div w(\rmb) & = & \mathscr{F}_4 & \text{in} & \Omega_0 ,\\
% %
% %
% w(\rmb)  & = & 0 & \text{on} & \Gamma_{\omega} ,\\
% %
% %
% \left(
% \mu \gd w(\rmb) - s(\rmb) \Id  -  \nu ( \gd v(\rmb) ) F( \eT^0 ) + q(\rmb)  G( \eT^0 )  \right)n_{0} &  = &  \Bf_5 ( \rmb ) + \mathscr{F}_5 & \text{on} & \Gamma_{0} .
% \end{array}
% \right.
% \end{equation}
%
%
%

In the same way as done in Section \ref{sec:S_constraction}, we 
can prove that for any $\rmb \in (H^{3}(\Omega_0))^2$, there exists a unique solution $( v (\rmb) , q (\rmb) , w (\rmb) ,  s (\rmb) ) \in \rmH$ to this problem, allowing us to define a map $\contract ( \rmb ) = w (\rmb)$.  
Indeed, the fields $\Bf_1 ( w )$, $\Bf_2 ( w )$, and $\Bf_5 ( w )$ have the required regularity to apply consecutively Theorems \ref{thm:cattabriga-perturbe} and \ref{thm:resol_stokes_mixedBC}, and $\Bf_2 ( w )$ together with $\mathscr{F}_2$ satisfy the compatibility condition \eqref{eq:compatibility-condition:2}. 
Moreover, since $\dsp^0$ is the displacement solution of the coupling FSI problem \eqref{pbm:FSI:t:domaine-fixe} for $t=0$, 
we have that $\eT^0 := \id_{\RR^2} + \mathcal{R} \gamma ( \dsp^0 ) $ is such that $F( \eT^0  )$ and $G( \eT^0  )$ satisfy the assumption \eqref{eq:stokes-norm-proche-idt} of Theorem \ref{thm:cattabriga-perturbe}. 
Thus applying consecutively Theorems \ref{thm:cattabriga-perturbe} and \ref{thm:resol_stokes_mixedBC}, we obtain a unique solution $( v (\rmb) ,  q (\rmb) ,  w (\rmb) , s (\rmb) ) \in \rmH$ to Problem \eqref{eq:pbm:diff-diffeom-b}.

Now we want to show that this map is a contraction for data $f$ and $g$ small enough.
Let $\rmb_1$ and  $\rmb_2$ be in 
$ ( H^3(\Omega_0) )^2$. 
We set $\boldsymbol{\delta} v :=  v (\rmb_1) - v (\rmb_2) $, 
$\boldsymbol{\delta} q := q (\rmb_1) - q (\rmb_2) $, 
$\boldsymbol{\delta} w := w (\rmb_1) - w (\rmb_2) $,
and $\boldsymbol{\delta} s  := s (\rmb_1) - s (\rmb_2) $. 
By linearity of Problem \eqref{eq:pbm:diff-diffeom-b},
and applying Theorem \ref{thm:cattabriga-perturbe} for $( \boldsymbol{\delta} v , \boldsymbol{\delta} q )$ and Theorems \ref{thm:resol_stokes_mixedBC} for 
$( \boldsymbol{\delta} w , \boldsymbol{\delta} s )$, we have
\begin{equation}
\lVert \boldsymbol{\delta} v  \rVert_{3,2,\Omega_0^c} + \lVert \boldsymbol{\delta} q  \rVert_{2,2,\Omega_0^c}
\leq
C_{\mrm{F}} (  
\lVert \Bf_1 ( \rmb_1 - \rmb_2 )   \rVert_{1,2,\Omega_0^c}  
+ \lVert \Bf_2 ( \rmb_1 - \rmb_2 )   \rVert_{2,2,\Omega_0^c}  ) ,
\end{equation}
and 
\begin{equation}
\lVert  \boldsymbol{\delta}  w  \rVert_{3,2}
+ \lVert  \boldsymbol{\delta} s \rVert_{2,2}
\leq C_{\mrm{s}}  (
\lV \Bf_5 ( \rmb_1 - \rmb_2 ) \rV_{H^{3/2} (\Gamma_0)}  
+ C ( \lVert \boldsymbol{\delta} v  \rVert_{3,2,\Omega_0^c} + \lVert \boldsymbol{\delta} q  \rVert_{2,2,\Omega_0^c} )
),
\end{equation} 
where $C_{\mrm{F}}$, $C_{\mrm{s}}$ depend only on $\Omega_0$ and $C_1$, $C_2$ in \eqref{eq:point-fixe:05bis}, \eqref{eq:point-fixe:05}.
%, and $\bd{C_{\M}}$ depends only on $\M$.
We can see in expressions \eqref{eq:diff:j1}, \eqref{eq:diff:j2}, and \eqref{eq:diff:j5}, by using Lemma \ref{lem:algebra} and same kind of estimations that those written in Section \ref{sec:S_constraction}, that the norms of the linear maps $\Bf_1$, $\Bf_2$, and $\Bf_5$ are bounded by the norms of $\vcT^0$, $\prT^0$, and the volume force $f$. 
Yet, from Theorem \ref{thm:FSI:pt-fixe}, we have that
\begin{equation}
\lV \vcT^0 \rV_{3,2,\Omega^c_0}
+ \lV \prT^0  \rV_{2,2,\Omega^c_0}
+ \lV \dsp^0 \rV_{3,2,\Omega_0}
+ \lV \LagVol^0 \rV_{2,2,\Omega_0} 
\leq 
C_{FS}
(\lV f \rV_{2,2, \RR^2 }
+ \lV g \rV_{1,2, \Omega_0 }) .
\end{equation}
Then we can choose the data $f$ and $g$ of our problem small enough so that $\contract$ is a contraction on $(H^{3}(\Omega_0))^2$. 
Therefore, $\contract$ admits a unique fixed point showing that Problem \eqref{eq:isomorph_system} has a unique solution $\X = (v, q, w, s) \in \rmH$. 

Finally, from Problem \eqref{eq:pbm:diff-diffeom} we have the following estimates:
\begin{equation}
    \lVert v  \rVert_{3,2,\Omega_0^c} + \lVert q  \rVert_{2,2,\Omega_0^c} +
\leq
C_{\mrm{F}} \Big[ 
\sum_{i=1}^2   {\lVert \mathscr{F}_i  \rVert}_{K_i}
+ (\lVert  \Bf_1 \rVert_{\mathcal{L}(\mrm{H}_3 , \mrm{K}_1 )}
+ \lVert  \Bf_2 \rVert_{\mathcal{L}(\mrm{H}_3 , \mrm{K}_2 )}
)
\lVert   w  \rVert_{3,2,\Omega_0}
\Big] ,
\label{eq:apriori-estimate:diff:1}
\end{equation}
and
\begin{align}
    \lVert   w  \rVert_{3,2,\Omega_0}
+ \lVert s \rVert_{2,2,\Omega_0}
 \leq
C_{\mrm{s}} \Big[  
\sum_{i=3}^5   {\lVert \mathscr{F}_i  \rVert}_{K_i}
& +  \lVert  \Bf_5 \rVert_{\mathcal{L}(\mrm{H}_3 , \mrm{K}_5 )}
\lVert   w  \rVert_{3,2,\Omega_0} 
 \notag \\
& %\hspace{15ex} 
+  C ( \lVert v  \rVert_{3,2,\Omega_0^c} + \lVert q  \rVert_{2,2,\Omega_0^c} )
\Big] .
\label{eq:apriori-estimate:diff:2}
\end{align}
Once again, $\lVert  \Bf_1 \rVert_{\mathcal{L}(\mrm{H}_3 , \mrm{K}_1 )}$, $\lVert  \Bf_2 \rVert_{\mathcal{L}(\mrm{H}_3 , \mrm{K}_2 )}$, and $\lVert  \Bf_5 \rVert_{\mathcal{L}(\mrm{H}_3 , \mrm{K}_5 )}$ can be chosen small enough so that combining \eqref{eq:apriori-estimate:diff:1} and \eqref{eq:apriori-estimate:diff:2}, we obtain that the solution $\X = (v, q, w s) \in \rmH$ of the linear elliptic system  \eqref{eq:isomorph_system} (see also \eqref{eq:pbm:diff-diffeom}), satisfies the following estimate
\begin{equation}
    \lVert v \rVert_{3,2,\Omega_0^c} + \lVert q  \rVert_{2,2,\Omega_0^c} +
    \lVert   w  \rVert_{3,2,\Omega_0}
+ \lVert s \rVert_{2,2,\Omega_0}
\leq
C \sum_{i=1}^5   {\lVert \mathscr{F}_i  \rVert}_{K_i}  \ ,
\end{equation}
where $C$ is a positive constant depending on the norms of  $(\vcT^0,\prT^0,\dsp^0,\LagVol^0)$, $f$ and $g$.
Then,
$D_{\X} \rrF (0,\X^0)$ is a bi-continuous isomorphism.

\end{proof}

\section{Shape derivative of $\SHPfun (\Omega) $}
\label{sec:shp-der}

%*********************************************************************
%*********************************************************************
\subsection{Direct calculus}
\label{sec:shader:costfunc}

In this paragraph, we compute the shape derivative of functionals depending on the FSI problem.
In the last part, we give an example of an energy type functional.

We consider a functional of the form
\begin{equation}
\label{eq:def:grnl:shpfun}    
\SHPfun (\Omega_0) \hspace{-0.5ex}
= \hspace{-0.5ex} \SHPfun _S (\Omega_0) + \SHPfun _F(\Omega_0) \hspace{-0.5ex}
= \hspace{-0.5ex} \int_{\Omega_0} \hspace{-0.5ex} j_S (Y, \dsp(Y), \gd \dsp(Y)) \,\mathrm{d}Y
+ \int_{\Omega_F} \hspace{-0.5ex} j_F (x, \vct(x), \gd \vct(x)) \,\mathrm{d}x ,
\end{equation}
where $j_S$ and $j_F$ are differentiable functions. 
As we have done in the previous Section, we consider a 1-parameter family of shapes $\Omega_{0,t}$ defined in \eqref{eq:def:Omega0_t}.

Performing the shape derivative of $\SHPfun$ with respect to the deformation chosen amounts to compute the derivative of $t\mapsto \SHPfun(\Omega_{0,t})$ at $t=0$.

The shape functional evaluated on the domain $\Omega_{0,t}$ is given by:
\begin{align}
\SHPfun (\Omega_{0,t}) = \SHPfun _S (\Omega_{0,t}) + \SHPfun _F(\Omega_{0,t})
& = \int_{\Omega_{0,t}} j_S (Y, \dsp_t (Y), \gd \dsp_t (Y)) \,\mathrm{d}Y \notag\\
&+ \int_{\Omega_{F,t}} j_F (x, \vct_t (x), \gd \vct_t (x)) \,\mathrm{d}x .\notag
\end{align} 
where $(\dsp_t,\vct_t)$ are the solution fields of the FSI problem \eqref{eq:complete-syst-stokes-stokes-transported-shape-transf1},\eqref{eq:complete-syst-stokes-stokes-transported-shape-transf2}.

Let us first compute the derivative of $\SHPfun_S(\Omega_{0,t})$. 
After transporting the integral from $\Omega_{0,t}$ to $\Omega_0$, we obtain
\begin{align}
\label{eq:gnrl:shpfun:1}
\SHPfun _S (\Omega_{0,t})  
& = \int_{\Omega_{0 }} j_S \left( \Phi_t (Y ) , \dsp_t \circ \Phi_t (Y ), (\gd \dsp_t ) \circ \Phi_t (Y) \right) \det ( \gd \Phi_t) \,\mathrm{d}Y  .
\end{align}
Thus the shape derivative of $\SHPfun_S$ is given by
\begin{align}\label{eq:shape_deriv_J1}
\SHPfun _S' (\Omega_0) 
&= \int_{\Omega_0} j_S (Y, \dsp(Y), \gd \dsp(Y)) \div V \,\mathrm{d}Y \notag\\
&+ \int_{\Omega_0} D_1 j_S (Y, \dsp(Y), \gd \dsp(Y)) V \,\mathrm{d}Y \notag\\
&+ \int_{\Omega_0} D_2 j_S (Y, \dsp(Y), \gd \dsp(Y)) \dot{\dsp}\,  \mathrm{d}Y  \notag\\
&+ \int_{\Omega_0} D_3 j_S (Y, \dsp(Y), \gd \dsp(Y)) (\gd \dot{\dsp} - \gd \dsp \gd V) \,\mathrm{d}Y ,
\end{align}
where $\dot{\dsp}$ is the material derivative of $\dsp_t$ at $t=0$, defined by 
\begin{align}
\label{eq:def:matder:w}
\dot{\dsp}  := \frac{\mathrm{d}}{\mathrm{d}t} \Big\vert_{t=0}  (  \dsp^t) 
 = \frac{\mathrm{d}}{\mathrm{d}t} \Big\vert_{t=0}  (  \dsp_t \circ \Phi  _t),
\end{align} 
and $D_1$, $D_2$, $D_3$ stand for the differential on each argument of $\shpFUN_S$. 
In \eqref{eq:shape_deriv_J1}, we have used the relation
%\eqref{eq:def:Phit} of $\Phi_t$
\begin{equation}\label{eq:ddt(det)_divV}
\frac{\mathrm{d}}{\mathrm{d} t} \Big\vert_{t=0}  \det ( \gd \Phi_t ) = \div V
\end{equation}
with the definition \eqref{eq:def:Phit} of $\Phi_t$ (see \eqref{eq:Diff:det} in Appendix \ref{sec:FSI:appendix}).
 \\

Secondly we consider the  shape derivative of $\SHPfun_F$ with respect to $t$. 
We perform a change of variable $x = \eT_{t} \circ \Phi  _t (X)$, in order to rewrite the integrals from $\Omega_{F,t}$ to $\Omega_{0}^c$. 
This gives 
\begin{align}
\label{eq:shader:gnrl:1}
\SHPfun _F(\Omega_{0,t})
& = \int_{\Omega_{0}^c} \Big( j_F ( \eT_{t} \circ \Phi_t (X), \vct_t \circ \eT_{t} \circ \Phi_t (X), (\gd \vct_t ) \circ \eT_{t} \circ \Phi_t (X) ) \notag\\
& \hspace{6.5cm} \det ( \gd ( \eT_{t} \circ \Phi_t (X) ))  \Big) \,\mathrm{d}X .
\end{align}
We calculate the shape derivative of $\SHPfun _F$, setting 
\begin{equation}
\vcT = \vct \circ \eT 
\end{equation}
where $T=T_0=\id_{\RR^{2}} +R\gamma(\dsp)$.
This gives 
\begin{align}
\SHPfun _F' (\Omega_0)
&= \int_{\Omega_0^c} j_F(\eT,\vcT,\gd \vcT (\gd \eT)^{-1})   \tr(\cof(\nabla \eT)^{\tp}\nabla\dot{\eT})  \,\mathrm{d}X  \notag\\
&+ \int_{\Omega_0^c} D_1 j_F(\eT,\vcT,\gd \vcT (\gd \eT)^{-1}) \dot{\eT} \det(\gd \eT)  \,\mathrm{d}X  \notag\\
&+ \int_{\Omega_0^c} D_2 j_F(\eT,\vcT,\gd \vcT (\gd \eT)^{-1}) \dot{\vcT} \det(\gd \eT) \,\mathrm{d}X  \notag\\
&+ \int_{\Omega_0^c} D_3 j_F(\eT,\vcT,\gd \vcT (\gd \eT)^{-1})  \left(\gd\dot{\vcT} - \gd \vcT(\gd \eT)^{-1}\gd\dot{\eT} \right)\cof(\nabla \eT)^{\tp} \,\mathrm{d}X  , 
\label{eq:shader:gnrl:2}
\end{align}
where we denote by 
$\dot{\vcT}$ the material derivative of $\vcT$ defined by
\begin{align}
\label{eq:def:matder:v}
\dot{\vcT} & :=  \frac{\mathrm{d}}{\mathrm{d}t} \Big\vert_{t=0}  (   \vcT^t )
= \frac{\mathrm{d}}{\mathrm{d}t} \Big\vert_{t=0}  ( \vcT_t \circ \Phi  _t) ,
\end{align}
and by  
$\dot{\eT}$ the material derivative of $\eT_t$  defined by 
\begin{align}
\label{eq:def:matder:T}
\dot{\eT}  &= \frac{\mathrm{d}}{\mathrm{d}t} \Big\vert_{t=0} \left( \eT_{t}\circ\Phi  _{t} \right) .
\end{align}
From the definitions of $\eT_t$ in \eqref{eq:def:Tt:new} and of $\dot{\eT}$, we have 
\begin{equation}
\dot{\eT}  = V +  \mathcal{R}\gamma ( \dot{\dsp} ) .
\end{equation}
The term $\tr(\cof(\nabla \eT)^{\tp}\nabla\dot{\eT})$ in \eqref{eq:shader:gnrl:2} comes from the differentiation of $\det ( \gd ( \eT_{t} \circ \Phi_t (X) ))$ in \eqref{eq:shader:gnrl:1}. 
The terms $ \dot{\eT}$ and $\dot{\vcT} $ in \eqref{eq:shader:gnrl:2} are respectively the results of the differentiation through the chain rule of the terms $ \eT_{t} \circ \Phi_t (X) $ and $  \vct_t \circ \eT_{t} \circ \Phi_t (X)$ in  \eqref{eq:shader:gnrl:1}. 
For the last term $(\gd\dot{\vcT} - \gd \vcT(\gd \eT)^{-1}\gd\dot{\eT} )\cof(\nabla \eT)^{\tp}$ in \eqref{eq:shader:gnrl:2} deriving from $(\gd \vct_t ) \circ \eT_{t} \circ \Phi_t (X)$ in \eqref{eq:shader:gnrl:1}, we can write
\begin{align}
(\gd \vct_t ) \circ \eT_{t} \circ \Phi_t (X) 
&= (\gd ( \vct_t  \circ \eT_{t} \circ \Phi_t) )  (X)  (\gd (\eT_t \circ \Phi_t ))^{-1}  (X) , \notag \\
&  = (\gd ( \vcT_t  \circ \Phi_t) )  (X)  (\gd (\eT_t \circ \Phi_t ))^{-1}  (X) ,
\label{eq:shader:gnrl:3}
\end{align}
with $\vcT_t = \vct_t \circ \eT_t$ (see \eqref{def:vt-}). 
From there, we can write in the following proposition the formula of the shape derivative $\SHPfun ' (\Omega_0 )$ of the abstract shape functional $\SHPfun (\Omega_0 ) $
defined by \eqref{eq:def:grnl:shpfun}.
%%
%, and from the definitions of $\dot{\vcT}$ and $\dot{\eT}$ given in \eqref{eq:def:matder:v} and \eqref{eq:def:matder:T}, we find by the differentiation with respect to $t$ of \eqref{eq:shader:gnrl:3}, the last term of \eqref{eq:shader:gnrl:2}. 
%\begin{align}
%\label{eq:def:matder:v}
%\dot{\vcT} & :=  \frac{d}{dt} \Big\vert_{t=0}  (   \vcT^t )
%= \frac{d}{dt} \Big\vert_{t=0}  ( \vcT_t \circ \Phi  _t),  \\
%%
%\label{eq:def:matder:q}
%\dot{\prT} & :=  \frac{d}{dt} \Big\vert_{t=0}  (   \prT^t )
%= \frac{d}{dt} \Big\vert_{t=0}  (  J(\Phi_t)  \,\, \prT_t \circ \Phi  _t)  , 
%\end{align}
%\begin{align}
%\label{eq:def:matder:T}
%\dot{\eT}(X) &= \frac{d}{dt} \Big\vert_{t=0} \left( \eT_{t}\circ\Phi  _{t}(X)\right)
%\quad \forall X \in \had   ,
%\end{align}
%%

\begin{prop}
Let $\SHPfun$ be the shape functional defined by \eqref{eq:def:grnl:shpfun}, where $j_S$ and $j_F$ are differentiable functions. 
Let $V$ be a velocity field belonging to the space $\Theta$ introduced in \eqref{eq:def:THETA}.
Then, the shape derivative of $\SHPfun$ in the direction $V$ computed at $\Omega_0$ is given by 
\begin{align}
\SHPfun ' (\Omega_0 ) 
&= \int_{\Omega_0} j_S (Y, \dsp , \gd \dsp ) \div V \,\mathrm{d}Y %\notag\\
%&
+ \int_{\Omega_0} D_1 j_S (Y, \dsp , \gd \dsp ) V \,\mathrm{d}Y \notag\\
&+ \int_{\Omega_0} D_2 j_S (Y, \dsp , \gd \dsp ) \dot{\dsp} \,\mathrm{d}Y 
%\notag\\
%
%&
+ \int_{\Omega_0} D_3 j_S (Y, \dsp , \gd \dsp ) (\gd \dot{\dsp} - \gd \dsp \gd V) \,\mathrm{d}Y \notag\\
&+ \int_{\Omega_0^c} j_F(\eT,\vcT,\gd \vcT (\gd \eT)^{-1})   \tr(\cof(\nabla \eT)^{\tp}\nabla\dot{\eT})  \,\mathrm{d}X  \notag\\
&+ \int_{\Omega_0^c} D_1 j_F(\eT,\vcT,\gd \vcT (\gd \eT)^{-1}) \dot{\eT} \det(\gd \eT)  \,\mathrm{d}X  \notag\\
&+ \int_{\Omega_0^c} D_2 j_F(\eT,\vcT,\gd \vcT (\gd \eT)^{-1}) \dot{\vcT} \det(\gd \eT) \,\mathrm{d}X  \notag\\
&+ \int_{\Omega_0^c} D_3 j_F(\eT,\vcT,\gd \vcT (\gd \eT)^{-1})  \left(\gd\dot{\vcT} - \gd \vcT(\gd \eT)^{-1}\gd\dot{\eT} \right)\cof(\nabla \eT)^{\tp} \,\mathrm{d}X  .
\label{eq:shpder:final}
\end{align}
\end{prop}

\begin{ex}
\label{sec:FSI:energy:func}
Let $\vct$ and $\dsp$ be the velocity and displacement solutions of the FSI problem associated to $\Omega_0$.
 Let us consider the following energy shape functional 
\begin{equation}
\label{eq:def:energy-type:shpfun}
\SHPfun_E (\Omega_{0}) = 
\frac{1}{2}\int_{\Omega_{F }} \lvert \gs_{x }(\vct ) \rvert^{2} \mathrm{d}x  
+ \frac{1}{2} \int_{\Omega_{0 }}  \lvert \gs_{Y }(\dsp) \rvert^{2} \mathrm{d}Y   ,
\end{equation} 
where $\gs$ is defined in \eqref{eq:def:sym:matrix}, and the norm of a matrix is defined in \eqref{eq:def:matrixnorm}.

The shape derivative of $\SHPfun$ in direction $V$ evaluated at $\Omega_0$ is
\begin{align}
\SHPfun_E  ' (\Omega_{0} ) 
&=   \int_{\Omega_{0}}    \gs \dsp \Mscalp  \left( \nabla \dot{\dsp} - \nabla \dsp \nabla V \right) \, \mathrm{d}Y 
+ \frac{1}{2} \int_{\Omega_{0}}  \lvert \gs \dsp \rvert^{2}\div(V)\, \mathrm{d}Y  \notag\\
&+ \int_{\Omega_{0}^{c}}\left[ \nabla \vcT(\nabla \eT)^{-1}  \right]^s  \Mscalp    \left(\nabla \dot{\vcT} - \nabla \vcT(\nabla \eT)^{-1}\nabla\dot{\eT}  \right)\cof(\nabla \eT)^{\tp}      \mathrm{d}X  \notag\\
& \hspace{2cm} + 
\frac{1}{2} \int_{\Omega_{0}^{c}}
\left\lvert \left[ \nabla \vcT(\nabla \eT)^{-1}  \right]^s  \right\rvert^{2}  \tr(\cof(\nabla \eT)^{\tp}\nabla\dot{\eT})\, \mathrm{d}X .
\end{align}
\end{ex}

Notice that the expression \eqref{eq:shpder:final} of $\SHPfun'$ depends on the material derivatives $\dot{\vcT}$ and $\dot{\dsp}$ of the velocity and of the displacement. 
These material derivatives can be computed as solutions of boundary value problems which depend on the direction $V$ (see \cite[Section 3.4.4]{calisti2021}).  For a practical use of the shape derivative -- within a shape optimization algorithm for example -- it is suitable to find an expression which does not depend on $\dot{\vcT}$ and $\dot{\dsp}$.
For this, we apply in the next section the classical \emph{adjoint method} allowing for a simplified expression of $\SHPfun'$.

%#####################################################################
%#####################################################################
\subsection{Adjoint method, or C\'{e}a's method}
\label{sec:adjoint-state}

The \emph{adjoint method}, or \emph{C\'{e}a's method}, was introduced in \cite{cea1986calcul_rapide}. %used for a formal and useful calculation of the shape derivative of a shape functional.
%This method 
It allows to guess straightforwardly the \emph{adjoint states} we need to introduce in order to simplify the expression of the shape derivative.
%Notably it enables to write this derivative  in such a way that it does not depend on the material derivatives of the solutions anymore. 
After the introduction to this method, we apply it to the FSI problem.
We refer to \cite[Section 6.4.3]{allaire2007conception} for a more detailed presentation of this method.

%C\'{e}a's method, also called \emph{Lagrangian method}.

%*********************************************************************
%*********************************************************************
\subsubsection{Presentation of the method}
\label{sec:adjoint-state:prez}

Let $\Omega_0$ be an admissible domain of the fluid-structure problem, standing for the elastic material (see Figure \ref{fig:contour}). We denote by $\Omega$ any smooth perturbation of $\Omega_0$. 
For example we can consider $\Omega = (\id_{\RR^{2}} + tV) (\Omega_0)$, for $V \in\Theta $ 
where $\Theta$ is defined in \eqref{eq:def:THETA} and for $t>0$. 
 First we define the following functional space: 
\begin{equation}
\mathbf{Q} := ( H^1_0(\Omega^c_{0}) )^{2} \times L^2_0(\Omega^c_{0}) \times ( H^1_{0, \Gamma_{\omega} } (\Omega_{0}) )^{2} \times L^2 (\Omega_{0}).
\end{equation} 
We will denote by $\X_0 = (\vcT , \prT , \dsp , \LagVol)$ the solution of the fluid-structure problem with initial datum $\Omega_0$, and by $\Y = (\vcTDual , \prTDual , \dspDual , \LagVolDual )$ a test quadruplet. 
Let us rewrite the FSI problem with these notations.
Let $\BF :  \mathcal{U}_{\mrm{ad}}\times \RR^{2 \times 2} \times  \mathbf{Q}  \times  \mathbf{Q}  \rightarrow \mathbb{R}$ be a differentiable map, bilinear on $ \mathbf{Q}  \times  \mathbf{Q} $
and $\LF : \mathcal{U}_{\mrm{ad}}\times \RR^{2 \times 2} \times  \mathbf{Q}   \rightarrow \mathbb{R}$ be a differentiable map, linear on $ \mathbf{Q} $, where $\mathcal{U}_{ad}$ is a class of admissible domains for the FSI problem. 
Finally, let $\NL :  \mathbf{Q}  \rightarrow \RR^{2 \times 2}$ be a non linear differentiable map. We consider the solution $\X_{\Omega}$ of the following problem:
\begin{equation}\label{eq:general_form1}
\text{Find } \X_{\Omega}\in \mathbf{Q}  \text{ such that: }
\hspace{1ex}
\BF \big( \Omega , \NL (\X_{\Omega}) , \X_{\Omega}, \Y \big) = 
\LF \big( \Omega , \NL (\X_{\Omega}) , \Y \big) ,
\hspace{1ex}
\forall \Y \in  \mathbf{Q} .
\end{equation}
Now we define a shape functional of the form:  
\begin{equation}
\SHPfun(\Omega) = \shpFUN ( \Omega , \X_{\Omega} ).
\quad 
\end{equation}
This suggests the definition of the following Lagrangian for all $\Omega$ and  $\forall \X , \Y \in\mathbf{Q} $:
\begin{equation}
\label{eq:def:lagrangien}
\mathcal{L} \big( \Omega ,  \X , \Y \big)
=  \shpFUN ( \Omega , \X )
+ \BF \big( \Omega , \NL (\X), \X, \Y \big) 
- \LF \big( \Omega , \NL (\X), \Y \big) .
\end{equation}
By definition we have $\forall \Y \in  \mathbf{Q} $:
\begin{equation}
\label{pbm:adjoint-A3}
\mathcal{L} \big( \Omega ,  \X_\Omega , \Y \big)
=  \shpFUN ( \Omega , \X_\Omega )
\end{equation}
Thus the shape derivative of $\shpFUN$ is 
\begin{align}
\shpFUN'(\Omega_0 , \X_0) 
&=  \frac{\partial \mathcal{L}}{\partial \Omega} \big( \Omega_0 ,  \X_0 , \Y \big)
+ \left\langle \frac{\partial \mathcal{L}}{\partial \X} \big( \Omega_0 ,  \X_0 , \Y \big) , \dot{\X_0} \right\rangle ,
\label{pbm:adjoint-A2}
\end{align}
where $\dot{\X_0} = \partial_\Omega \X_0$ is the material derivative of $\X_0$.

If the following problem admits a solution
%be the solution of the following problem: 
\begin{equation}
\text{Find } \Y_0 \in \mathbf{Q}  \text{ such that: }
\quad
  \left\langle \frac{\partial \mathcal{L}}{\partial \X} \big( \Omega_0 ,  \X_0 , \Y_0 \big) , \Z \right\rangle = 0 
\quad
\forall \Z \in  \mathbf{Q} ,
\label{pbm:adjoint-A}
\end{equation}
then $\Y_0$ is called the \emph{adjoint solution}, or \emph{adjoint state}.

In general the existence of $\Y_0$ has to be proved. However, in the following we do not care about this aspect and we write (formally) the expression of $\SHPfun'$, assuming the well-posedness of \eqref{pbm:adjoint-A}. We finally have
\begin{equation}
\SHPfun'(\Omega_0) = \shpFUN'(\Omega_0 , \X_0)  =  \frac{\partial \mathcal{L}}{\partial \Omega} \big( \Omega_0 ,  \X_0 , \Y_0 \big).
\label{pbm:adjoint-A1}
\end{equation}

We can see that the shape derivative $\SHPfun'(\Omega_0)$ in expression \eqref{pbm:adjoint-A1} does not depend on the material derivative $\dot{\X_0} $, unlike expression \eqref{pbm:adjoint-A2}.
Let us give a slightly more detailed expression of $\SHPfun'(\Omega_0)$.
We develop problem \eqref{pbm:adjoint-A}:
\begin{align}
 \left\langle \frac{\partial \mathcal{L}}{\partial \X} \big( \Omega_0 ,  \X_0 , \Y \big) , \Z \right\rangle
&= 
\left\langle \frac{\partial \shpFUN}{\partial \X} \big( \Omega_0 ,  \X_0  \big) , \Z \right\rangle 
+ \BF \big( \Omega_0 , \NL (\X_0), \Z, \Y \big) \notag \\
&\hspace{-2.6cm}+
\left\langle \frac{\partial \BF}{\partial \NL}  \big( \Omega_0 , \NL (\X_0), \X_0, \Y \big)  \mathfrak{F}' (\X_0) , \Z  \right\rangle 
-
\left\langle  \frac{\partial \LF}{\partial \NL} \big( \Omega_0 , \NL (\X_0), \Y \big) \mathfrak{F}'(\X_0)  , \Z \right\rangle  ,
\end{align}
where we have used the fact that $\BF$ is linear with respect to $\Z$, and we develop  the shape derivative
\begin{align}
\frac{\partial \mathcal{L}}{\partial \Omega} \big( \Omega_0 ,  \X_0 , \Y_0 \big)
&= \frac{\partial \shpFUN}{\partial \Omega} ( \Omega_0 , \X_0 )
+ \frac{\partial \BF}{\partial \Omega} \big( \Omega_0 , \NL (\X_0), \X_0, \Y_0 \big)
- \frac{\partial \LF}{\partial \Omega} \big( \Omega_0 , \NL (\X_0),  \Y_0 \big) .
\end{align}
Finally we can write the shape derivative involving the adjoint state $\Y_0$ as follows:
\begin{align}
\label{eq:shape-der:j-A-L}
\SHPfun'(\Omega_0) 
&= \frac{\partial \shpFUN}{\partial \Omega} ( \Omega_0 , \X_0 )
+ \frac{\partial \BF}{\partial \Omega} \big( \Omega_0 , \NL (\X_0), \X_0, \Y_0 \big)
- \frac{\partial \LF}{\partial \Omega} \big( \Omega_0 , \NL (\X_0),  \Y_0 \big) .
\end{align}
Let us apply this method to the FSI problem.

%~~~~~~~~~~~~~~~~~~~~~~~~~~~~~~~~~~~~~~~~~~~~~~~~~~~~~~~~~~~~~~~~~~~~~
%~~~~~~~~~~~~~~~~~~~~~~~~~~~~~~~~~~~~~~~~~~~~~~~~~~~~~~~~~~~~~~~~~~~~~
\subsubsection{Shape functional and its related Lagrangian}
\label{sec:shp:and:lag}

We consider the shape functional defined by \eqref{eq:def:grnl:shpfun} that we can rewrite as
\begin{align}
\label{def:shape-func:2}
\SHPfun  (\Omega_{0 })
& = \int_{\Omega_{0 }^c} j_F (\eT   , \vcT , \gd \vcT  (\gd \eT  )^{-1}) \det(\gd \eT  ) \dLf   
+ \int_{\Omega_{0 }} j_S ( Y  , \dsp, \gd \dsp  ) \dLs  .
\end{align}
We want to explicitly construct the related Lagrangian of $\SHPfun  (\Omega_{0 })$ as in \eqref{eq:def:lagrangien}. 
Then we will turn to the calculation of its derivatives with respect to $( \vcT , \prT , \dsp , \LagVol )$ as well as with respect to the parameter $t$ which are required for computing the shape derivative of $\SHPfun  $ (see \eqref{pbm:adjoint-A2}). 
Writing $\SHPfun$ on a perturbed domain $\Omega_{0,t}$ leads to
\begin{align}
\label{def:shape-func}
\SHPfun  (\Omega_{0,t})
&= \shpFUN ( t , \vcT_t , \dsp_t ) , \notag \\
& = \int_{\Omega_{0,t}^c} j_F (\eT_{t}  , \vcT_t, \gd \vcT_t (\gd \eT_{t} )^{-1}) \det(\gd \eT_{t} ) \dLft  
+ \int_{\Omega_{0,t}} j_S ( Y_t , \dsp_t, \gd \dsp_t ) \dLst .
\end{align}
where $\eT_{t}$ is defined in \eqref{eq:def:Tt:new}, with $\dsp_t \in ( H^1_{0, \Gamma_{\omega}}  (\Omega_{0,t}) )^2$ the displacement solution of the problem \eqref{pbm:elast-incomp-t:var-form}, and where $\vcT_t \in ( H^1_0(\Omega^c_{0,t}) )^2$ is the velocity solution on the reference domain of the problem \eqref{eq:var:NS:sur-Omega0tc-2}. 

We want to apply Cea's method presented in Section \ref{sec:adjoint-state:prez} to find the adjoint states needed for the calculation of the shape derivative of $\SHPfun$.  
For this, we need to define a Lagrangian functional having independent variables lying in the space $H^1_0(\Omega_{0}^c) \times L^2_0(\Omega_{0}^c) \times  H^1_{0, \Gamma_{\omega}}(\Omega_{0}) \times  L^2(\Omega_{0})$ which is independent of $t$. 
This is done by writing down the FSI formulation \eqref{eq:var:NS:sur-Omega0c2}-\eqref{eq:lag:structure:pbm} into the general form depicted in \eqref{eq:general_form1}.
 \\

The construction of the Lagrangian requires to use the transformation  $\lT^t_{ w }$ defined in \eqref{eq:def:T-t-up-w-down}. 
%with $w=\dsp$, in place of $T_t$ (see \eqref{eq:def:Tt:new}). 
We recall that 
% \begin{equation}
%  \label{eq:def:lT^t}
% \lT^t_{\dsp}\hspace{0cm} = \lT_{t,\dsp}\hspace{0cm}  \circ \Phi  _t
% \end{equation}
% with
% \begin{equation}
%  \label{eq:def:lT^t2}
% \lT_{t,\dsp} = \id_{\RR^2} + \mathcal{R}\gamma(\dsp) \circ \Phi_t^{-1}
% \end{equation}
% and we have
\begin{equation}\label{eq:def_Ttw2}
\lT^t_{ w } = \Phi_t + \mathcal{R}\gamma( w ).
\end{equation}

For $t\geq 0$, for all $( v ,  q )$ and $( \vcTDual, \prTDual )$ in $( H^1_0(\Omega^c_{0}) )^2  \times L^2_0(\Omega^c_{0})$, and for all $(  w  ,  s   )$ and $( \dspDual ,  \LagVolDual )$ in $ ( H^1_{0, \Gamma_{\omega} } (\Omega_{0}) )^2 \times L^2 (\Omega_{0})$, 
as it is suggested by \eqref{eq:def:lagrangien} in the adjoint method introduced in the previous section, we define the Lagrangian by:
\begin{align}
 \mathcal{L}  (t, ( v , q  , w  , s  ), ( \vcTDual, \prTDual , \dspDual , \LagVolDual ) ) & =
\SHPfun^t  (\Omega_{0};  v , w )  \notag\\
& + a_{F}^t ( w  ;  v  , \vcTDual ) + b_{F}^t ( w  ; \vcTDual,  q  ) - f_{F}^t ( w  ; \vcTDual    ) +
 b_{F}^t ( w ;  v   , \prTDual )  \notag \\ 
& +  a_{S}^t ( w  , \dspDual) + b_{S}^t ( \dspDual,  s  ) - f_{S}^t ( w  ; v  ; q  ;  \dspDual ) +  b_{S}^t ( w  , \LagVolDual )  ,  
\label{eq:def:LAG}
\end{align}
where 
\begin{equation}
\SHPfun^t  (\Omega_{0};  v , w ) :=
\int_{\Omega_{0}^c} j_F (  \lT^t_{ w }\hspace{0cm}  , v , \gd  v  \gd (  \lT^t_{ w }\hspace{0cm} )^{-1}) J(  \lT^t_{ w }) 
+ \int_{\Omega_{0}} j_S ( \Phi  _t ,  w  , \gd w  \nabla \Phi_t ^{-1}  ) J(\Phi_t)  .
\end{equation}
In the expression \eqref{eq:def:LAG} of $\mathcal{L}$, the functionals $a_{F}^t$, $b_{F}^t$, $f_{F}^t$ are defined in \eqref{eq:def:aF^t}, \eqref{eq:def:bF^t}, \eqref{eq:def:fF^t}, and the functionals $a_{S}^t$, $b_{S}^t$, $f_{S}^t$ are defined in  \eqref{eq:def:aS^t}, \eqref{eq:def:bS^t}, \eqref{eq:def:fS^t}
in which the transform 
%$T_{\dsp}^t=\eT_{t,\dsp}\circ \Phi  _t$ 
$\eT_{ w }^t$ is used in place of $T_t\circ \Phi_t$ that is, e.g. 
$$
a_{F}^t( w ; v  , \vcTDual)  :=
%
%\int_{\Omega_{0}^c} A_{F}^t(\dsp^t ; \vcT , \vcTDual) & :=
% 
\nu\int_{\Omega_{0}^c} ( \gd  v   ) F( \eT_{ w }^t ) \Mscalp  \gd \vcTDual.
$$

The expression of the Lagrangian is then given by:
\begin{align}
\mathcal{L}  (t, ( v , q &, w , s  ), ( \vcTDual, \prTDual , \dspDual , \LagVolDual ) ) =   \notag \\
& \hspace{3ex} \int_{\Omega_{0}^c} j_F (  \lT^t_{ w }\hspace{0cm}  ,  v , \gd  v \gd (  \lT^t_{ w }\hspace{0cm} )^{-1}) J(  \lT^t_{ w }\hspace{0cm} ) 
+ \int_{\Omega_{0}} j_S ( \Phi  _t ,  w  , \gd  w  \nabla \Phi_t ^{-1}  ) J(\Phi_t)   \notag \\
&+ \int_{\Omega_{0}^c} \Big( \nu (\gd v ) F(  \lT^t_{ w }\hspace{0cm} ) \Mscalp  \gd  \vcTDual 
-   q  ( 
%
% J(\Phi_t) ^{-1} 
%
G(  \lT^t_{ w }\hspace{0cm} ) \Mscalp  \gd  \vcTDual )  
\Big) \notag \\
& \hspace{1cm} - \int_{\Omega_{0}^c}  \Big( \prTDual  ( 
%
%  J(\Phi_t) ^{-1} 
%
G(  \lT^t_{ w }\hspace{0cm} ) \Mscalp  \nabla  v )   
+  ( f {\circ   \lT^t_{ w }\hspace{0cm} }\cdot  \vcTDual )  J(  \lT^t_{ w }\hspace{0cm} )  \Big)   \notag \\
&+ \int_{\Omega_{0}}  \Big(
\mu (\gd w ) F ( \Phi_t ) \Mscalp  (\gd  \dspDual ) 
-    s   G ( \Phi_t ) \Mscalp 
 \gd  \dspDual 
-   ( (g \circ \Phi_t) \cdot  \dspDual  )  J(\Phi_t)   
\Big)
\notag\\
&\hspace{1cm} 
- \int_{\Gamma_{0}}  \dspDual   \cdot(\nu(\gd  v ) F(  \lT^t_{ w }\hspace{0cm} ) 
-  q   G(  \lT^t_{ w }\hspace{0cm} ) ) n_{0}
- \int_{\Omega_{0} }   \LagVolDual  G ( \Phi_t ) \Mscalp 
 \gd  w   .
\label{eq:der:lag:2}
\end{align}  

\bigskip
Let $( \vcT^t , \prT^t , \dsp^t , \LagVol ^t )$ defined in 
\eqref{eq:def:vt2}, \eqref{eq:def:pt2}, \eqref{eq:def:wt_phi}, and \eqref{eq:def:st_phi}, be the transported solutions of the coupling Stokes problem \eqref{eq:var:NS:sur-Omega0c2} and incompressible elasticity problem \eqref{eq:lag:structure:pbm}.
We start remarking that 
$\lT^{t}_{ \dsp^t } = T_t \circ \Phi_t$ (see \eqref{eq:def:Tt:new}) and as a result the following property holds 
\begin{equation}
\SHPfun^t  (\Omega_{0}; \vcT^t, \dsp^t) = \SHPfun(\Omega_{0,t}),
\end{equation}
where $\SHPfun(\Omega_{0,t})$ is given by \eqref{def:shape-func}.
Such as for \eqref{pbm:adjoint-A3} in the introductory paragraph \ref{sec:adjoint-state:prez}, we have that for all $( \vcTDual, \prTDual ,\dspDual ,  \LagVolDual )$ in $ ( H^1_0(\Omega^c_{0}) )^2 \times L^2_0(\Omega^c_{0}) \times  ( H^1_{0, \Gamma_{\omega} } (\Omega_{0}) )^2 \times L^2 (\Omega_{0})$:
\begin{align}
\mathcal{L} (t, (\vcT^t, \prT ^t , \dsp^t,  \LagVol ^t), ( \vcTDual, \prTDual , \dspDual , \LagVolDual )) 
= \SHPfun (\Omega_{0,t}) .
\end{align}  

Now we can compute the partial derivatives of the Lagrangian in order to find the suitable adjoint states allowing to simplify the expression \eqref{eq:shpder:final} of the shape derivative.

%~~~~~~~~~~~~~~~~~~~~~~~~~~~~~~~~~~~~~~~~~~~~~~~~~~~~~~~~~~~~~~~~~~~~~
%~~~~~~~~~~~~~~~~~~~~~~~~~~~~~~~~~~~~~~~~~~~~~~~~~~~~~~~~~~~~~~~~~~~~~
\subsubsection{Derivatives of the Lagrangian}
\label{sec:lag:der}

In order to obtain the adjoint problems, we need to derive the Lagrangian $\mathcal{L}$ with respect to the variables $ v $,  $ q $, $ w $, and $ s $. 
%Let $( \vcT^t , \prT^t , \dsp^t , \LagVol ^t )$ defined in  \eqref{eq:def:vt2}, \eqref{eq:def:pt2}, \eqref{eq:def:wt_phi}, and \eqref{eq:def:st_phi}, be the transported solutions of the FSI problem~\eqref{pbm:FSI:t:domaine-fixe}.
The derivatives of $\mathcal{L}$ are evaluated at $t\in\mathbb{R}_+$, 
$( v   ,  q   ), ( \vcTDual, \prTDual ) \in ( H^1_0(\Omega^c_{0}) )^2 \times L^2_0(\Omega^c_{0})$ and  $( w   ,  s   ), ( \dspDual,  \LagVolDual )  \in ( H^1_{0, \Gamma_{\omega} } (\Omega_{0}) )^2 \times L^2 (\Omega_{0})$. 
%$( v =\vcT^t ,  q =\prT^t ), ( \vcTDual, \prTDual ) \in ( H^1_0(\Omega^c_{0}) )^2 \times L^2_0(\Omega^c_{0})$ and  $( w =\dsp^t ,  s =\LagVol^t ), ( \dspDual,  \LagVolDual )  \in ( H^1_{0, \Gamma_{\omega} } (\Omega_{0}) )^2 \times L^2 (\Omega_{0})$. 
For the sake of readability, the dependence on these variables is not explicitly stated in the calculation of the derivatives.

We first derive the Lagrangian with respect to the variables $ q $ and $ s $. Let $d\in L^2_{0}(\Omega_0^c)$ and $e \in L^2(\Omega_0)$. We have
\begin{align}
\langle \frac{\partial \mathcal{L} }{\partial  q }
%
%(t, (\vcT, \prT ), ( \vcTDual, \prTDual ), (\dsp, \LagVol ), ( \dspDual , \LagVolDual )) 
%
, d \rangle
&= - \int_{\Omega_{0}^c} d (
%
% J(\Phi_t) ^{-1}
%
 G(  \lT^t_{ w }\hspace{0cm} ) \Mscalp \nabla  \vcTDual)  
+ \int_{\Gamma_0}  \dspDual  \cdot d ( 
%
% J(\Phi_t) ^{-1} 
%
G(  \lT^t_{ w }\hspace{0cm} ) ) n_0  , \label{pbm:dual:pression} \\
\langle \frac{\partial \mathcal{L} }{\partial  s }
%
%(t, (\vcT, \prT ), ( \vcTDual, \prTDual ), (\dsp, \LagVol ), ( \dspDual , \LagVolDual )) 
%
, e \rangle
&= -  \int_{\Omega_{0} }   e G( \Phi_t ) \Mscalp  \gd  \dspDual  .
\label{pbm:dual:lag-mult}
\end{align}
For the derivative of the Lagrangian with respect to the variable $ v $ and $ w $, 
% we shall simply write $D_\alpha j_F$ instead of $D_\alpha j_F (  \lT^t_{\dsp}\hspace{0cm}  , \vcT, \gd \vcT \gd (  \lT^t_{\dsp}\hspace{0cm} )^{-1})$ and $D_\alpha j_S$ instead of $D_\alpha j_S ( \Phi  _t , \dsp, \gd \dsp (\gd \Phi_t)^{-1} )$, for $\alpha = 1,2,3$.
we shall simply write $D_\alpha j_F$ and $D_\alpha j_S$ instead of $D_\alpha j_F (  \lT^t_{ w }\hspace{0cm}  , v  , \gd v  \gd (  \lT^t_{ w }\hspace{0cm} )^{-1})$ and $D_\alpha j_S ( \Phi  _t ,  w  , \gd w (\gd \Phi_t)^{-1} )$ respectively, for $\alpha = 1,2,3$.
Let  $h \in H^{1}_0(\Omega_0^c)$ and $k \in H^{1}_{0, \Gamma_{\omega} }(\Omega_0)$. We have
\begin{align}
\langle \frac{\partial \mathcal{L} }{\partial  v } , h \rangle &=  
\int_{\Omega_{0}^c}  (  (D_2 j_F) h
+ (D_3 j_F)  \gd h \gd (   \lT^t_{ w }\hspace{0cm} )^{-1} )  J(  \lT^t_{ w }\hspace{0cm} )
  \notag \\
& + \int_{\Omega_{0}^c}  
\Big(
\nu(\gd h ) F(  \lT^t_{ w }\hspace{0cm} ) \Mscalp  \gd  \vcTDual  
-  \prTDual  
 G(  \lT^t_{ w }\hspace{0cm} ) \Mscalp \nabla h  \Big)
%\notag \\
%&
- \int_{\Gamma_{0}}  \dspDual   \cdot \nu(\gd h ) F(  \lT^t_{ w }\hspace{0cm} )   n_{0}   ,
\label{pbm:dual:velocity}
\end{align}
and
\begin{align}
\langle \frac{\partial \mathcal{L} }{\partial w } , k \rangle & =  
 \int_{\Omega_{0}^c} \Big( (j_F)  D_{ w }   J(  \lT^t_{ w }\hspace{0cm} )    k    
+ [ (D_1 j_F) D_{ w } (  \lT^t_{ w }\hspace{0cm}  )k  
+ (D_3 j_F) \gd  v    D_{ w } (\gd (  \lT^t_{ w }\hspace{0cm}  )^{-1})  k ] J(  \lT^t_{ w }\hspace{0cm} )  \Big)  \notag \\
&+ \int_{\Omega_{0}} \Big( (D_2 j_S)k  J(\Phi_t)  + (D_3 j_S)\gd k  \nabla \Phi_t ^{-1}  J(\Phi_t)  \Big)  \notag \\
&+ \int_{\Omega_{0}^c} \Big(  [ \nu \gd v   D_{ w } F(  \lT^t_{ w }\hspace{0cm} )   k  
-    q 
%
%   J(\Phi_t) ^{-1} 
% 
D_{ w } G(  \lT^t_{ w }\hspace{0cm} )   k ] \Mscalp  \gd  \vcTDual  
- ( D_{ w } G(   \lT^t_{ w }\hspace{0cm}  ) k  \Mscalp  \nabla  v )  \prTDual  \Big) 
 \notag\\
&\hspace{0.2cm} 
-\int_{\Omega_{0}^c} \Big(  ( D_{ w } (f \circ\lT^t_{ w } )   k \cdot  \vcTDual ) J(  \lT^t_{ w }\hspace{0cm} )   
+ ( f {\circ   \lT^t_{ w }\hspace{0cm} } \cdot  \vcTDual ) D_{ w } J(  \lT^t_{ w }\hspace{0cm} ) k  \Big) 
 \notag \\
&+  \int_{\Omega_{0}}  \mu  (\gd k) F( \Phi_t) \Mscalp   \gd  \dspDual 
- \int_{\Omega_{0} }     \LagVolDual G ( \Phi_t) \Mscalp  \gd k 
  \notag\\
&\hspace{0.2cm} 
- \int_{\Gamma_{0}}  \dspDual   \cdot \left( \nu \gd \vcT  D_{ w } F(   \lT^t_{ w }\hspace{0cm}  )  k  -  \prT  
%
% J(\Phi_t) ^{-1}
%
 D_{ w } G(   \lT^t_{ w }\hspace{0cm}  )  k \right)   n_{0} ,
\label{pbm:dual:deplacement}
\end{align}
where the derivatives $D_{ w }(\cdot)$ with respect to the variable $ w $ are given in Appendix \ref{sec:FSI:appendix} expressions \eqref{eq:DJbar}, \eqref{eq:DGbar}, and \eqref{eq:DFbar}. \\

Finally we calculate partial derivative of the Lagrangian with respect to the variable $t$, referring once again to Appendix \ref{sec:FSI:appendix} for the expressions \eqref{eq:DJ2Vt}, \eqref{eq:DG2Vt}, and \eqref{eq:DF2Vt} of the time derivatives $D_t(\cdot)$ of $  J(   \lT^t_{w}  )$, $ G(   \lT^t_{w}  )$, and $  F(   \lT^t_{w}  )$.  
With the use of \eqref{eq:ddt(det)_divV}, we obtain
\begin{align}
\frac{\partial \mathcal{L} }{\partial t} & = 
%1
\int_{\Omega_{0}^c}  \Big(
(j_F) D_t J(   \lT^t_{ w }\hspace{0cm}  )   
+  (D_1 j_F) D_t(   \lT^t_{ w }\hspace{0cm} )  J(   \lT^t_{ w }\hspace{0cm}  ) 
+  (D_3 j_F) \gd  v  D_t (\gd (   \lT^t_{ w }\hspace{0cm} )^{-1})      
J(   \lT^t_{ w }\hspace{0cm}  )  
\Big) \notag \\
%2
&+ \int_{\Omega_{0}} \Big(  
(j_S) \div V  + (D_1 j_S) V   J(\Phi_t)  + (D_3 j_S) \gd  w  D_t  \nabla \Phi_t ^{-1}  J(\Phi_t)
\Big)  \notag \\
%3
&+ \int_{\Omega_{0}^c} \Big( 
[  \nu (\gd v )  D_t ( F(   \lT^t_{ w }\hspace{0cm}  ))
-   q   D_t ( G(   \lT^t_{ w }\hspace{0cm}  ) )  ] \Mscalp  \gd  \vcTDual 
- \prTDual  D_t ( G(   \lT^t_{ w }\hspace{0cm}  ) )  \Mscalp \nabla v 
\Big)  \notag\\
%4
&\hspace{1cm} 
- \int_{\Omega_{0}^c} \Big( 
( f {\circ   \lT^t_{ w }\hspace{0cm} } \cdot  \vcTDual ) D_t J(   \lT^t_{ w }\hspace{0cm}  )
+ ( D_t (f {\circ   \lT^t_{ w }\hspace{0cm} })  \cdot  \vcTDual ) J(   \lT^t_{ w }\hspace{0cm}  ) 
\Big)  \notag \\
& + \int_{\Omega_{0}} \Big( 
[ \mu  (\gd w  ) D_t F( \Phi_t )  - \LagVol  D_t G( \Phi_t ) ] \Mscalp    \gd  \dspDual 
-   \LagVolDual  D_t G( \Phi_t )  \Mscalp  \gd  w   
\Big) \notag\\
&\hspace{1cm} 
- \int_{\Gamma_{0}}  \dspDual   \cdot(\nu \gd  v     D_t F(   \lT^t_{\dsp}\hspace{0cm}  ) -  q   D_t (
%
% J(\Phi_t) ^{-1} 
%
G(   \lT^t_{ w }\hspace{0cm} ) ) ) n_{0} , 
\label{eq:Lag:der:t}
\end{align}
where we recall that 
%$\lT^t_{\dsp} = \lT_{t,\dsp}\hspace{0cm}  \circ \Phi  _t = \Phi  _t + \mathcal{R}( \gamma (\dsp  ) )$ (see \eqref{eq:def:lT} ), 
$\lT^t_{ w }  = \Phi  _t + \mathcal{R}( \gamma ( w  ) )$ (see  \eqref{eq:def:T-t-up-w-down}).
This formula can be simplified by noticing that finally, $D_t \lT^t_{ w } = V_t := D_t \Phi_t $.

%*********************************************************************
%*********************************************************************
\subsubsection{Definition of the adjoint states}
\label{sec:FSI:def-adj}

%Formally , we have (see \eqref{pbm:adjoint-A2})
%\begin{equation}
%\SHPfun  '(\Omega_0) 
%=
%\begin{pmatrix}
%\partial_t   \\
%\partial_\vcT   \\
%\partial_{ \prT }   \\
%\partial_\dsp   \\
%\partial_{ \LagVol }  
%\end{pmatrix}
%\mathcal{L} 
%(0, (\vcT^{0}, \prT ^{0} , \dsp^{0},  \LagVol ^{0} ), ( \vcTDual, \prTDual , \dspDual , \LagVolDual ))
%\cdot
%\begin{pmatrix}
%1 \\
%\dot{\vcT} \\
%\dot{ \prT }  \\
%\dot{\dsp} \\
%\dot{ \LagVol }
%\end{pmatrix}
%\end{equation}
Let us write the adjoint equations. For this, the partial derivatives of the Lagrangian calculated in the previous section are evaluated at 
$(t,\vcT, \prT, \dsp,  \LagVol ) = (0,\vcT^0, \prT^0, \dsp^0,  \LagVol^0 ) $
where $(\vcT^{0}, \prT ^{0})$ is the solution of Problem \eqref{eq:var:NS:sur-Omega0tc-2} written at $t=0$ and
$(\dsp^{0},  \LagVol ^{0})$ is the solution of Problem \eqref{pbm:elast-incomp-t:var-form} written at $t=0$. 
Since for $t=0$ we have $T_{\dsp}^0=\id_{\RR^2} + \R\gamma(\dsp) = T$, 
we obtain from equations \eqref{pbm:dual:pression}, \eqref{pbm:dual:lag-mult}, \eqref{pbm:dual:velocity}, and \eqref{pbm:dual:deplacement} that for all $( \vcTDual, \prTDual ),(h,d)\in ( H^1_0(\Omega^c_{0}) )^2 \times L^2_0(\Omega^c_{0})$, and for all $( \dspDual , \LagVolDual ),(k, e) \in ( H^1_{0,\Gamma_{\omega} } (\Omega_{0}) )^2 \times L^2 (\Omega_{0})$:
\begin{align}
&\langle \frac{\partial \mathcal{L} }{\partial  \prT }( \vcTDual, \prTDual , \dspDual  , \LagVolDual ) , d \rangle
= - \int_{\Omega_{0}^c} d  (G(\eT) \Mscalp \nabla  \vcTDual) 
+ \int_{\Gamma_0}  \dspDual  \cdot d \,  G(\eT)   n_0 , 
\label{eq:adjoint-state-q} \\
&\langle \frac{\partial \mathcal{L} }{\partial  \LagVol }( \vcTDual, \prTDual , \dspDual , \LagVolDual ) , e \rangle
= -  \int_{\Omega_{0} }    e \div  \dspDual  , 
\label{eq:adjoint-state-s} \\
&\langle \frac{\partial \mathcal{L} }{\partial \vcT}(( \vcTDual, \prTDual ), ( \dspDual , \LagVolDual )) , h \rangle 
=  
\int_{\Omega_{0}^c} \Big( (D_2 j_F)  h  J(\eT) 
+ (D_3 j_F)  \gd h \gd (\eT)^{-1} J(\eT) \Big)  \notag \\
& \hspace{2ex}
+ \int_{\Omega_{0}^c} \Big(  \nu (\gd h ) F(\eT) \Mscalp   \gd  \vcTDual 
- \prTDual   ( G(\eT) \Mscalp  \nabla h)  \Big)
    % \\
%&
- \int_{\Gamma_{0}}  \dspDual   \cdot(\nu(\gd h )   F(\eT)  )  n_{0} , 
\label{eq:adjoint-state-v}\\
&\langle \frac{\partial \mathcal{L} }{\partial \dsp}(( \vcTDual, \prTDual ),   ( \dspDual , \LagVolDual )) , k \rangle 
=   
 \int_{\Omega_{0}} \Big(  (D_2 j_s)k   + (D_3 j_s)\gd k  \Big) \notag \\
& \hspace{2ex}
+ \int_{\Omega_{0}^c} \Big( (j_F) D_{ w }  J (\eT)   k   
+ \big[  (D_1 j_F) (D_{ w } (\eT)   k)  
+ (D_3 j_F) \gd \vcT  D_{ w } (\gd (\eT)^{-1}) k \big]  J(T)  \Big)  \notag \\
%
%
%-  (D_3 j_F) \gd \vcT   \gd \eT^{-1} \gd (D_{ w } (\eT)   k)  \gd \eT^{-1} J   \notag \\
%
%
& \hspace{2ex}
+\int_{\Omega_{0}^c} \Big( [ \nu (\gd \vcT)  (D_{ w }  F(\eT)   k )  
-      q ( (D_{ w }  G(\eT)   k ) ] \Mscalp  \gd  \vcTDual   
- \prTDual (D_{ w }  G(\eT)   k  \Mscalp  \nabla \vcT)  \Big)
  \notag\\
& \hspace{2ex}
-\int_{\Omega_{0}^c} \Big(  ( \gd f ) {\circ \eT} D_{ w } (\eT)   k   \cdot  \vcTDual  J(T)
+ ( f {\circ \eT} \cdot  \vcTDual ) D_{ w }  J(\eT)   k   \Big)   \notag \\
&\hspace{2ex}
+  \int_{\Omega_{0}} \Big(  \mu   \gd k    \Mscalp   \gd  \dspDual 
-   \LagVolDual \div  k \Big)
- \int_{\Gamma_{0}}  \dspDual   \cdot(\nu(\gd \vcT )  D_{ w }  F(\eT)   k   - q   D_{ w }  G(\eT)  k )   n_{0}.
\label{eq:adjoint-state-w}
\end{align}
Once again we have written $D_\alpha j_F$ instead of $D_\alpha j_F (\eT  , \vcT, \gd \vcT \gd (\eT )^{-1})$ and $D_\alpha j_S$ instead of $D_\alpha j_S ( Y , \dsp, \gd \dsp )$, for $\alpha = 1,2,3$. 

With these expressions, we can write in the following proposition the problem satisfied by the adjoint states associated to the shape functional $\SHPfun $ defined in \eqref{def:shape-func:2} and to the FSI problem \eqref{eq:complete-syst-stokes-stokes-transported}.

\begin{prop}
\label{thm:adj:state}
Let $\mathcal{L}$ be the Lagrangian defined in \eqref{eq:der:lag:2}, associated to the shape functional $\SHPfun $ defined in \eqref{def:shape-func:2} and to the FSI problem \eqref{eq:complete-syst-stokes-stokes-transported}. 
The related adjoint state $( \vcTDual, \prTDual , \dspDual , \LagVolDual ) \in ( H^1_0(\Omega_{0}^c) )^2 \times L^2_0(\Omega_{0}^c) \times ( H^1_{0,\Gamma_{\omega}}(\Omega_{0}) )^2 \times L^2(\Omega_{0})$ defined by \eqref{pbm:adjoint-A}
satisfies formally the following equation:
\begin{equation}
\left\lbrace
\begin{aligned}
% & \text{Find } ( \vcTDual, \prTDual , \dspDual , \LagVolDual ) \text{ in } ( H^1_0(\Omega_{0}^c) )^2 \times L^2_0(\Omega_{0}^c) \times ( H^1_{0,\Gamma_{\omega}}(\Omega_{0}) )^2 \times L^2(\Omega_{0})  \text{ such that: } \\
%
&\langle \frac{\partial \mathcal{L} }{\partial  \prT }(( \vcTDual, \prTDual ), ( \dspDual , \LagVolDual )) , d \rangle 
+\langle \frac{\partial \mathcal{L} }{\partial  \LagVol }(( \vcTDual, \prTDual ), ( \dspDual , \LagVolDual )) , e \rangle \\
&\hspace{15ex}
+\langle \frac{\partial \mathcal{L} }{\partial \vcT}(( \vcTDual, \prTDual ), ( \dspDual , \LagVolDual )) , h \rangle 
+
\langle \frac{\partial \mathcal{L} }{\partial \dsp}(( \vcTDual, \prTDual ),   ( \dspDual , \LagVolDual )) , k \rangle 
= 0 , \\
& \forall ( h, d ,k , e) \in ( H^1_0(\Omega_{0}^c) )^2 \times L^2_0(\Omega_{0}^c) \times ( H^1_{0,\Gamma_{\omega}}(\Omega_{0}) )^2 \times  L^2(\Omega_{0})  , 
\end{aligned}
\right.
\label{eq:pbm:FSI:adj} 
\end{equation}
given by expressions \eqref{eq:adjoint-state-q}, \eqref{eq:adjoint-state-s}, \eqref{eq:adjoint-state-v}, and \eqref{eq:adjoint-state-w}.
\end{prop}

We emphasise that the method we have just presented in Section \ref{sec:adjoint-state} is a formal method, since we have not shown that problem \eqref{pbm:adjoint-A} is well-posed. 

% From there, in the next section we shall simplify the formula \eqref{eq:shpder:final} of the shape derivative $\SHPfun ' (\Omega_0 ) $, using the adjoint states defined in Proposition \ref{thm:adj:state}.

%*********************************************************************
%*********************************************************************
\subsection{Simplified formula for the shape derivative $\SHPfun ' (\Omega_0 ) $}
\label{sec:FSI:ShpDer:simplified}

% \grn{\it J'enlèverais tout ce paragraphe\\
% $\Big[$}
%The method we have just presented in Section \ref{sec:adjoint-state} is a formal method.  
% Indeed, we have defined a Lagrangian and carried out several derivative calculations without any justification. 
% \grn{\sout{To be rigorous, we should show the differentiability with respect to the variable $t$ of the solutions of the fluid-structure interaction problem, that is to say their shape differentiability,}} and justify that the adjoint problems defined in Section \ref{sec:FSI:def-adj} are well defined (which should be shown in the same way as in Section \ref{sec:FSI:resolution} for the Fluid-Structure Interaction problem).
% These justifications are part of the ongoing work and the perspectives for the continuation of this thesis.
% \\
% \grn{\it $\Big]$}
We can simplify the formula of the shape derivative $\SHPfun ' (\Omega_0 ) $ given by \eqref{eq:shpder:final} obtained in Section \ref{sec:shader:costfunc}. 
For this, we follow what is done in Section \ref{sec:adjoint-state:prez}, using the formula \eqref{pbm:adjoint-A1}, that is
\begin{equation}
\SHPfun ' ( \Omega_0 ) = \frac{\partial \mathcal{L} }{\partial t} ( 0 , ( \vcT, \prT , \dsp , \LagVol ) , ( \vcTDual, \prTDual , \dspDual , \LagVolDual ) )
\end{equation}
with the formula obtained in \eqref{eq:Lag:der:t}, where $( \vcT, \prT , \dsp , \LagVol )$ is the solution of the FSI problem \eqref{eq:complete-syst-stokes-stokes-transported} and $( \vcTDual, \prTDual , \dspDual , \LagVolDual )$ is the solution of the adjoint problem \eqref{eq:pbm:FSI:adj}. 
This leads to the following theorem.

\begin{thm}
\label{thm:ShapeDer:simplified}
Let $\SHPfun (\Omega_0) $ be the shape functional defined by \eqref{def:shape-func:2}.    
Let $( \vcT, \prT , \dsp , \LagVol ) \in ( H^1_0(\Omega_{0}^c) )^2 \times L^2_0(\Omega_{0}^c) \times ( H^1_{0,\Gamma_{\omega}}(\Omega_{0}) )^2 \times L^2(\Omega_{0})$ be the solution of the FSI problem \eqref{eq:complete-syst-stokes-stokes-transported}, and $( \vcTDual, \prTDual , \dspDual , \LagVolDual ) \in ( H^1_0(\Omega_{0}^c) )^2 \times L^2_0(\Omega_{0}^c) \times ( H^1_{0,\Gamma_{\omega}}(\Omega_{0}) )^2 \times L^2(\Omega_{0})$ be the adjoint states solution of the adjoint problem \eqref{eq:pbm:FSI:adj}. 
Then the shape derivative of $\SHPfun (\Omega_0)$ can be written as follows:
\begin{align}
\SHPfun ' (\Omega_0) & = 
%1
\int_{\Omega_{0}^c}  j_F (\eT,\vcT,\gd \vcT (\gd \eT)^{-1})  D J(V)   
+  D_1 j_F (\eT,\vcT,\gd \vcT (\gd \eT)^{-1}) V  J( \lT  )   \notag\\
&+ \int_{\Omega_{0}^c}  D_3 j_F (\eT,\vcT,\gd \vcT (\gd \eT)^{-1}) \gd \vcT (- \gd T ^{-1} \gd V  \gd T ^{-1}       )
J(  \lT  )\notag \\
%2
%
%
&+ \int_{\Omega_{0}} \Big(  j_S (Y, \dsp , \gd \dsp ) \div V + D_1 j_S (Y, \dsp , \gd \dsp ) V  + D_3 j_S( Y, \dsp , \gd \dsp  ) \gd \dsp (-\gd V)   \Big) , \notag \\
& + \mathcal{A} ' ( (\vcT , \prT , \dsp , \LagVol) , (\vcTDual , \prTDual , \dspDual , \LagVolDual) , V ) ,
\label{eq:SHAPE-DER-SIMPLIF}
\end{align}
where $\mathcal{A} '$ is given by
\begin{align}
\mathcal{A} ' ( (\vcT , \prT , \dsp , \LagVol) , (\vcTDual , \prTDual , \dspDual , \LagVolDual) , V )
&:= \int_{\Omega_{0}^c} \Big( [  \nu \gd \vcT  D  F( V )
-   \prT    DG(V)   ] \Mscalp  \gd  \vcTDual 
- \prTDual DG(V)  \Mscalp \nabla \vcT \Big)  
  \notag\\
%4
%
%
&\hspace{1cm} 
- \int_{\Omega_{0}^c} \Big( ( f {\circ   \lT  } \cdot  \vcTDual ) D J(V)
+ ( D_t (f {\circ   \lT  })  \cdot  \vcTDual ) J( \lT ) \Big) 
  \notag \\
%5
%
%
& + \int_{\Omega_{0}} \Big( [ \mu  (\gd \dsp ) D F(V)  - \LagVol  D G(V) ] \Mscalp    \gd  \dspDual 
-   \LagVolDual  D G(V)  \Mscalp  \gd \dsp \Big) \notag\\
&\hspace{1cm} 
- \int_{\Gamma_{0}}  \dspDual   \cdot (\nu \gd \vcT    D  F( V  ) 
-  \prT   DG(V)   ) n_{0} , 
\label{eq:SHAPE-DER-SIMPLIF:2}
\end{align}
and where $\eT := \eT_0$ is given by \eqref{eq:def:Tt:new}, $V$ is the velocity of the transformation $\Phi_t$ given by \eqref{eq:def:Phit}, and $DJ(V)$, $DG(V)$, and $DFJ(V)$ are given by 
\begin{align}
\label{eq:DJ2V} 
% D_{t=0} J(   \lT^t_{w}  ) = 
DJ (V) &= 
\tr ( \cof(\gd \eT)^\tp   \gd V  )
 , \\
\label{eq:DG2V}
% D_{t=0} G(   \lT^t_{w}  ) = 
DG (V) &=
\cof(\gd \eT)  
\left[
\tr \left( (\gd \eT)^{-1} \gd V \right) \mathrm{I} 
 - [ (\gd \eT)^{-1}  \gd V   ]^\tp
\right]  , \\
\label{eq:DF2V}
% D_{t=0} F(   \lT^t_{w}  ) = 
DF (V) &= 
\cof(\gd \eT)^{\tp}
\left[   \tr \left( (\gd \eT)^{-1} \gd V \right) \mathrm{I}   
 - 2[ \gd V (\gd \eT)^{-1}]^s  \right]  (\gd \eT)^{-\tp} , 
\end{align}
and denote the time derivatives of $  J(   \lT^t_{w}  )$, $ G(   \lT^t_{w}  )$, and $  F(   \lT^t_{w}  )$ computed in \eqref{eq:DJ2Vt}, \eqref{eq:DG2Vt}, and \eqref{eq:DF2Vt}, and evaluated at $t=0$ and $w = \dsp$.
\end{thm}

\begin{ex}
% We can observe that the expression of the shape derivative $\SHPfun ' (\Omega_0)$ given in Theorem~\ref{thm:ShapeDer:simplified} is related to the abstract one given in \eqref{eq:shape-der:j-A-L}, where the four last lines of expression \eqref{eq:SHAPE-DER-SIMPLIF} correspond to the term
% \begin{align} 
% \frac{\partial \BF}{\partial \Omega} \big( \Omega_0 , \NL (\X_0), \X_0, \Y_0 \big)
% - \frac{\partial \LF}{\partial \Omega} \big( \Omega_0 , \NL (\X_0),  \Y_0 \big),
% \end{align}
%in \eqref{eq:shape-der:j-A-L}. 
In the case of the energy-type shape functional $\SHPfun_E$ given by \eqref{eq:def:energy-type:shpfun}, denoting by 
% \begin{align}
% &  \partial_t a_{F}^0 ( \dsp  ;  \vcT  , \vcTDual ) + \partial_t b_{F}^0 ( \dsp  ; \vcTDual,  \prT  ) - \partial_t f_{F}^0 ( \dsp  ; \vcTDual    ) +
%  \partial_t b_{F}^0 ( \dsp ;  \vcT   , \prTDual )  \notag \\ 
% %
% + &  \partial_t a_{S}^0 ( \dsp  , \dspDual) +  \partial_t b_{S}^0 ( \dspDual,  \LagVol  ) - \partial_t f_{S}^0 ( \dsp  ; \vcT  ; \prT  ;  \dspDual ) +  \partial_t b_{S}^0 ( \dsp  , \LagVolDual )  ,  
% \end{align}
$( \vcT, \prT , \dsp , \LagVol )$ the solution of the FSI problem \eqref{eq:complete-syst-stokes-stokes-transported}, and by $( \vcTDual, \prTDual , \dspDual , \LagVolDual )$ the adjoint states solution of the adjoint problem \eqref{eq:pbm:FSI:adj} written for $\SHPfun_E$, we have from Theorem~\ref{thm:ShapeDer:simplified} that
\begin{align}
\SHPfun_E' (\Omega_{0}) &=
\frac{1}{2} \int_{\Omega_{0}}  \lvert \gs \dsp \rvert^{2}\div(V)  
+ \frac{1}{2} \int_{\Omega_{0}^{c}}
\big\lvert \big[ \nabla \vcT(\nabla \eT)^{-1}  \big]^s  \big\rvert^{2}  \tr(\cof(\nabla \eT)^{\tp}\nabla V )  \notag\\
& -  \int_{\Omega_{0}}    \gs \dsp \Mscalp  \left(  \nabla \dsp \nabla V \right) 
-  \int_{\Omega_{0}^{c}}\big[ \nabla \vcT(\nabla \eT)^{-1}  \big]^s  \Mscalp    \big( \nabla \vcT(\nabla \eT)^{-1} \nabla V \cof(\nabla \eT)^{\tp}  \big)         \notag\\
&+ \mathcal{A} ' ( (\vcT , \prT , \dsp , \LagVol) , (\vcTDual , \prTDual , \dspDual , \LagVolDual) , V ) ,
%
% & + \partial_t a_{F}^0 ( \dsp  ;  \vcT  , \vcTDual ) + \partial_t b_{F}^0 ( \dsp  ; \vcTDual,  \prT  ) - \partial_t f_{F}^0 ( \dsp  ; \vcTDual    ) +
%  \partial_t b_{F}^0 ( \dsp ;  \vcT   , \prTDual )  \notag \\ 
% %
% %
% & +  \partial_t a_{S}^0 ( \dsp  , \dspDual) +  \partial_t b_{S}^0 ( \dspDual,  \LagVol  ) - \partial_t f_{S}^0 ( \dsp  ; \vcT  ; \prT  ;  \dspDual ) +  \partial_t b_{S}^0 ( \dsp  , \LagVolDual )  .
\end{align}
where $\mathcal{A} '$ is given by \eqref{eq:SHAPE-DER-SIMPLIF:2}.
\end{ex}

%#####################################################################
%#####################################################################
%\section{Conclusion}
% \label{sec:FSI:ccl}

% In this chapter, we have presented a Fluid Structure Interaction model, for which we define a shape optimisation problem.
% The purpose is to optimize an abstract shape functional $\SHPfun (\Omega_0)$, 
% defined by%defined in \eqref{eq:def:shp:fun}, 
% \begin{equation}
% \SHPfun (\Omega_0) 
% = \int_{\Omega_0} \shpFUN_S (Y, \dsp(Y), \gd \dsp(Y)) \,dY
% + \int_{\Omega_F} \shpFUN_F (x, \vct(x), \gd \vct(x)) \,dx ,
% \end{equation}
% depending on the initial elastic domain $\Omega_0$ of the FSI problem, and where $\vct$ is the velocity field of the fluid whereas $\dsp$ is the displacement field of the elastic structure. 
% We have shown the wellposedness of the considered FSI model in Section \ref{sec:FSI:resolution}, 
% and then we have shown the domain differentiability of this model in Section \ref{sec:hadam:method}.
% Finally, we have calculated the shape derivative of $\SHPfun (\Omega_0)$ by using the velocity method, and we have simplified the formula of the shape derivative we the use of an adjoint method. 

%\appendix
%#####################################################################
%#####################################################################
\section{Appendix}
\label{sec:FSI:appendix}

%*********************************************************************
%*********************************************************************
\subsection{Derivatives of $J$, $G$, and $F$ maps}
\label{sec:derivJGF}

Let $\alpha : U \subset \VV \mapsto \varphi_\alpha \in (H^3 ( \Omega ))^2$ be a differentiable map, where $\VV$ is a normed vector space endowed with the norm $\lV\cdot\rV_{\VV}$, and $U$ is an open subset of $\VV$, and $\Omega$ is an open subset of $\RR^2$. 
Thus $\alpha : U \subset \VV \mapsto \gd \varphi_\alpha \in (H^2 ( \Omega ))^{2 \times 2}$ is differentiable, and we  
denote by $D_\alpha ( \gd \varphi_\alpha )$ the differential of $\alpha \mapsto \gd \varphi_\alpha$ at $\alpha$. 
Namely $D_\alpha ( \gd \varphi_\alpha )$ is the continuous linear map from $\VV$ to $(H^2 ( \RR^2 ))^{2 \times 2}$ such that for all $d\alpha \in \VV$:
\begin{equation}
\gd \varphi_{\alpha + d\alpha} =  \gd \varphi_{\alpha } + D_\alpha ( \gd \varphi_\alpha )d\alpha + o ( \lV d\alpha \rV_{\VV} ).
\end{equation}
Assuming $\gd \varphi_\alpha$ being invertible, we define the following maps depending on $\varphi_\alpha$:
\begin{align}
\label{eq:def:barJ}
J (\varphi_\alpha) &:= \det ( \gd \varphi_\alpha) , \\
\label{eq:def:barG}
G (\varphi_\alpha) &:= \cof( \gd \varphi_\alpha) , \\
\label{eq:def:barF}
F (\varphi_\alpha) &:= (\gd \varphi_\alpha)^{-1} \cof( \gd \varphi_\alpha ) ,
\end{align}
where $\cof( \gd \varphi_\alpha ) $ is the cofactor matrix of $\gd \varphi_\alpha$ defined by
\begin{equation}
\cof( \gd \varphi_\alpha ) = \det( \gd \varphi_\alpha  ) \gd \varphi_\alpha ^{-T} .
\end{equation}
We recall that the determinant $\det ( \cdot )$, the inverse $( \cdot )^{-1}$, and the cofactor $\cof (\cdot )$ matrix are differentiable maps defined on the open set of invertible matrices, and their differentials are given by the following expressions. Let $A,B \in \RR^{2 \times 2}$, $A$ being invertible, and %$\lV B \rV_{\RR^{\dn \times \dn}}$ 
$\lv B \rv$ sufficiently small so that $A+B$ is invertible, where $\lv B \rv$ is given in \eqref{eq:def:matrixnorm}. 
We have
\begin{align}
\label{eq:Diff:det}
\det (A + B) &= \det (A) + \tr ( \cof (A)^\tp  B )  + o ( \lv B \rv ) , \\
\label{eq:Diff:inv}
( A + B )^{-1} &= A^{-1}  -  A^{-1} B A^{-1} + + o (  \lv B \rv  ) , \\
\label{eq:Diff:cof}
\cof (A + B) &= \cof (A) 
+ \left( \tr ( \cof (A)^\tp  B ) \Id - \cof( A ) B^\tp    \right) A^{-\tp}  
+ o ( \lv B \rv ) .
\end{align}
As it is shown in Section \ref{sec:FS_resol_prelim}, the maps $J$, $G$, and $F$ are well-defined and differentiable because of the Banach algebra structure of $H^2 (\Omega)$. 
From there, applying the chain rule and using expressions \eqref{eq:Diff:det}, \eqref{eq:Diff:inv}, and \eqref{eq:Diff:cof},  we can compute the differentials $D_\alpha  J ( \varphi_\alpha  )$, $D_\alpha  G ( \varphi_\alpha  )$, and $D_\alpha  F ( \varphi_\alpha  )$. 
% \begin{align}
% \label{eq:DJbar}
% D_\alpha J (\varphi_\alpha) d\alpha
% &= \tr ( \cof (\gd \varphi_\alpha)^\tp  D_\alpha (\gd \varphi_\alpha) (d\alpha) ) \\
% %
% \label{eq:DGbar}
% D_\alpha G (\varphi_\alpha) d\alpha
% &= \left[ \tr ( \gd \varphi_\alpha^{-1} D_\alpha (\gd \varphi_\alpha) (d\alpha) )\mathrm{I} \right. \notag  \\
% & \hspace{8ex} \left.  - \gd \varphi_\alpha^{-\tp} D_\alpha (\gd \varphi_\alpha) (d\alpha) ^\tp \right] \cof ( \gd \varphi_\alpha ) \\
% %
% \label{eq:DFbar}
% D_\alpha F (\varphi_\alpha) d\alpha
% &= \cof (\gd \varphi_\alpha)^\tp \left[ \tr ( \gd \varphi_\alpha^{-1} D_\alpha (\gd \varphi_\alpha) (d\alpha) )\mathrm{I} \right. \notag \\
% & \hspace{17ex} \left.
% - 2 ( D_\alpha (\gd \varphi_\alpha) (d\alpha) \gd \varphi_\alpha^{-1}  )^s \right]  \gd \varphi_\alpha^{-\tp} .
% \end{align}
We give their expressions in the case where $\alpha =t $, $\alpha=w$, and $\varphi_\alpha = \lT^t_w$.

% $D_\alpha ( J ( \gd \varphi_\alpha ) )$, $D_\alpha ( J ( \gd \varphi_\alpha ) )$, and $D_\alpha ( J ( \gd \varphi_\alpha ) )$.

% %~~~~~~~~~~~~~~~~~~~~~~~~~~~~~~~~~~~~~~~~~~~~~~~~~~~~~~~~~~~~~~~~~~~~~
% %~~~~~~~~~~~~~~~~~~~~~~~~~~~~~~~~~~~~~~~~~~~~~~~~~~~~~~~~~~~~~~~~~~~~~
% \subsubsection{Derivatives with respect to $t$}
% \label{sec:deriv:t}

We recall that $\Phi  _t$ is the map defined in \eqref{eq:def:Phit} in Section \ref{subsec:settings} by
\begin{equation}
\Phi  _{t} := \id_{\RR^{\dn}} + t V ,
\end{equation}
and that we have defined in \eqref{eq:def:T-t-up-w-down}
the following $H^{3}(\Omega_0^c)$-valued map for all $(t , w) \in \RR_+ \times (H^1_{0, \Gamma_{\omega}}(\Omega_{0}) \cap H^3 (\Omega_{0}) )^2 $ by:
\begin{equation}
\Tupt_w :=  \Phi_t + \mathcal{R} \gamma (w)    .
\end{equation}
This map is differentiable, and its differential with respect to $w$ is given by 
\begin{align}
\label{eq:Diff:T_tw}
 D_w ( \lT^t_{ w } )   k 
&= \mathcal{R} ( \gamma (k) ) ,
\end{align} 
for all $k  \in (H^1_{0, \Gamma_{\omega}}(\Omega_{0}) \cap H^3 (\Omega_{0}) )^2 $.
Thus, from the definitions \eqref{eq:def:barJ}-\eqref{eq:def:barG}-\eqref{eq:def:barF}, the expressions \eqref{eq:Diff:det}-\eqref{eq:Diff:inv}-\eqref{eq:Diff:cof} and \eqref{eq:Diff:T_tw}
and in view of the chain rule,  
we can deduce the values of the following differentials: 
\begin{align}
\label{eq:DJbar}
D_w J ( \lT^t_{w} ) k
&= \tr ( \cof (\gd  \lT^t_{w} )^\tp  \mathcal{R} ( \gamma (k) ) )  ,  \\
\label{eq:DGbar}
D_w G ( \lT^t_{w} ) k
&= \left[ \tr ( (\gd  \lT^t_{w}) ^{-1} \mathcal{R} ( \gamma (k) ) )\mathrm{I} %\right. \notag  \\
%& 
%\hspace{8ex} \left. 
- (\gd  \lT^t_{w}) ^{-\tp} \mathcal{R} ( \gamma (k) ) ^\tp \right] \cof ( \gd  \lT^t_{w}  )  ,  \\
\label{eq:DFbar}
D_w F ( \lT^t_{w} ) k
&= \cof (\gd  \lT^t_{w} )^\tp \left[ \tr ( (\gd  \lT^t_{w} )^{-1} \mathcal{R} ( \gamma (k) ) )\mathrm{I} 
% \right. \notag \\
% & \hspace{17ex} \left.
- 2 ( \mathcal{R} ( \gamma (k) ) (\gd  \lT^t_{w}) ^{-1}  )^s \right]  (\gd  \lT^t_{w}) ^{-\tp} .
\end{align}
% \begin{align}
% \label{eq:DJbar}
% D_\alpha J (\varphi_\alpha) d\alpha
% &= \tr ( \cof (\gd \varphi_\alpha)^\tp  D_\alpha (\gd \varphi_\alpha) (d\alpha) ) \\
% %
% \label{eq:DGbar}
% D_\alpha G (\varphi_\alpha) d\alpha
% &= \left[ \tr ( \gd \varphi_\alpha^{-1} D_\alpha (\gd \varphi_\alpha) (d\alpha) )\mathrm{I} \right. \notag  \\
% & \hspace{8ex} \left.  - \gd \varphi_\alpha^{-\tp} D_\alpha (\gd \varphi_\alpha) (d\alpha) ^\tp \right] \cof ( \gd \varphi_\alpha ) \\
% %
% \label{eq:DFbar}
% D_\alpha F (\varphi_\alpha) d\alpha
% &= \cof (\gd \varphi_\alpha)^\tp \left[ \tr ( \gd \varphi_\alpha^{-1} D_\alpha (\gd \varphi_\alpha) (d\alpha) )\mathrm{I} \right. \notag \\
% & \hspace{17ex} \left.
% - 2 ( D_\alpha (\gd \varphi_\alpha) (d\alpha) \gd \varphi_\alpha^{-1}  )^s \right]  \gd \varphi_\alpha^{-\tp} .
% \end{align}

Noting that the derivative of $\lT^t_w$ with respect to $t$ is given by
\begin{equation}
    \frac{\mathrm{d}}{\mathrm{d}t} \Tupt_w
 = V_t 
 := \frac{\mathrm{d}}{\mathrm{d}t}  \Phi_t   ,
\end{equation}
we can also deduce the time derivatives of  $  J(   \lT^t_{w}  )$, $ G(   \lT^t_{w}  )$, and $  F(   \lT^t_{w}  )$, given by
\begin{align}
\label{eq:DJ2Vt} 
% D_{t=0} J(   \lT^t_{w}  ) = 
D_t J ( \lT^t_{w} ) &= 
\tr ( \cof(\gd  \lT^t_{w} )^\tp   \gd V_t  )
 , \\
\label{eq:DG2Vt}
% D_{t=0} G(   \lT^t_{w}  ) = 
D_t G ( \lT^t_{w} ) &=
\cof(\gd  \lT^t_{w} )  
\left[
\tr \left( (\gd \lT^t_{w} )^{-1} \gd V_t \right) \mathrm{I} 
 - [ (\gd \lT^t_{w} )^{-1}  \gd V_t   ]^\tp
\right]  , \\
\label{eq:DF2Vt}
% D_{t=0} F(   \lT^t_{w}  ) = 
D_t F ( \lT^t_{w} ) &= 
\cof(\gd \lT^t_{w} )^{\tp}
\left[   \tr \left( (\gd \lT^t_{w} )^{-1} \gd V_t \right) \mathrm{I}   
 - 2[ \gd V_t (\gd \lT^t_{w} )^{-1}]^s  \right]  (\gd \lT^t_{w} )^{-\tp} .
\end{align}
By setting $\lT^t_{w} = \eT_0$ and $V_t = V$ in these expressions, we retrieve the fields $DJ (V) $, $DG (V) $, and $DF (V) $ involved in Theorem~\ref{thm:ShapeDer:simplified}.

\printbibliography  

\end{document}